\documentclass[11pt]{amsart}

\usepackage{amssymb,amsmath,amsthm,amscd,enumerate}
\usepackage{verbatim}
\usepackage[francais]{babel}
\usepackage[all]{xy}
\usepackage[T1]{fontenc}

\usepackage[colorlinks,linkcolor=blue,citecolor=blue,urlcolor=red]{hyperref} 

\newcommand{\n}{\noindent}
\newcommand{\rrr}{\longrightarrow}
\newcommand{\s}{{\rm Spec}}

\newcommand{\oo}{{\mathcal O}}

\newcommand{\bb}{\bigskip}
\newcommand{\e}{{\sf E}}

\newcommand{\scd}{sch\'ematiquement dominant}
\newcommand{\isl}{isomorphisme local\,}
\newcommand{\rt}{\, \widetilde{\longrightarrow}\, }
\newcommand{\wt}{\widetilde}
\newcommand{\sm}{{\sf Sm}}
\newcommand{\et}{{\sf Et.sep}}

\newcommand{\ul}{\underleftarrow{\lim}\,}

\newcounter{spec}
\newenvironment{thlist}{\begin{list}{\rm{(\roman{spec})}}%
{\usecounter{spec}\labelwidth=20pt\itemindent=0pt\labelsep=10pt}}%
{\end{list}}%

\numberwithin{equation}{section}

\newtheorem{Th}{Théorème}
\swapnumbers
\newtheorem{lemme}{Lemme}[subsection]
\newtheorem{thm}[lemme]{Théorème}
\newtheorem{prop}[lemme]{Proposition}
\newtheorem{cor}[lemme]{Corollaire}
\newtheorem{conj}[lemme]{Conjecture}
\newtheorem{scholie}[lemme]{Scholie}
\theoremstyle{definition}
\newtheorem{defn}[lemme]{Définition}

\theoremstyle{remark}
\newtheorem{rque}[lemme]{Remarque}
\newtheorem{rques}[lemme]{Remarques}
\newtheorem{ex}[lemme]{Exemple}

\newtheorem{para}[lemme]{}
\newtheorem{vide}[lemme]{}
\begin{document}
\title{Recoller pour s\'eparer}
\author{Daniel Ferrand}
\address{IMJ-PRG, Case 247, 4 place Jussieu, 75252 Paris Cedex 05, France}
\email{daniel.ferrand@imj-prg.fr}
\author{Bruno Kahn}
\address{IMJ-PRG, Case 247, 4 place Jussieu, 75252 Paris Cedex 05, France}
\email{bruno.kahn@imj-prg.fr}
\date{22 octobre 2015}
\begin{abstract} On introduit une notion de \emph{s\'eparateur} d'un morphisme de sch\'emas $f : T\rrr S$: en particulier, un s\'eparateur est universel parmi les morphismes de $T$ vers un $S$-sch\'ema s\'epar\'e $E$. Lorsque $f$ est quasi-s\'epar\'e, son s\'eparateur existe si et seulement si l'adh\'erence sch\'ematique de la diagonale se projette sur les deux facteurs par des morphismes plats de type fini. En particulier, $f $ admet un s\'eparateur si  $T$ est  noeth\'erien de Dedekind et $S = \s({\bf Z})$, ou si $f$ est \'etale de pr\'esentation finie et que $S$ est normal. Dans tout schéma normal de type fini sur un anneau noethérien, il existe un ouvert contenant les points de codimension $1$ et admettant un séparateur. A contrario, on donne plusieurs exemples de morphismes $f$ n'admettant pas de s\'eparateur. 

Comme application, on attache \`a tout sch\'ema lisse $T$ sur une base normale $S$ un morphisme vers un $S$-sch\'ema \'etale de pr\'esentation finie et \emph{s\'epar\'e}, qui est universel (variante séparée du ``sch\'ema des composantes connexes des fibres''), g\'en\'eralisant simultan\'ement le cas classique o\`u la base est un corps et le cas d'un morphisme propre et lisse (factorisation de Stein).\\

\noindent {\sc Abstract.} 
We introduce the notion of a \emph{separator} for a morphism of schemes $f : T\rrr S$; in particular, it is universal among morphisms from $T$ to separated $S$-schemes $E$. A separator is a local isomorphism; this property conveys the intuition of \emph{gluing some affine covering more}, in order to make the scheme separated. When $f$ is quasi-separated, its separator exists if and only if the schematic closure of the diagonal projects on both factors by flat morphisms of finite type. In particular, $f$ admits a separator if $T$ is Noetherian Dedekind and $S=\s({\bf Z})$, or if $f$ is étale of finite presentation and $S$ is normal. Any normal scheme of finite type over a Noetherian ring admits an open subset containing all the points of codimension 1, which has  a separator. A contrario, we give several examples of morphisms $f$ that do not admit a separator. 

As an application, we attach to every smooth scheme $T$ over a normal base $S$ a morphism to a  \emph{separated} étale $S$-scheme of finite presentation, which is universal (a kind of separated alternative for ``scheme of connected components of the fibres''). This simultaneously generalizes the classical case where the base is a field, and the case of a smooth and proper morphism (Stein factorisation).
\end{abstract}
\maketitle
\newpage

\tableofcontents






\section{Introduction}

\enlargethispage*{20pt}

\subsection{}  La question de l'existence d'une \emph{enveloppe s\'epar\'ee} d'un sch\'ema ne sera pas 
abord\'ee dans toute sa g\'en\'eralit\'e, mais plut\^ot  sous une forme restreinte : \'etant donn\'e un sch\'ema $T$, disons int\`egre et noeth\'erien, existe-t-il un sch\'ema \emph{s\'epar\'e} $E$ et un isomorphisme local surjectif  $h:T\rrr~E$ ? La propriété d'\^{e}tre un isomorphisme local exprime exactement l'intuition que le passage de $T$ à $E$ consiste à \emph{recoller davantage} les ouverts d'un recouvrement. En effet, pour un tel morphisme, chaque point de $T$ est contenu dans un ouvert $U$ sur lequel $h$ induit une immersion ouverte vers $E$ ; si $U$ et $V$ sont des ouverts avec cette propriété, l'image $h(U\cup V) = h(U) \cup h(V)$ est la réunion d'ouverts  isomorphes respectivement à $U$ et à $V$, mais recollés le long de l'ouvert $h(U) \cap h(V)$, lequel contient, éventuellement strictement, l'ouvert $h(U\cap V)$, isomorphe à $U \cap V$.

Dans le cas g\'en\'eral o\`u $T$ n'est plus suppos\'e noeth\'erien, ni surtout int\`egre, il est apparu qu'une hypoth\`ese sur les images dans $E$ des points maximaux de $T$ est indispensable, et que sa  bonne formulation utilise le morphisme diagonal, $\Delta : T \rrr T\times_{E}T$ : il doit \^etre \scd\, i.e. l'homomorphisme $\oo_{T\times_{E}T} \rrr \Delta_{\star}(\oo_{T})$ doit être injectif.\medskip

Dans la situation relative (à un schéma de base $S$) on est conduit à la définition suivante : partant d'un morphisme $ f : T \rrr S$, nous dirons qu'un morphisme de $S$-sch\'emas $h : T \rrr E$  est un {\it s\'eparateur} pour $f$ si $E$ est un $S$-sch\'ema s\'epar\'e et si $h$ est un isomorphisme local surjectif quasi-compact et quasi-s\'epar\'e, dont le morphisme diagonal est \scd\footnote{Rappelons qu'un morphisme de sch\'emas  $f : T\rrr S $ est dit \emph{quasi-s\'epar\'e} si son morphisme diagonal $\Delta_{f} $ est quasi-compact ; cette hypoth\`ese implique que l'image directe $(\Delta_{f})_{\star}(\oo_{T})$ est une $\oo_{T\times_{S} T}$-alg\`ere quasi-coh\'erente, et que donc l'adh\'erence sch\'ematique de $\Delta_{f} $ existe ; de plus cette adh\'erence sch\'ematique a pour espace sous-jacent l'adh\'erence (topologique) de $\Delta_{f}(T)$ dans $T\times_{S} T$ \cite[6.10.5]{EGAG}.
Enfin, un morphisme est \emph{de type fini} s'il est quasi-compact et localement de type fini.}. Si le sch\'ema $S$ lui-m\^eme est s\'epar\'e, les notions de s\'eparateur relatif \`a $S$ ou relatif \`a $\s({\bf Z})$ co\"{i}ncident. 
\medskip

Si $h : T \rrr E$ est un s\'eparateur pour $f$, on montre que la restriction de $h$ \`a \emph{tout} ouvert  $U$ de $T$, s\'epar\'e sur $S$, induit un isomorphisme de $U$  sur son image $h(U)$; de plus, ce morphisme $h$ est universel au sens o\`u tout $S$-morphisme de $T$ dans un sch\'ema s\'epar\'e sur $S$ se factorise de fa\c con unique par $h$ ; autrement dit, un séparateur, quand il existe, est aussi une \emph{ enveloppe séparée}.
 \medskip

\subsection{} \label{s1.2} L'existence d'un s\'eparateur pour $T$ \'equivaut \`a l'existence d'un schéma quotient de $T$ par une relation  d'\'equivalence convenable $R$  (qui n'est autre que $\xymatrix{ T \times_{E}T  \ar@<0.5ex>[r] \ar@<-0.5ex>[r]& T}$, lorsque le s\'eparateur existe). Le quotient $T/R$ est s\'epar\'e sur $S$ si le morphisme $R \rrr T\times_{S} T$ est une immersion ferm\'ee. Il est donc bien naturel d'envisager a priori, comme relation, \emph{l'adh\'erence sch\'ematique  de la diagonale} $T \rrr T\times_{S} T$. Mais cette adh\'erence sch\'ematique n'est pas toujours une relation d'\'equivalence.
\medskip

 Cette difficult\'e appara\^it d\'ej\`a dans la cat\'egorie des espaces topologiques. Soient $T$ un espace topologique et $R$ l'adh\'erence de la diagonale dans $T\times T$. La relation $(x, y) \in R$ signifie que chaque voisinage de $x$ rencontre chaque voisinage de $y$  ; cette relation n'est en g\'en\'eral pas transitive. Mais on v\'erifie imm\'ediatement que cette adh\'erence  $R$  est  une relation d'\'equivalence d\`es que les applications compos\'ees $
\xymatrix{R \ar@<0.5ex>[r] \ar@<-0.5ex>[r]& T}$ sont ouvertes ; l'espace topologique quotient $T/R$ est alors s\'epar\'e et l'application $T \rrr T/R$ est ouverte (voir \cite[I.55]{TG} pour ces derniers points).
\medskip

La d\'emarche, propos\'ee ici pour les sch\'emas, pr\'esente des analogies avec la remarque pr\'ec\'edente, les morphismes plats de sch\'emas rempla\c cant les applications ouvertes d'espaces topologiques.
\bb

En effet, soient $T$ un $S$-sch\'ema et $R \subset T\times_{S} T$ l'adh\'erence sch\'ematique de la diagonale.
On montre que  $R$ d\'efinit une relation d'\'equivalence sur $T$ d\`es que les morphismes compos\'es $
\xymatrix{R \ar@<0.5ex>[r] \ar@<-0.5ex>[r]& T}$ sont plats ; si, de plus, ils sont de type fini ces morphismes sont alors des isomorphismes locaux, du moins sous une hypothèses de finitude assez anodine sur $T$.

On d\'egage ensuite un \'enonc\'e g\'en\'eral : soit  \; $d_{0}, d_{1} :
\xymatrix{R \ar@<0.5ex>[r] \ar@<-0.5ex>[r]& T}
$ une relation d'\'equivalence sur un sch\'ema $T$ telle que les morphismes $d_{0}$ et $d_{1}$ soient des isomorphismes locaux et que le morphisme ``de r\'eflexivit\'e'' $\varepsilon : T \rrr R$ soit \scd. Alors, le faisceau fppf quotient $\wt{T/R}$ est repr\'esentable \emph{par un sch\'ema.}
\medskip

\subsection{} Cela ram\`ene donc le probl\`eme de l'existence d'un s\'eparateur  \`a une question de platitude. Voici les cas principaux où nous savons démontrer cette existence:
 
\begin{Th}\label{t1} Soit $f:T\to S$ un morphisme de schémas. 
\begin{enumerate}
\item  (prop. \ref{p5.3}) Si $S=\s({\bf Z})$ et $T$ est noethérien régulier de dimension $1$, $f$ admet un séparateur.
\item (cor. \ref{c5.1}) De même si $S$ est normal connexe et $f$ étale de présentation finie.
\item (th. \ref{t6.1}) Si $S=\s(A)$ où $A$ est noéthérien, si $f$ est de type fini et si $T$ est normal, il existe un ouvert $U\subset  T$, contenant les points de codimension $1$, tel que $f_{|U}$ admette un séparateur.
\end{enumerate}
\end{Th}


\subsection{} Le langage des anneaux locaux \emph{apparentés}, certes un peu désuet aujourd'hui, éclaire certains aspects de nos constructions. En effet, le passage d'un schéma intègre $T$ à son séparateur, quand ce dernier existe, consiste simplement à identifier les points de $T$ dont les anneaux locaux sont apparentés.

Mais nous définissons un schéma intègre noethérien de dimension 1, n'ayant que deux points fermés et qui n'admet pas de séparateur ; il fournit en un sens  l'archétype de la non-séparation. Ce schéma admet cependant une enveloppe séparée, ce qui montre que les deux notions sont différentes.
\medskip

Pour mettre en lumière ce qu'implique l'existence d'un séparateur, et aussi pour expliquer pourquoi elle n'est pas plus fréquente, nous construisons des morphismes de schémas $f : T \rightarrow S$, o\`u $T$ est réunion de deux ouverts affines, et  qui n'admettent pas de séparateurs, bien que $f$ soit, dans un cas, lisse de dimension relative 1 et $S$ régulier, et dans l'autre cas étale sur une base affine de dimension 1. Notons que ces exemples portent sur des schémas noethériens, et qu'il en est de m\^{e}me de tous les contre-exemples du texte; ce ne sont donc pas des défauts de finitude qui provoquent les défauts de séparation.

\subsection{} On cherche ensuite \`a d\'efinir un foncteur de la cat\'egorie des sch\'emas (en tout cas plats de pr\'esentation finie) sur une base $S$ vers la cat\'egorie des $S$-sch\'emas \'etales de pr\'esentation finie, foncteur qui soit adjoint \`a gauche de l'inclusion de ces cat\'egories ; autrement dit, il s'agit d'associer \`a tout tel sch\'ema $T/S$ un sch\'ema $E$ \'etale et de pr\'esentation finie sur $S$, muni d'un morphisme de $S$-sch\'emas $h : T \rrr E$ qui soit universel pour les morphismes de $T$ vers un $S$-sch\'ema \'etale de pr\'esentation finie.
 
 On montre que cette propri\'et\'e universelle du morphisme $h : T \rrr E$ est \'equivalente au fait que les fibres de $h$ soient g\'eom\'etriquement connexes. C'est d'ailleurs ce qu'on constate dans deux situations classiques :

\begin{itemize}
\item Si $f : T \rrr S$ est un morphisme propre et lisse, avec $S$ noeth\'erien, la factorisation de Stein 
$$
h : T \rrr {\rm Spec}(f_{\star}(\oo_{T}))
$$
 d\'efinit l'adjoint cherch\'e ; et  dans ce cas, les fibres de  $h$ sont, bien s\^ur,  g\'eom\'etriquement connexes. 

\item Si $S$ est le spectre d'un corps, cet adjoint est bien connu :  il est not\'e $\pi_{0}(T/S)$, et il ``repr\'esente'' les composantes connexes de $f$.
\end{itemize}
\medskip

 Un r\'esultat développé par  {\sc Romagny} \cite{Rom11} (mais déjà signalé dans le livre  {\sc Laumon--Moret-Bailly} \cite{LMB00}) \'eclaire la suite. Il montre que le ``foncteur des composantes connexes des fibres'' d'un morphisme lisse $T \rrr S$ est repr\'esentable par un \emph{espace alg\'ebrique}  qu'il note $\pi_{0}(T/S)$ ; il est muni d'un morphisme $h : T \rrr \pi_{0}(T/S)$ dont les fibres sont, comme pr\'evu, g\'eom\'etriquement connexes ; de plus $\pi_{0}(T/S)$ est \'etale sur $S$. Il ne semble pas qu'on dispose de conditions g\'en\'erales assurant que cet espace alg\'ebrique soit un sch\'ema, hormis s'il est {\it s\'epar\'e}.
\medskip

\n Cela justifie qu'on se restreigne dans la suite \`a la sous-cat\'egorie des sch\'emas  \'etales de pr\'esentation finie, et  qui sont de plus {\it s\'epar\'es}.

\subsection{} On cherche donc \`a associer fonctoriellement \`a tout $S$-sch\'ema plat de pr\'esentation finie $T$, un morphisme $h : T \rrr E$ de $S$-sch\'emas, o\`u  $E$ est \'etale et s\'epar\'e sur $S$, qui soit universel. On constate d'abord que la cat\'egorie des $S$-morphismes surjectifs de $T$ vers un $S$-sch\'ema \'etale s\'epar\'e est tr\`es simple : elle est isomorphe \`a un ensemble ordonn\'e filtrant \`a gauche, et elle poss\`ede un \'el\'ement initial si $S$ est int\`egre. Ainsi l'adjoint cherch\'e existe d\`es que $S$ est int\`egre.

\n Mais, sans hypoth\`eses suppl\'ementaires, ce foncteur, que l'on note $\pi^s(T/S)$, n'a pas les propri\'et\'es attendues ; en particulier il ne commute  en g\'en\'eral  pas aux changements de base $S' \rightarrow S$, m\^eme s'ils sont \'etales.
\bb

Or, une variante des r\'esultats de s\'eparation cit\'es en \ref{s1.2} s'applique aux espaces alg\'ebriques et montre que, sur un sch\'ema \emph{normal} $S$, un $S$-espace alg\'ebrique \'etale et quasi-compact $F$ admet un morphisme \'etale et birationnel $F \rrr F^{\rm sep}$ sur un espace alg\'ebrique \'etale \emph{et s\'epar\'e}, qui est alors (repr\'esentable par) un sch\'ema, d'apr\`es  un r\'esultat classique. En appliquant cette construction 
\`a l'espace alg\'ebrique 
de {\sc Romagny}, 
on obtient l'énoncé suivant:

\begin{Th}[\protect{Th. \ref{p9.4} et prop. \ref{p10.1}}] Soient $S$ un schéma normal, $\sm(S)$ la catégorie des $S$-schémas lisses de présentation finie et $\et(S)$ la sous-catégorie pleine de $\sm(S)$ formée des $S$-schémas étales \emph{séparés}. Alors le foncteur d'inclusion
\[\et(S)\to \sm(S)\]
admet un adjoint à gauche, qui commute aux changements de base lisses et aux produits  (finis).
\end{Th}

Voici un survol du texte.

\begin{itemize}
\item Le \S \ref{s2} introduit les s\'eparateurs et en donne les premi\`eres propri\'et\'es ; une attention particuli\`ere est port\'ee \`a la condition pour le morphisme diagonal d'\^etre \scd. Un critère technique pour qu'un morphisme plat soit un isomorphisme local est donné ici ; il resservira plus loin.

\item Le \S \ref{s3} montre, entre autres choses, qu'un séparateur, quand il existe, est aussi une enveloppe séparée. Ce paragraphe contient aussi un premier exemple de schéma [quasi-séparé] qui admet une enveloppe séparée sans admettre  de s\'eparateur. 

\item Dans le \S \ref{s4}, on montre que le faisceau quotient fppf $\wt{T/R}$, d'un sch\'ema $T$ par une relation d'\'equivalence $R$, est repr\'esentable par un sch\'ema si, essentiellement,  les deux projections $\xymatrix{R \ar@<0.5ex>[r] \ar@<-0.5ex>[r]& T}$ sont des isomorphismes locaux.

\item Le \S \ref{s5}, qui est le coeur du texte, contient les crit\`eres d'existence de s\'eparateurs ; ils reposent sur le résultat de \S \ref{s4}.

\item Dans le \S \ref{s6}, il est démontré qu'un schéma normal de type fini sur un anneau noethérien possède un ouvert contenant les points de codimension 1 et qui admet un séparateur.

\item Dans le \S \ref{s07}, adopter  l'ancien point de vue des anneaux  locaux apparentés s'avère cependant instructif, et cela conduit à l'archétype de ce qui s'oppose à la séparation.

\item Le \S \ref{s7} propose deux autres exemples de morphismes qui n'admettent pas de séparateur.

 \item Le \S \ref{s8} contient une extension succincte aux espaces algébriques des résultats du \S \ref{s5}.
 
\item Ces r\'esultats de s\'eparation sont utilis\'es au \S \ref{s9}, o\`u on consid\`ere l'adjoint \`a gauche de l'inclusion de la cat\'egorie des $S$-sch\'emas \'etales, dans celle des $S$-sch\'emas lisses ; on montre pourquoi il faut se restreindre aux sch\'emas \'etales qui sont de plus \'epar\'es ; cet adjoint (s\'epar\'e) n'existe, en g\'en\'eral, que si $S$ est int\`egre, et il ne poss\`ede les propri\'et\'es attendues que si la base $S$ est normale. 

\item L'Appendice \ref{sA} rassemble des d\'efinitions et des r\'esultats portant sur les isomorphismes locaux et les adh\'erences sch\'ematiques, avec une insistance particulière sur les isomorphismes locaux. Ces résultats sont utilisés partout dans le texte.

\item  L'Appendice \ref{sB} montre   le r\^ole de la platitude pour forcer l'adh\'erence sch\'ematique de la diagonale \`a \^etre une relation d'\'equivalence.
\end{itemize}
\vspace{1cm}

{\it Il faut souligner, \`a l'or\'ee d'un travail sur la s\'eparation des sch\'emas, que la terminologie actuelle \,\emph{ sch\'ema / sch\'ema s\'epar\'e},\, introduite dans la version de \emph{EGA  I} parue en 1971, remplace maintenant la terminologie \, \emph{pr\'esch\'ema/ sch\'ema}\,  qui \'etait utilis\'ee jusqu'\`a la fin des ann\'ees 60 ; les 7 volumes des \emph{EGA II }\`a \emph{IV}, traitent de pr\'esch\'emas, ainsi que les r\'e\'editions de \emph{SGA 1} et de \emph{SGA 2} ; mais la r\'e\'edition de \emph{SGA 3} adopte la nouvelle terminologie.}

\newpage

\section{S\'eparateurs}\label{s2}

\subsection{Définition et propriétés générales}

\begin{defn}\phantomsection\label{d2.1} Soit $f : T \rrr S$ un morphisme de sch\'emas. On appelle \emph{s\'eparateur} pour $f$ un morphisme de $S$-sch\'emas $h : T \rrr E$, de but  un sch\'ema s\'epar\'e sur $S$, et qui v\'erifie les propri\'et\'es  suivantes :

\n $i)$ $h$ est un isomorphisme local surjectif, quasi-compact et quasi-s\'epar\'e.

\n $ii)$ Le morphisme diagonal $\Delta_{h}$ est \scd\, (définition \ref{dA.1}).
\end{defn}

Lorsque $S = \s({\bf Z})$ on parle simplement de s\'eparateur pour $T$.

\begin{rques}\phantomsection\label{r2.1} \

$i)$\, Un morphisme $f : T \rrr S$ qui admet un s\'eparateur  est quasi-s\'epar\'e,  ou de manière équivalente, le morphisme diagonal $\Delta_f$ est quasi-compact. En effet, en notant $h : T \rightarrow E$ un séparateur,  $\Delta_{f}$ se factorise en 
$$
T \; \stackrel{\Delta_{h}}{\rrr} \; T \times_{E}T \; \stackrel{u}{\rrr}\; T\times_{S} T .
$$
Puisque $h$ est quasi-séparé, $\Delta_{h}$ est quasi-compact, de plus $u$ est une immersion ferm\'ee puisque qu'elle provient, par changement de base, du morphisme diagonal du $S$-sch\'ema s\'epar\'e  $E$ ;  donc le composé est quasi-compact. 
\medskip

\n $ii)$ Si $T$ est intègre, la propriété $i)$ de \ref{d2.1}implique $ii)$, comme il résultera de l'équivalence entre $a)$ et $c)$ dans la proposition \ref{p2.1} ci-dessous.
\medskip 

\n $iii)$ Pour un isomorphisme local $h : T \rrr S$, la propri\'et\'e d'avoir un morphisme diagonal \scd\, signifie intuitivement que $h$ induit une injection sur l'ensemble des points maximaux de $T$; c'est \`a rapprocher de la notion de morphisme \emph{birationnel} tel que d\'efini dans \cite[2.3.4]{EGAG} ; en général, un morphisme $h : T \rrr S$ y est dit {\it birationnel} s'il induit une bijection de l'ensemble des points maximaux de $T$ sur l'ensemble de ceux de $S$, si, pour tout point maximal $s$ de $S$, l'ensemble $h^{-1}(s)$ est r\'eduit \`a un point $t$ et enfin si le morphisme $\oo_{S, s} \rrr \oo_{T, t}$ est bijectif. 

Lorsque $h : T \rrr S$ est un isomorphisme local, la remarque 2.3.4.1 de loc. cit. entra\^ine que $h$ est birationnel  d\`es qu'il induit une {\it bijection} entre les ensembles des points maximaux (cf. la propri\'et\'e $c)$ de la proposition \ref{p2.1} ci-dessous). 
\end{rques}

\begin{lemme}\phantomsection Soit $h : T \rrr E$ un morphisme fid\`element plat quasi-compact et quasi-s\'epar\'e. Si le morphisme diagonal $\Delta_{h}$ est \scd, alors le morphisme canonique $\theta : \oo_{E} \rrr~h_{\star}(\oo_{T})$ est un isomorphisme.\end{lemme}

\begin{proof} Consid\'erons le diagramme
\[
\xymatrix{ T \ar[r]^{\Delta_{h}} & T\times_{E}T \ar[r]^{p_{1}} \ar[d]_{p_{0}} & T\ar[d]^h\\
& T \ar[r]_{h} & E.}
\]

Puisque $h$ est fid\`element plat quasi-compact, il suffit de v\'erifier que l'application $h^{\star}(\theta)$ obtenue par changement de base  est bijective. Comme le morphisme $h$ est aussi suppos\'e quasi-s\'epar\'e, on dispose d'apr\`es \cite[9.3.3]{EGAG} d'un isomorphisme
$$
h^{\star}h_{\star}(\oo_{T}) \; \simeq \; {p_{1}}_{\star}{p_{0}}^{\star}(\oo_{T}) .
$$
On est donc ramen\'e \`a voir que l'application compos\'ee, qui sera not\'ee $\theta'$,
$$
\oo_{T} \simeq h^{\star}(\oo_{E}) \; \stackrel{h^{\star}(\theta)}{\rrr} \; h^{\star}h_{\star}(\oo_{T}) \; \simeq \; {p_{1}}_{\star}(\oo_{T\times_{E}T})
$$
est bijective. Cette applicationn $\theta'$ est isomorphe \`a l'application canonique associ\'ee au morphisme $p_{1} : T\times_{E}T \rrr T$, comme on s'en convainc par r\'eduction au cas o\`u $T$ et $E$ sont affines. Notons $\varphi : \oo_{T\times_{E}T} \rrr {\Delta_{h}}_{\star}(\oo_{T})$ l'application associ\'ee  au morphisme diagonal. Puisque le compos\'e $p_{1}\Delta_{h}$ est l'application identique, l'application compos\'ee
$$
\oo_{T}\, \stackrel{\theta'}{\rrr} \, {p_{1}}_{\star}(\oo_{T\times_{E}T}) \, \stackrel{{p_{1}}_{\star}(\varphi)}{\rrr} \, {p_{1}}_{\star}{\Delta_{h}}_{\star}(\oo_{T})
$$
est bijective, et par suite, ${p_{1}}_{\star}(\varphi)$ est surjective. Mais, $\varphi$ \'etant injective par hypoth\`ese, il en est de m\`eme de ${p_{1}}_{\star}(\varphi)$, donc $\theta'$ est bien bijective.
\end{proof}

\begin{ex}\phantomsection Soit $T$ un sch\'ema int\`egre noeth\'erien r\'egulier de dimension 1 dont l'ensemble des points est fini. Alors le morphisme de $T$ dans son enveloppe affine est un s\'eparateur.
\end{ex}

\n Posons en effet $B = \Gamma(T, \oo_{T})$ ; le morphisme canonique $T \rrr \s(B)$ est \emph{l'enveloppe affine} de $T$ \cite[9.1.21]{EGAG}. Soit $K$ le corps des fractions de $B$, i.e. le corps des fonctions rationnelles sur $T$. Pour tout point ferm\'e $t$ de $T$, l'anneau local en $t$ est muni d'une injection
$$
j_{t} : \oo_{T, t} \; \rrr \; K
$$
dont l'image est un anneau de valuation discr\`ete de $K$ ; on a \cite[8.5.1.1]{EGAG}
$$
B = \Gamma(T, \oo_{T})\; = \; \bigcap_{t\in T} j_{t}(\oo_{T, t}) .
$$
Les anneaux locaux de points ferm\'es distincts peuvent avoir la m\^eme image dans $K$ (voir \S 6) ; soient $A_{1}, \cdots A_{n}$ les anneaux de valuation \emph{distincts} qui proviennent de points de $T$, de sorte que $B = A_{1}\cap \cdots \cap A_{n}$. Le th\'eor\`eme d'approximation \cite[6, \S 7.1, prop. 1]{AC} indique que, en notant $\mathfrak{p}_{i} = B \cap \mathfrak{m}(A_{i})$, l'inclusion $B \subset A_{i}$ induit un isomorphisme
$$
B_{\mathfrak{p}_{i}}\quad \wt{\rrr} \quad A_{i} .
$$
Cela montre que le morphisme $T \rrr \s(B)$ est un isomorphisme local ; il est surjectif d'apr\`es la proposition 2 de loc.cit, enfin son morphisme diagonal est \scd\  puisque $T$ est int\`egre (remarque \ref{r2.1} {\it ii)}). \qed
\medskip

Nous montrerons que tout sch\'ema localement noeth\'erien r\'egulier de dimension 1 admet un s\'eparateur (proposition \ref{p5.3}), et nous donnerons  un exemple de sch\'ema noeth\'erien int\`egre de dimension 1, ayant deux points ferm\'es, et qui n'admet pas de s\'eparateur (exemple \ref{ex6.1}).

\subsection{Le morphisme diagonal}

La proposition qui suit indique quelques cons\'equences de l'hypoth\`ese {\it ii)} de la définition \ref{d2.1}, portant sur le morphisme diagonal ; nous avons d\'evelopp\'e certains d\'etails de leur d\'emonstration pour mettre en \'evidence des arguments que l'on retrouvera souvent dans la suite de ce texte.

\begin{prop}\phantomsection \label{p2.1}Consid\'erons les propri\'et\'es suivantes relatives \`a un isomorphisme local quasi-s\'epar\'e  $h:T \rrr S$.

\begin{itemize}
\item[ a)] Le morphisme diagonal $\Delta_{h} : T \rrr T\times_{S}T$ est \scd.

\item[ b)] Pour tout ouvert $U$ de $T$ tel que $h$ induise une immersion ouverte de $U$ dans $S$, le morphisme 
 $U \rrr h^{-1}(h(U))$ est \scd. 
  
 \item[b')]  Il existe un recouvrement de $T$ par des ouverts ayant les deux propri\'et\'es \'evoqu\'ees en $b)$.
 
\item[ c)] La restriction du morphisme $h$ \`a l'ensemble des points maximaux de $T$ est injective.

\end{itemize}

\n Alors  la propri\'et\'e $a)$ est \'equivalente aux propri\'et\'es  $b)$  et  $b')$, et elles impliquent  $c)$.

\n Si $T$ est r\'eduit,  les propri\'et\'es  $a), b), b')$ et $c)$ sont \'equivalentes.

\n Si $T$ est localement noeth\'erien, de sorte que l'ensemble {\rm Ass}$(T)$ des points associés à $T$ est défini, les propri\'et\'es $a), b)$et  $b')$  sont aussi \'equivalentes à 

d) {\it La restriction de $h$ \`a l'ensemble {\rm Ass}$(T)$ est injective}.
\end{prop}

\begin{proof} Montrons l'\'equivalence des propri\'et\'es $a), b)$ et $b')$. Soit $i : U \rrr T$ une immersion ouverte. L'immersion  ouverte $j : U \rrr h^{-1}h(U)$  est \'egale au compos\'e des morphismes
 $$
 U\; \stackrel{(1,i)}{\rrr} \; U\times_{S}T \; \stackrel{h\times 1}{\rrr} \; h(U)\times_{S}T \; \simeq \; h^{-1}h(U)\leqno{(\star)}
 $$
 Le morphisme $(1, i)$ est la restriction du morphisme diagonal $\Delta_{h} : T \rrr T\times_{S}T$ \`a l'ouvert $U\times_{S}T$. Si  $h$ induit un isomorphisme de $U$ sur l'ouvert $h(U)$, $h\times 1$ est un isomorphisme et $j$ est \scd\, si et seulement si le morphisme $(1, i)$ l'est. Cela montre l'implication a) $\Rightarrow$ b).

\n Comme $h$ est un isomorphisme local, $T$ poss\`ede un recouvrement par des ouverts sur lesquels $h$ induit une immersion ouverte, donc b) $\Rightarrow$ b'). 

\n Montrons  que b') implique a). Consid\'erons un recouvrement de $T$ par des ouverts $U$ isomorphes \`a leur image, de sorte que $T\times_{S}T$ est recouvert par les ouverts correspondants $U\times_{S}T$. Comme on l'a vu ci-dessus, les morphismes not\'es $(1, i)$ dans $(\star)$ sont \scd s, et donc $\Delta_h$ aussi.
\medskip
 
 Avant d'aborder la propri\'et\'e c) qui concerne les points maximaux, remarquons que si $h : T \rrr S$ est un isomorphisme local, un point $t \in T$ est maximal si et seulement si $h(t)$ est maximal pour $S$ ; en effet, le morphisme $\oo_{S,h(t)} \rrr \oo_{T, t}$ est un isomorphisme, et dire que $t$ est maximal dans $T$ signifie que l'anneau local $\oo_{T, t}$ a un seul id\'eal premier.
 
  Pour v\'erifier les autres implications de la proposition \ref{p2.1}, nous utiliserons la propri\'et\'e suivante :
 
 \n \cite[5.4.3]{EGAG} {\it Soit $j : U \rrr V$ une immersion ouverte de sch\'emas. Si $j$ est \scd e, alors les points maximaux de $V$ sont dans $U$ ; r\'eciproquement, si $V$ est r\'eduit et si les points maximaux de $V$ sont dans $U$ alors $j$ est \scd e}.
 
 \medskip

 Montrons que $a)$ entra\^ine $c)$. Soit $\eta$ un point maximal de $h(T) \subset S$ et notons  $S_{0} = \s(\oo_{S, \eta})$ ; il s'agit de v\'erifier que la fibre $T_{0} = T\times_{S}S_{0}$ est irr\'eductible, i.e que l'intersection de deux ouverts non vides de $T_{0}$ est non vide. Puisque le morphisme $S_{0} \rrr S$ est plat, et que l'immersion diagonale $\Delta_{h}$ est quasi-compacte, le morphisme diagonal $T_{0} \rrr T_{0}\times_{S_{0}}T_{0}$ est encore \scd\, (lemme \ref{lA.6}); soient $U$ et $V$ des ouverts non vides de $T_{0}$ ; puisque $S_{0}$ a un seul point, le produit fibr\'e $U\times_{S_{0}}V$ est non vide ; comme l'immersion $U \cap V \rrr U\times_{S_{0}}V$ est \scd e, l'intersection $U\cap V$ est non vide.
\medskip

Voici une autre fa\c con de voir la relation entre les propri\'et\'es $a)$ et $c)$.

\n Remarquons d'abord que pour un isomorphisme local $h : T \rrr S$, l'ensemble sous-jacent au sch\'ema $T \times_{S}T$ est \'egal au produit fibr\'e des ensembles sous-jacents aux facteurs ; en effet, soient $x$ un point de $T\times_{S}T$, $x_{i} = d_{i}(x)$ ses deux projections dans $T$, et $s$ son image dans $S$ ; comme les extensions $\kappa(s) \rrr \kappa(x_{i})$ sont des isomorphismes, le sch\'ema $\s(\kappa(x_{0})\otimes_{\kappa(s)}\kappa(x_{1})) \subset T\times_{S}T$ est r\'eduit \`a un point, n\'ecessairement \'egal \`a  $x$ ; de plus les homomorphismes locaux
$$
\xymatrix{& \oo_{T,x_{0}} \ar[dr] &\\
\oo_{S, s} \ar[ur] \ar[dr] &&   \oo_{T\times_{S}T, x}\\
& \oo_{T, x_{1}} \ar[ur] &}
$$
sont tous des isomorphismes. Par suite, si $x$ est un point maximal de $T\times_{S}T$, alors $x_{0}$ et $x_{1}$ sont des points maximaux de $T$ qui ont la m\^eme image dans $S$. Cela montre que $h$ induit une injection sur l'ensemble des points maximaux de $T$(conditionn $c)$)  si et seulement si les points maximaux de $T\times_{S}T$ sont dans l'image du morphisme diagonal. Enfin, puisque ce morphisme diagonal est une immersion ouverte, cette derni\`ere propri\'et\'e \'equivaut, lorsque $T$ est r\'eduit,  \`a la propri\'et\'e $a)$ : le morphisme diagonal $\Delta_{h} : T \rrr T\times_{S}T$ est \scd.
\medskip

\n $d) \Leftrightarrow a)$\, Rappelons d'abord le crit\`ere suivant:

 \n \cite[3.1.8]{EGAIV2} {\it Soit $j : U \rrr V$ une immersion ouverte de sch\'emas localement no\'eth\'eriens. Alors  $j$ est \scd e si et seulement si les points associ\'es de $V$ sont dans $U$, i.e.  si {\rm Ass}$(U) = {\rm Ass}(V)$.}
 \medskip
 
Supposons maintenant que $T$ soit localement noeth\'erien. Soit $x \in T$ un point associ\'e ; pour tout $x' \in T$ tel que $h(x) = h(x')$, on a un isomorphisme $\oo_{S, h(x)} \, \wt{\rrr}\, \oo_{T, x'}$, donc $x'$ est aussi associ\'e. Le raisonnement fait plus haut pour les points maximaux, mais appliqu\'e maintenant aux points associ\'es, montre que le morphisme $h$ induit une injection {\rm Ass}$(T) \rrr ${\rm Ass}$(S)$ si et seulement si l'ensemble {\rm Ass}$(T\times_{S}T)$ est contenu dans l'ouvert $\Delta_{h}(T)$, c'est-\`a-dire si $\Delta_{h}$ est \scd.
\end{proof}


\subsection{Existence d'isomorphismes locaux} (cf. A.3) 
Sont rassemblés ici divers critères assurant que des morphismes sont des isomorphismes locaux ; ils seront souvent utilisés dans la suite. 

\begin{prop}\phantomsection\label{cp} Soient $f : Y \rrr X$ et $t : X_{0}\rrr X$ des morphismes de schémas, et $f_{0} : Y_{0}=X_{0}\times_{X}Y \rrr X_{0}$ le morphisme déduit de $f$ par changement de base. On suppose que $f$ est plat, que $t$ est quasi-compact, quasi-séparé et \scd , et enfin que $f_{0}$ est un monomorphisme. Alors, 
\begin{itemize}
\item[a)] Si $f$ est affine, c'est un isomorphisme.
\item[b)] Si $f$ est ouvert, c'est un isomorphisme local.
\end{itemize}
\end{prop}

L'existence du morphisme $t : X_{0}\rrr X$ possédant les propriétés indiquées est en particulier vérifiée lorsque $X$ et $Y$ sont intègres et que $f$ est birationnel; il suffit alors de prendre pour $X_{0}$ le point générique de $X$.

\begin{proof}
Considérons d'abord le cas où $f$ est affine. Quitte à se restreindre à un ouvert affine de $X$, on peut supposer de plus que $f$ est fidèlement plat. Introduisons le $\oo_{X}$-module $F = {\rm Coker}(\oo_{X} \rrr f_{\star}(\oo_{Y}))$. Il est quasi-coh\'erent, et il est plat \cite[I \S 3.5, prop. 9]{AC}. Puisque $f$ est affine, l'homomorphisme canonique $t^{\star}f_{\star}(\oo_{Y}) \rrr  {f_{0}}_{\star}(\oo_{Y_{0}})$  est bijectif \cite[9.3.3]{EGAG}; d'où l'on tire l'isomorphisme
$$
t^{\star}(F) \simeq  {\rm Coker}(\oo_{X_{0}} \rrr f_{\star}(\oo_{Y_{0}})).
$$
 D'après l'hypothèse $f_{0}$ est un monomorphisme  plat et affine, et il est surjectif ; c'est donc un isomorphisme; par suite, on a   $t^{\star}(F) = 0$ ; comme $t$ est un morphisme quasi-compact, quasi-séparé et  \scd, $F$ est nul (lemme \ref{lA.5}).
\medskip

On suppose maintenant que $f$ est ouvert, et on va se ramener au cas $a)$. Soit $y$ un point de $Y$ et $V$ un ouvert affine de $Y$ contenant $y$ et tel que l'ouvert $U = f(V)$ soit contenu dans un ouvert affine de $X$; il est donc séparé et la restriction de $f$ à $V$ donne un morphisme $g : V \rrr U$ qui est fidèlement plat et \emph{affine} \cite[5.3.10]{EGAG}, et on va voir que $g$ est un isomorphisme. Pour pouvoir utiliser $a)$, il faut vérifier que le morphisme $g_{0} : V_{0} \rrr U_{0}$ est un isomorphisme (notations évidentes). Or, posons $W = f^{-1}(U) \subset Y$, ce qui permet d'écrire $g$ comme le composé
$$
V \; \stackrel{j}{\rrr} \; W \; \stackrel{f'}{\rrr} \; U ,
$$
où $j$ est une immersion ouverte et où $f' = f\times_{X}1_{U}$ est fidèlement plat. On en tire la décomposition $g_{0} = f'_{0}j_{0}$ ; mais $g_{0}$ est surjectif, et $f'_{0}$ est un isomorphisme ; l'immersion ouverte $j_{0}$ est donc surjective, c'est-à-dire un isomorphisme ; cela montre que $g_{0}$ est un isomorphisme
\end{proof}

\begin{prop}\phantomsection\label{l2.2} Soient $U \stackrel{v}{\rrr} W \stackrel{w}{\rrr} V$ deux morphismes de sch\'emas  poss\`edant les propri\'et\'es suivantes : 
\begin{thlist}
\item  $w$ est affine plat, et, en outre, l'une des hypoth\`eses suivantes est satisfaite :
\begin{itemize}
\item[a)] le morphisme $w$ est ouvert ; 
\item[b)]  l'ensemble des composantes irr\'eductibles de $V$ est localement fini, et $w$ est  de type fini ;
\end{itemize}
\item $v$ est quasi-compact, et \scd, i.e. l'application canonique $\oo_{W} \rightarrow ~v_{\star}(\oo_{U})$ est injective  (cf. définition \ref{dA.1});
\item le compos\'e $t = wv : U\rrr V$ est une immersion ouverte.
\end{thlist}
\n Alors $w$ est une immersion ouverte.
\end{prop}

Notons que si le morphisme $w$ est de présentation finie et plat alors il est ouvert, et l'hypothèse (i)  $a)$ est donc satisfaite ; la condition sur les composantes irréductibles dans (i)   $b)$ n'est donc requise que si $w$ est de type fini sans \^{e}tre  de présentation finie.

La d\'emonstration proc\`ede par \'etapes.
\medskip

\n I. Montrons d'abord que, en posant  $W_{0} = w^{-1}t(U)$, $v$ et $w$ induisent des isomorphismes 
$$
U \; \;  \wt{\rrr} \; \; W_{0}\; \; \wt{\rrr}\; \; t(U) ,
$$
et que l'immersion ouverte $W_{0} \rrr W$ est quasi-compacte et \scd e.
 Notons $v_{0} : U \rrr W_{0}$  et  $w_{0} : W_{0} \rrr t(U)$ les morphismes en questions. Leur composé $w_{0}v_{0}$ est un isomorphisme d'après (iii) donc $v_{0}$ est une immersion fermée puisque $w_{0}$ est affine (propriété (i)). Par ailleurs, $v_{0}$ est \scd \, comme $v$ (propriété (ii)); donc $v_{0}$ est un isomorphisme, et par suite, $w_{0}$ est aussi un isomorphisme ; enfin, l'immersion $W_{0} \rrr W$ est quasi-compacte et \scd\,  puisque $v$ a ces deux propriétés.
\medskip 

\n II. Montrons que si l'ensemble $w(W)$  est un ouvert de $V$, en particulier si le morphisme $w$ est ouvert (hypothèse (i) $a)$), alors $w$ est une immersion ouverte.

Notons $\wt{w} : W \rrr w(W)$ le morphisme d\'eduit de $w$. Il est fid\`element plat, et affine \cite[9.1.1]{EGAG}, et d'après l'étape  I, la restriction de $\wt{w}$ à l'ouvert rétro-compact  et \scd\, $W_{0} \subset W$  est un isomorphisme. En appliquant \ref{cp}, on voit que  $\wt{w}$ est un isomorphisme, donc que $w$ est une immersion ouverte.
\medskip

\n III. \, Montrons que si $w$ est de type fini, alors c'est une immersion.

\n Il s'agit d'exhiber un ouvert $V'$ de $V$, contenant $w(W)$ et tel que le morphisme $w' : W \rrr V'$ soit une immersion ferm\'ee. Notons d'abord que  la conclusion \'etant locale sur $V$, on peut supposer que $V$ est affine ; cela entra\^{i}ne que $W$ l'est aussi puisque le morphisme $w$ est affine ; notons $B = \Gamma(V) \rrr \Gamma(W) = C$ l'homomorphisme d'anneaux sous-jacent \`a $w$, et soit $\bar{B} \subset C$ l'image de $B$ dans $C$. Pour tout \'el\'ement $x \in C$ l'ensemble $V(x) \subset V$ des $\mathfrak{p}$ tels que $x \in \bar{B}_{\mathfrak{p}}$, est form\'e des $\mathfrak{p}$ qui ne contiennent pas \emph{l'id\'eal des d\'enominateurs} de $x$, soit l'id\'eal $\{s \in B, sx \in \bar{B} \}$ ; $V(x)$ est donc un ouvert. On a $w(W)\subset V(x)$  : en effet, soit  $\mathfrak{p}$ un id\'eal premier de $B$ qui se rel\`eve \`a $C$, donc tel que l'homomorphisme  $B_{\mathfrak{p}} \rrr C_{\mathfrak{p}}$ soit fid\`element plat ; la proposition \ref{cp} entra\^ine que cet homomorphisme est un isomorphisme, donc que l'on a $B_{\mathfrak{p}} \simeq (\bar{B})_{\mathfrak{p}} \, = \, C_{\mathfrak{p}}$ ; en particulier, $\mathfrak{p} \in V(x)$.

\n Comme $C$ est une $B$-alg\`ebre de type fini, engendr\'ee, disons, par $x_{1}, \cdots , x_{n}$, la partie $V' = V(x_{1})\cap V(x_{2}) \cap \cdots \cap V(x_{n})$ est un ouvert de $V$, et elle  contient  $w(W)$ ; son image r\'eciproque $w^{-1}(V')$ est donc \'egale \`a $W$. Par construction, pour tout id\'eal premier $\mathfrak{p} \in V'$, on a $(\bar{B})_{\mathfrak{p}} \, = \, C_{\mathfrak{p}}$, donc $w$ induit une immersion ferm\'ee  $W \rrr V'$.
\bb

\n IV.\, Fin de la d\'emonstration. La conclusion sous l'hypothèse (i) $b)$ résulte du 

\begin{lemme}\phantomsection Soit $w : W \rrr V$ une immersion plate. Si l'ensemble des composantes irréductibles de $V$ est localement fini, alors $w$ est une immersion ouverte.
\end{lemme}
\medskip

Sans hypothèse sur les composantes irréductibles, il existe des immersions fermées plates qui ne sont pas ouvertes  \cite[4.2.3]{EGAG}.
 \medskip
 
\begin{proof} Puisque $w$ est une immersion fermée dans un ouvert de $V$, on peut se restreindre à cet ouvert et supposer que $w$ est une immersion fermée. Pour tout $x \in W$ l'homomorphisme surjectif $\oo_{V, w(x)} \rrr \oo_{W,x}$  est aussi injectif  puisqu'il est fidèlement plat.
 Ces isomorphismes $\oo_{V, w(x)} \, \wt{\rrr}\,   \oo_{W,x}$ montrent que  l'ensemble fermé $w(W)$ est stable par générisation, et il s'agit de voir qu'il est ouvert. Or, son complémentaire est stable par spécialisation ; c'est donc la réunion des composantes irréductibles de $V$ qui sont disjointes de $w(W)$. Par hypothèse, l'ensemble de ces composantes est localement fini ; leur réunion est donc localement (i.e sur chaque ouvert affine) un fermé ; c'est un fermé.
\end{proof}

{\footnotesize Voici enfin une variante de \ref{cp} où l'hypothèse ``générique'' (i.e. liée à $t : X_{0}\rrr X$), est conséquence d'une hypothèse sur la diagonale. 

\begin{cor}\phantomsection\label{l2.1} Soit $f : Y \rrr X$ un morphisme plat quasi-s\'epar\'e et ouvert. On suppose que  le morphisme diagonal $\Delta_{f}$ est \scd.

\n Alors $f$ est un isomorphisme local.
\end{cor}

\begin{proof} En procédant comme dans le début de la démonstration de \ref{cp} b), dont nous reprenons les notations, on se ramène à la situation suivante : on considère un morphisme $g : V \rrr U$, qui est le composé 
$$
V \; \stackrel{j}{\rrr} \; W \; \stackrel{f'}{\rrr} \; U ,
$$
où $j$ est une immersion ouverte et où le morphisme $f' = f\times_{X}1_{U}$ est fidèlement plat, ouvert et à morphisme diagonal $\Delta_{f'}$  \scd; de plus, le morphisme $g = f' j$ est supposé fidèlement plat et {\it affine}. Il s'agit de montrer que $g$ est un isomorphisme. 

C'est le morphisme $f' : W \rrr U$ qui tient ici le rôle du morphisme $t : X_{0}\rrr X$ de loc. cit. On va donc vérifier que le morphisme
$w = g\times_{U} 1_{W}$ est un isomorphisme ; comme il est surjectif, il suffit de vérifier que c'est une immersion ouverte; cela découlera de \ref{l2.2}.

Considérons le diagramme
$$
\xymatrix{V \ar[r]^{g} \ar@<-1ex>[d]_{v} \ar[dr]_{j}& U\\
W' \ar[u] \ar[r]_{w} & W \ar[u]_{f'}} ,
$$
où $W' = V\times_{U}W$, et où $v$ est caractérisé par l'égalité $j = wv$. On sait déjà que $w$ est plat affine et ouvert (propriété (i) de \ref{l2.2}). Pour voir que le morphisme $v$ est quasi-compact et \scd\, (propriété (ii)), on remarque que le carré 
$$
\xymatrix{W \ar[r]^{\Delta_{f'}} &W\times_{U}W\\
V \ar[u] \ar[r]_<<<<<<{v} & W'=V\times_{U}W \ar[u]_{j\times 1}}
$$
est cartésien et que $\Delta_{f'}$ a ces propriétés par hypothèse; enfin la propriété (iii) requiert que $wv$ soit une immersion ouverte, ce qui est bien le cas.
\end{proof}
}
\bigskip

\section{Propriété universelle}\label{s3}

\subsection{Propriété universelle des séparateurs}

\begin{prop}\phantomsection\label{p3.1} Soient $f : T \rrr S$ un morphisme et $T \stackrel{h}{\rrr} E \stackrel{g}{\rrr} S$ un s\'eparateur de $f$. Alors
\begin{thlist}
\item Pour tout ouvert $U$ r\'etrocompact dans $T$, $h(U)$ est un ouvert de $E$, et le morphisme $U \rrr h(U)$ induit par $h$ est un s\'eparateur de $U$.
\item Soit  $U$ un ouvert de $T$ qui est s\'epar\'e sur $S$ (par exemple un ouvert s\'epar\'e). Alors,  la restriction de $h$ induit un isomorphisme de $U$ sur l'ouvert $h(U)$. En particulier, si $T$ est s\'epar\'e sur $S$, $h$ est un isomorphisme.
\item {\rm (Propri\'et\'e universelle du s\'eparateur)}\, Pour tout $S$-morphisme $h' : T \rrr E'$  dans un sch\'ema s\'epar\'e sur $S$, il existe un unique $S$-morphisme $u : E \rrr E'$ tel que $h' = uh$.
\end{thlist}
\end{prop}

\begin{proof} $i)$ L'hypoth\`ese sur $U$ signifie que le morphisme $i : U \rrr T$ est quasi-compact, i.e. que pour tout ouvert affine $V$ de $T$, $U\cap V$ est quasi-compact. Sans hypoth\`ese sur $U$, il est clair que $h(U)$ est un ouvert de $E$ ; c'est donc un sch\'ema s\'epar\'e sur $S$. Notons $h' : U \rrr h(U)$ la restriction de $h$ ; c'est un isomorphisme local. Consid\'erons le carr\'e cart\'esien suivant.
$$
\xymatrix{U \ar[r]^>>>>>{\Delta_{h'}} \ar[d]_{i} & U\times_{h(U)}U \ar[r]^u & U\times_{E}U \ar[d]^{i\times i} \\
T \ar[rr]_{\Delta_{h}} & &T\times_{E}T }
$$
Le morphisme $u$ est un isomorphisme puisque $h(U) \rrr E$ est une immersion ; comme $i\times i$ est une immersion ouverte et que l'immersion $\Delta_{h}$ est quasi compacte et \scd e, il en est de m\^eme de $\Delta_{h'}$ (lemme \ref{lA.6}). Pour voir que $h'$ est un s\'eparateur, il reste \`a v\'erifier que $h'$ est quasi-compact ; or, on peut l'\'ecrire comme le compos\'e de l'immersion ouverte $U \rrr h^{-1}(h(U))$ qui est quasi-compacte puisque $i$ est quasi-compact, et du morphisme $h^{-1}(h(U)) \rrr h(U)$ qui est quasi-compact car $h$ l'est.
\medskip

\n $ii)$\, Soit $U$ un ouvert de $T$ s\'epar\'e sur $S$ ; le morphisme $U \rrr E$ induit par $h$ est s\'epar\'e d'apr\`es le lemme \ref{lA.1}. L'hypoth\`ese sur $\Delta_{h}$ implique d'apr\`es la proposition \ref{p2.1} que $h$ induit une injection sur l'ensemble des points maximaux de $U$ ; donc, (proposition \ref{pA.1} $ii)$), la restriction de $h$ \`a $U$ est une immersion ouverte.\medskip

\n $iii)$ Soit $h' : T \rrr E'$  un $S$-morphisme vers un schéma séparé sur $S$. Dans le carr\'e commutatif
$$
\xymatrix{T \ar[r]^{\Delta_{h}} \ar[d]_{\Delta_{h'}} & T\times_{E}T \ar@{-->}[dl]^w \ar[d]^{\phi}\\
T\times_{E'}T \ar[r]_{\phi'} & T\times_{S} T}
$$
les morphismes $\phi$ et $\phi'$ sont des immersions ferm\'ees puisque $E$ et $E'$ sont s\'epar\'es sur $S$, et $\Delta_{h}$ est \scd\, par hypoth\`ese ; le lemme \ref{lA.9} assure l'existence du morphisme $w$ rendant les triangles commutatifs. Consid\'erons maintenant le diagramme
$$
\xymatrix{T\times_{E}T \ar[d]_{w} \ar@<0.5ex>[r] \ar@<-0.5ex>[r] &T \ar@{=}[d] \ar[r]^h & E\ar@{-->}[d]^u\\
T\times_{E'}T  \ar@<0.5ex>[r] \ar@<-0.5ex>[r] &T \ar[r]_{h'} & E'}
$$
La ligne sup\'erieure est exacte puisque $h$ est fid\`element plat quasi-compact, donc un \'epimorphisme effectif \cite[VIII 5.3]{SGA1} ; d'o\`u l'existence et l'unicit\'e de $u$. 
\end{proof}

Le corollaire suivant pr\'ecise la relation entre séparateur sur $S$ et s\'eparateur  \og absolu \fg\, i.e. sur $S = \s({\bf Z})$. Rappelons (lemme \ref{lA.1}) que si le sch\'ema $E$ est s\'epar\'e tout morphisme $E \rrr S$ est s\'epar\'e, et que si $S$ est s\'epar\'e ainsi que le morphisme $E \rrr S$, alors le sch\'ema $E$ est s\'epar\'e. 

\begin{cor}\phantomsection Soient $f : T \rrr S$ un morphisme de sch\'emas et $u : S \rrr S'$ un morphisme \emph{s\'epar\'e}. 
\begin{thlist}
\item Si $f$ admet le s\'eparateur $T \stackrel{h}{ \rrr} E \stackrel{g}{\rrr} S$, alors $uf$ admet le s\'eparateur  $T \stackrel{h}{ \rrr} E \stackrel{ug}{\rrr} S'$. 

\item Si $uf$ admet le s\'eparateur $T \stackrel{h'}{ \rrr} E' \stackrel{g'}{\rrr} S'$, alors $g'$ se factorise par $u$ : $g' = ug''$, et $T \stackrel{h'}{ \rrr} E' \stackrel{g''}{\rrr} S$ est un s\'eparateur pour $f$.
\end{thlist}
En particulier, si $S$ est s\'epar\'e, c'est-\`a-dire si le morphisme $S \rrr \s({\bf Z})$ est s\'epar\'e, les notions de s\'eparateur pour $f$ et de s\'eparateur ``absolu'' sont \'equivalentes.
\end{cor}

\begin{proof} Rappelons que deux s\'eparateurs du m\^eme morphisme sont canoniquement isomorphes (proposition \ref{p3.1}, $iii)$). Le diagramme suivant r\'esume la situation de l'\'enonc\'e.
$$
\xymatrix{E \ar[dr]_{g} & T \ar[l]_{h} \ar[d]_{f} \ar[r]^{h'} &E' \ar@{-->}[dl]^{g''}  \ar@/^1pc/[ddl]^{g'}\\
& S \ar[d]_{u} &\\
&S' &}
$$
$i)$ \, Comme le morphisme $ug$ est s\'epar\'e, la conclusion de $i)$ d\'ecoule de la d\'efinition.

\n $ii)$\, D'apr\`es la propri\'et\'e universelle du $S'$-s\'eparateur $E'$ (proposition \ref{p3.1}, {\it iii)}), il existe un unique $S'$-morphisme $g'' : E' \rrr S$ tel que $f = g''h'$. Il faut voir que le morphisme $g''$ est s\'epar\'e ; or, l'immersion ferm\'ee $\Delta_{g'} : E' \rrr E'\times_{S'}E'$ se d\'ecompose en
$$
E' \; \stackrel{\Delta_{g''}}{\rrr} \; E'\times_{S}E' \; \stackrel{v}{\rrr}\;  E'\times_{S'}E'
$$
o\`u $v$ est une immersion ; il r\'esulte donc de \ref{sA.1.1} que $\Delta_{g''}$ est une immersion ferm\'ee.\end{proof}

\subsection{Enveloppes séparées}

\begin{defn}\phantomsection\label{d3.1}  Un $S$-schéma $f : T \rrr S$ {\it admet une enveloppe séparée} s'il existe un $S$-morphisme $\varphi : T \rrr F$, avec $F$ séparé sur $S$, qui vérifie la propriété universelle (iii) de la proposition \ref{p3.1}. \end{defn}

La proposition \ref{p3.1} dit qu'un séparateur est une enveloppe séparée. On trouvera en \ref{ex6.1} un exemple de schéma (intègre noethérien  de dimension 1) qui n'admet pas de séparateur bien qu'il possède une enveloppe séparée.
\medskip

 Voici une autre famille  d'exemples. (Voir aussi l'exemple \ref{ex5.1} et le \S\ref{s7.2}.)

\begin{prop}\phantomsection \label{p3.2} Soient $U_{0} \stackrel{i}{\rrr} U \stackrel{\theta}{\rrr} V$ deux morphismes de $S$-schémas. On suppose que les morphismes $i$ et $j = \theta i$ sont des immersions ouvertes, et qu'elles sont  \scd es (ce sera le cas si ces trois schémas sont intègres) ; enfin, on suppose que $V$ est séparé sur $S$.  On définit  un $S$-schéma $T$ par recollement de $U$ et de $V$ le long de $U_{0}$. Alors

$i)$  Ce schéma admet une \emph{enveloppe séparée}. 

 $ii)$ Supposons de plus que $\theta$ soit quasi-compact. Alors, pour que $T$ admette un séparateur sur $S$ il faut et il suffit que $U_{0}$ soit retrocompact dans $U$ et que $\theta$  soit un isomorphisme local quasi-séparé.
\end{prop}
 
\begin{proof} Désignons par $u : U \rrr T$ et $v : V \rrr T$ les immersions ouvertes canoniques, de sorte qu'on a  la figure suivante :
$$
\xymatrix{U_{0}\ar[r]^{i} \ar[d]_{j} \ar[dr]^{w}& U \ar[d]^{u}\\
V \ar[r]_{v} &T}
$$
où $w = ui = vj$ ; notons que  les immersions  ouvertes $u$, $v$ et $w$  sont \scd es puisque  $i : U_{0} \rrr U$ et $j : U_{0} \rrr V$ le sont  et que le carr\'e suivant  de faisceaux sur $T$ est cart\'esien
$$
\xymatrix{\oo_{T} \ar[r] \ar[d] & u_{\star}(\oo_{U}) \ar[d]\\
v_{\star}(\oo_{V}) \ar[r] & w_{\star}(\oo_{U_{0}})} 
$$
comme il découle de la définition du recollement.
 
 \n $i)$ Soit $\varphi : T \rrr V$ le morphisme défini par  $\varphi u = \theta $  et $\varphi v = {\rm id}_{V}$. Montrons que $\varphi$ est une enveloppe séparée. Soit $\psi : T \rrr Z$ un $S$-morphisme vers un $S$-schéma séparé $Z$. Comme $\varphi v = {\rm id}_{V}$ on voit qu'il y a au plus une factorisation de $\psi$ par $\varphi$. Pour l'existence de la factorisation on va montrer que $\psi = \psi v \varphi$, ce qui prouvera notre assertion. 
 $$
 \xymatrix{T \ar[r]^{\varphi} \ar[d]_{\psi} & V \ar@{-->}[dl]^{\psi v}\\
 Z &}
 $$
  Puisque $Z$ est séparé sur $S$, le {\it schéma des co\"incidences } Ker$(\psi, \psi v \varphi)$ est un sous-schéma fermé de $T$ \cite[5.2.5]{EGAG} ; pour voir qu'il est égal à $T$, il suffit de vérifier que l'immersion ouverte \scd e $w = ui = vj$ se factorise par lui, i.e. que l'on a $\psi v \varphi w  = \psi w$. Or, on a $\psi v \varphi u i = \psi v \theta i$ puisque $\varphi u = \theta $ ; comme $\theta i = j$ on trouve bien $\psi v \varphi (u i)  = \psi (vj)$.
  \medskip
  
 \n $ii)$ Supposons que $T$ admette un séparateur sur $S$ ; alors ce dernier est égal à $ \varphi$ en vertu de l'unicité de l'enveloppe séparée ; cela implique que le morphisme $\varphi u = \theta$ soit un isomorphisme local ; montrons que $\theta$ est quasi-séparé. Considérons le morphisme diagonal $\Delta_{\varphi}$. En identifiant $U$ et $V$ avec les ouverts de  $T$ qui leur correspondent, on voit que $T \times_{V} T$ est recouvert par les ouverts $U \times_{V}U ,\, U\times_{V} V, \, V\times_{V}U$  et $V \times_{V}V$ ; il faut prendre garde que dans ces produits fibrés la structure de schéma sur $V$ est donnée par $\varphi$ ; restreint \`a ces ouverts, le morphisme diagonal s'\'ecrit respectivement
$$
\xymatrix{ U\times_{V}U  & U\times_{V}V  & V\times_{V}U  & V\times_{V}V\\
U \ar[u]^{\Delta_{\theta}}& U_{0}\ar[u]_{(i, j)}& U_{0} \ar[u]^{(j, i)} & V\ar@{=}[u]}
$$
Puisque $\varphi$ est un séparateur, il est quasi-séparé, ce qui implique que les morphismes $\Delta_{\theta}$ et $(i, j)$ sont quasi-compacts, et, en particulier, que $\theta$ est quasi-séparé. Composé avec l'isomorphisme  $U\times_{V}V \simeq U$, le morphisme $(i, j)$ devient $i : U_{0} \rrr U$ ; la quasi-compacité   de cette immersion ouverte signifie que $U_{0}$ est rétrocompact dans $U$. Les conditions indiquées sont donc nécessaires.

Abordons la réciproque. Il faut montrer que si $\theta$ est un isomorphisme local quasi-compact et quasi-séparé, et si $U_{0}$ est rétrocompact, alors $\varphi$ est un séparateur. Le fait qu'il soit un isomorphisme local est clair, ainsi que sa quasi-compacité. Les hypothèses, et la décomposition du morphisme diagonal indiquée plus haut montrent que $\Delta_{\varphi}$ est quasi-compact, i.e. que $\varphi$ est quasi-séparé. Il reste à voir que $\Delta_{\varphi}$ est \scd, et il suffit de voir que $\Delta_{\theta}$ est \scd\, puisque, par hypothèse,  $i$ est \scd. Considérons le diagramme commutatif suivant
$$
\xymatrix{ U \ar[r]^<<<<<{\Delta_{\theta}} & U\times_{V}U\\
U_{0} \ar[u]^{i} \ar[r]_<<<<<{\Delta_{j}} &U_{0}\times_{V}U_{0} \ar[u]_{i\times i}}
$$
Comme $j$ est une immersion ouverte, $\Delta_{j}$ est un isomorphisme ; pour pouvoir conclure il suffit donc de montrer que $i \times i$ est \scd. Or l'immersion $i$ est quasi-compacte et \scd e ; le morphisme $i \times 1_{U_{0}} : U_{0}\times_{V}U_{0} \rrr U\times_{V}U_{0}$ est donc encore \scd \, puisque $j : U_{0} \rrr V$ est une immersion ouverte ; et la platitude de $\theta : U \rrr V$ entraîne que $1_{U}\times i : U\times_{V}U_{0} \rrr U\times_{V}U$ est \scd\, (lemme \ref{lA.6}).
\end{proof}

 Dans cet esprit, on peut hasarder la conjecture suivante :

\begin{conj}\phantomsection\label{cj3.1} Soit $T$ un schéma intègre noethérien. On suppose qu'il existe un morphisme universellement fermé surjectif et birationnel de schémas noethériens intègres $p : T' \rrr T$, et que $T'$ possède un séparateur. Alors $T$ possède une enveloppe séparée $\varphi : T \rrr F$ au sens de la définition \ref{d3.1}. 
\end{conj}

Il est même possible qu'en désignant encore par $T_{1}$ l'adhérence schématique de la diagonale de $T\times T$, on ait une suite exacte dans la catégorie des schémas
$$
\xymatrix{T_{1} \ar@<0.5ex>[r]^{d_{0}} \ar@<-0.5ex>[r]_{d_{1}} &T  \ar[r]^{\varphi} & F.}
$$


\section{Passage au quotient}\label{s4}

\subsection{Repr\'esentabilit\'e de certains faisceaux quotients}


\begin{prop}\phantomsection\label{p4.1} Soient $S$ un sch\'ema et $\xymatrix{d_{0}, d_{1} :T_{1} \ar@<0.5ex>[r] \ar@<-0.5ex>[r]& T_{0}}$ un couple de morphismes de $S$-sch\'emas provenant d'une relation d'\'equivalence sur $T_{0}$.  Notons $T_{0} \stackrel{\varepsilon}{\rrr} T_{1} \stackrel{d}{\rrr} T_{0}\times_{S} T_{0}$ la d\'ecomposition canonique du morphisme diagonal, o\`u $d = (d_{0}, d_{1})$. On suppose que le morphisme $\varepsilon$ est \scd, et que $d_{0}$  et  $d_{1}$ sont des isomorphismes locaux. \\
Alors, le faisceau fppf quotient $\wt{T_{0}/T_{1}}$ est repr\'esent\'e par un \emph{sch\'ema} $E$ sur $S$, et le morphisme canonique $h : T_{0}\rrr~E$ est un isomorphisme local. Si, de plus,  $d_{0}$ et $d_{1}$ sont quasi-compacts, alors $h : T_{0}\rrr E$ est quasi-compact.
\end{prop}

\begin{proof} On va d\'efinir un repr\'esentant du faisceau quotient $\wt{T_{0}/T_{1}}$ par recollement des ouverts s\'epar\'es de $T_{0}$ le long d'ouverts attach\'es \`a la relation d'\'equivalence, et pr\'ecis\'es plus bas.
\bb

Soit $U_{0} \subset T_{0}$ un ouvert s\'epar\'e ; soit $U_{1} \subset T_{1}$ la restriction \`a $U_{0}$ de la relation d'\'equivalence. On a donc un diagramme commutatif \`a  carr\'es cart\'esiens, et dont les fl\`eches verticales sont des immersions ouvertes
$$
\xymatrix{T_{0}\ar[r]^{\varepsilon} & T_{1} \ar[r]^>>>>>{d}
&T_{0}\times_{S}T_{0}\\
U_{0} \ar[u] \ar[r]_{\varepsilon'}& U_{1} \ar[u] \ar[r]_>>>>>{d'} &U_{0}\times_{S}U_{0}\ar[u]} 
$$
Montrons que le morphisme $\varepsilon'$ est un isomorphisme, i.e. une immersion ferm\'ee \scd e. En effet, c'est une immersion ferm\'ee d'après \ref{sA.1.1} puisque le morphisme diagonal $d'\varepsilon' : U_{0}\rrr U_{0}\times_{S}U_{0}$ est une immersion ferm\'ee ($U_{0}$ est s\'epar\'e sur $S$, cf. lemme \ref{lA.1}, et que $d'$ est un monomorphisme  ; de plus $\varepsilon'$ est \scd\, puisque c'est la restriction \`a l'ouvert $U_{1}$ du morphisme \scd\,  $\varepsilon$.\medskip

\n Cela montre que la relation d'\'equivalence induit l'\'egalit\'e sur les ouverts de $T_{0}$ qui sont s\'epar\'es.
\medskip

\n Soit maintenant $U'_{0}$ un second ouvert s\'epar\'e de $T_{0}$, et soit $W \subset T_{1}$ l'ouvert image r\'eciproque  de $U_{0}\times_{S}U'_{0} \subset T_{0}\times_{S}T_{0}$, de sorte que les carr\'es de la figure suivante sont  cart\'esiens

$$
\xymatrix{T_{0}\ar[r] & T_{1} \ar[r]^>>>>>{d}
&T_{0}\times_{S}T_{0}\\
U_{0}\cap U'_{0} \ar[u] \ar[r] & W \ar[u] \ar[r] &U_{0}\times_{S}U'_{0}\ar[u]} 
$$
Puisque $d_{0}$ et $d_{1}$ sont des isomorphismes locaux, ce sont des applications ouvertes ; donc $d_{0}(W)$ est un ouvert de $U'_{0}$, et $d_{1}(W)$ est un ouvert de $U_{0}$.
Montrons que $d_{0}$ et $d_{1}$ induisent des isomorphismes
$$
U'_{0} \supset d_{0}(W) \quad \widetilde{\longleftarrow} \quad W \quad \widetilde{\rrr}\quad d_{1}(W) \subset U_{0} .
$$
 Il suffit  de voir que les restrictions de $d_{0}$  et $d_{1}$  \`a $W$ sont des applications (ensemblistes) injectives puisqu'un isomorphisme local injectif est une immersion ouverte ; montrons-le pour $d_{1}$. Soient donc $w$ et $w'$ deux points de $W$ tels que $d_{1}(w) = d_{1}(w')$ ; notant $\sim$ la relation d'\'equivalence induite par $T_{1}$ sur les points de $T_{0}$, on a, puisque $W \subset T_{1}$,
$$
d_{0}(w) \; \sim \;  d_{1}(w) \; = \; d_{1}(w') \; \sim \; d_{0}(w') .
$$
Par transitivit\'e, on en tire l'\'equivalence $d_{0}(w) \sim d_{0}(w')$ ; mais ces deux \'el\'ements sont dans l'ouvert  s\'epar\'e $U'_{0}$, et sur un tel ouvert, on a vu que  la relation d'\'equivalence est l'identit\'e ; donc $d_{0}(w) = d_{0}(w')$, et cela entra\^ine l'\'egalit\'e voulue $w = w'$ puisque le morphisme $d = (d_{0}, d_{1})$ est une immersion.
\medskip

\n On va d\'efinir  le \emph{sch\'ema}  qui repr\'esentera le quotient en recollant les couples d'ouverts s\'epar\'es de $T_{0}$ le long des ouverts du type $d_{0}(W)$ et $d_{1}(W)$ introduits ci-dessus. Précisément, soit $(U_{\lambda})$ un recouvrement de $T_{0}$ par des ouverts s\'epar\'es, par exemple des ouverts affines \cite[5.3]{EGAG}. Pour tout couple $(\lambda, \mu)$ on note $W_{\lambda \mu}$ l'ouvert de $T_{1}$ induit par l'ouvert $U_{\lambda}\times_{S}U_{\mu}$, de sorte qu'on a un carr\'e cart\'esien
$$
\xymatrix{W_{\lambda \mu} \ar[d] \ar[r] & U_{\lambda}\times_{S}U_{\mu} \ar[d]\\
T_{1} \ar[r]_<<<<<<{d} & T_{0}\times_{S}T_{0} } 
$$
D'apr\`es ce qui pr\'ec\`ede, les projections $d_{0}$ et $d_{1}$ induisent des isomorphismes de $W_{\lambda \mu}$ sur les images $d_{0}(W_{\lambda \mu}) = V_{\mu \lambda}$ et $d_{1}(W_{\lambda \mu}) = V_{\lambda \mu}$, lesquelles sont des ouverts ; il existe donc un unique isomorphisme $\varphi_{\mu \lambda}$ rendant commutatif  le diagramme suivant
\begin{equation}\label{eq4.1}
\xymatrix{& W_{\lambda \mu} \ar[dl]_{d_{1}} \ar[dr]^{d_{0}} &\\
U_{\lambda} \supset V_{\lambda \mu} \ar[rr]_{\varphi_{\mu \lambda}} && V_{\mu \lambda} \subset U_{\mu}} 
\end{equation}
Comme ces trois fl\^eches sont des isomorphismes, et que l'on peut se repr\'esenter $W_{\lambda \mu}$ comme l'ensemble des couples $(x, y) \in U_{\lambda} \times U_{\mu}$ tels que $x \sim y$, l'isomorphisme $\varphi_{\mu \lambda}$ est caract\'eris\'e par la propri\'et\'e suivante : 
\medskip

\qquad {\it Pour $x \in V_{\lambda \mu}$, $\varphi_{\mu \lambda}(x)$ est l'unique \'el\'ement  $y \in U_{\mu}$ tel que $x \sim y$.} 
\medskip

Il faut vérifier la  {\it condition de recollement} ; elle sera une cons\'equence de la transitivit\'e de la relation d'\'equivalence comme on va le voir en suivant \cite[4.1.7]{EGAG}, dont nous adoptons les notations, et aussi \cite[I.16]{TG}. 

\n Il est clair que $\varphi_{\lambda \lambda}$ est l'identit\'e de $U_{\lambda}$. La condition de recollement \`a v\'erifier s'\'enonce alors ainsi ({\it loc. cit.}) :  pour tout triplet $(\lambda, \mu, \nu)$, on a $\varphi_{\mu \lambda}(V_{\lambda \mu} \cap V_{\lambda \nu})  \subset V_{\mu \nu}$, et, sur cette intersection, on a 
$$
\varphi_{\nu \lambda} \; = \; \varphi_{\nu \mu}\circ \varphi_{\mu \lambda} .
$$
Or, soit $x \in V_{\lambda \mu} \cap V_{\lambda \nu}$ ; la caract\'erisation des isomorphismes $\varphi$ conduit aux relations
$$
\varphi_{\mu \lambda}(x) \; \sim \; x \; \sim \; \varphi_{\nu \lambda}(x)
$$
La transitivit\'e de la relation d'\'equivalence montre qu'on a donc un couple d'\'el\'ements \'equivalents 
$$
(\varphi_{\mu \lambda}(x), \varphi_{\nu \lambda}(x)) \, \in \,   U_{\mu} \times U_{\nu} .
$$
L'interprétation des ouverts $W$ rappel\'ee en \eqref{eq4.1} implique l'existence d'un unique \'el\'ement  $ z \in W_{\mu \nu}$ tel que 
$\varphi_{\mu \lambda}(x) = d_{1}(z)$  et  $\varphi_{\nu \lambda}(x) = d_{0}(z)$ ;  cela  s'\'ecrit aussi 
$$
\varphi_{\nu \lambda}(x) \; = \; \varphi_{\nu \mu}(\varphi_{\mu \lambda}(x)).
$$
La condition de recollement est donc satisfaite.
\bb

\n Soit $E$ le sch\'ema d\'efini par le recollement des schémas $U_{\lambda}$ en identifiant l'ouvert $V_{\lambda \mu} \subset U_{\lambda}$ à l'ouvert $V_{\mu \lambda} \subset U_{\mu}$ au moyen de l'isomorphisme $\xymatrix{V_{\lambda \mu} \ar[r]^{\varphi_{\mu \lambda}} & V_{\mu \lambda}}$. On aura remarqué les inclusions $U_{\lambda}\cap U_{\mu} \subset V_{\lambda \mu}$ et que l'isomorphisme $\varphi_{\mu \lambda}$ induit l'identité sur  $U_{\lambda}\cap U_{\mu}$ ; ainsi les recollements qui définissent $T_{0}$ sont conséquences de ceux qui définissent $E$ : on recolle \og davantage\fg. Il existe donc un morphisme de schémas  $h : T_{0} \rrr E$ ;  c'est   un isomorphisme local surjectif puisque la restriction de $h$ aux ouverts $U_{\lambda}$  est une immersion ouverte ; enfin, la d\'efinition du recollement signifie que les carr\'es
$$
\xymatrix{W_{\lambda \mu} \ar[r]^{d_{1}} \ar[d]_{d_{0}} & U_{\lambda} \ar[d]^h\\
U_{\mu} \ar[r]_{h}& E}
$$
sont \emph{cart\'esiens} ; comme  les ouverts de la forme $U_{\lambda}\times_{E}U_{\mu}$ recouvrent $T_{0}\times_{E}T_{0}$, le morphisme canonique 
\begin{equation}\label{eq4.2}
 T_{1} \rrr T_{0}\times_{E}T_{0}
\end{equation}
  est un isomorphisme.
  \medskip
  
  Cet isomorphisme permet de conclure que $E$ repr\'esente le faisceau quotient $\wt{T_{0}/T_{1}}$ ; rappelons l'argument.
Par d\'efinition, on a une suite exacte de faisceaux (o\`u  la notation $\wt{X}$ d\'esigne ici le faiceau fppf repr\'esent\'e par le sch\'ema $X$)
$$
\xymatrix{
\wt{T_{1}} \ar@<0.7ex>[r]^{d_{1}} \ar@<-0.7ex>[r]_{d_{0}} & \wt{T_{0}} \ar[r] & \wt{T_{0}/T_{1}} }. 
$$
On a donc un morphisme de faisceaux $ \wt{T_{0}/T_{1}} \rrr \wt{E}$, qui est surjectif puisque $T_{0}\rrr E$ est un isomorphisme local surjectif, donc un morphisme fppf ; mais ce morphisme de faisceaux est aussi injectif en vertu de l'isomorphisme \eqref{eq4.2}.
\medskip

Supposons de plus que les morphismes $d_{0}$ et $d_{1}$ soient quasi-compacts, et montrons que $h : T_{0}\rrr E$ est alors quasi-compact. Lorsque $U$ parcourt les ouverts affines de $T_{0}$ les $h(U) \subset E$ sont des ouverts, d'ailleurs isomorphes \`a $U$, et qui recouvrent $E$. Il faut donc voir que pour un tel $U$, l'ouvert $h^{-1}h(U)$ est quasi compact. Or, on a $h^{-1}h(U) = d_{0}(d_{1}^{-1}(U))$ ; de plus  l'ouvert $d_{1}^{-1}(U)$ est quasi-compact puisque $d_{1}$ est un morphisme quasi-compact, et l'image par le morphisme surjectif $d_{0}$ d'un quasi-compact est quasi-compact.
\end{proof}

\begin{rque}\phantomsection Contrairement \`a l'\'enonc\'e trop optimiste figurant dans des versions ant\'erieures de ce texte, on ne peut pas se passer d'une hypoth\`ese sur le morphisme $\varepsilon : T_{0}\rrr T_{1}$. En effet, {\sc L. Moret-Bailly} nous a fait remarquer que, d'apr\`es {\sc Hironaka}, il existe un sch\'ema $X$, propre sur un corps et muni de l'action d'un groupe $G$ \`a deux \'el\'ements, tels que l'action soit libre et que l'espace alg\'ebrique quotient ne soit pas un sch\'ema \cite[p. 14]{Knu71}. Or, l'action \'etant libre, le morphisme $G\times X \rrr X\times_{S}X$ est un monomorphisme, et c'est m\^eme une immersion ferm\'ee puisque $X$ est propre ; donc $G\times X \simeq X \sqcup X$ est une relation d'\'equivalence sur $X$ et les deux projections sur $X$ sont des isomorphismes locaux.
\end{rque}


\section{Crit\`eres d'existence d'un s\'eparateur}\label{s5}

Dans ce paragraphe nous donnons des crit\`eres pour qu'un sch\'ema $T$ sur $S$ admette un s\'eparateur. 

\subsection{Deux énoncés proches}

\begin{thm}\phantomsection\label{p5.1} Soit $f : T \rrr S$ un morphisme quasi-s\'epar\'e, et soit $T_{1} \subset T\times_{S} T$ l'adh\'erence sch\'ematique de la diagonale $\Delta_{f} : T \rrr T\times_{S} T$. Consid\'erons les propri\'et\'es suivantes.
\begin{thlist}
\item Le morphisme  $f$ admet un s\'eparateur (définition \ref{d2.1}).

\item  Les morphismes  compos\'es $\xymatrix{T_{1} \ar[r] &T\times_{S} T\ar@<0.5ex>[r] \ar@<-0.5ex>[r]  &T}$  sont plats et de type fini.
\end{thlist}
Alors on a l'implication (i) $\Rightarrow$ (ii); la réciproque est vraie si, de plus, l'ensemble des composantes irr\'eductibes de $T$ est localement fini.
\end{thm}

\begin{rque}\phantomsection Soit $h : T \rrr E$ un s\'eparateur. La propri\'et\'e pour le morphisme $\Delta_{h}$ d'\^etre \scd\, signifie intuitivement que $h$ induit ``g\'en\'eriquement'' une injection ; c'est le cas si $T$ est int\`egre. Une caract\'erisation pr\'ecise est donn\'ee dans la proposition \ref{p2.1}.

Lorsque $T$ n'est pas int\`egre, la condition sur le morphisme diagonal $\Delta_{h}$ est bien n\'ecessaire pour obtenir la platitude de $T_{1}$ \'evoqu\'ee en $ii)$. Voici, en effet, un exemple d'isomorphisme local $h : T \rrr E$ sur un sch\'ema noeth\'erien affine, pour lequel $T_{1}$ n'est pas plat sur $T$ ; cet isomorphisme local $h$ n'admet donc pas de séparateur.

 Soient $k$ un corps et $A = k[X, Y]/(XY)$ ; on pose $ E = \s(A)$ (deux droites s\'ecantes). Notons que le compl\'ementaire de la droite $X = 0$ est la droite \'epoint\'ee $U_{0} = \s(A_{X}) = \s(k[X]_{X})$. Soit $T$ le sch\'ema obtenu en recollant  deux exemplaires $U$ et $V$ de $\s(A)$ le long de l'ouvert $U_{0}$ (deux droites s\'ecantes, l'une \'etant d\'edoubl\'ee). Comme $U$ et $V$ sont isomorphes \`a $E$, on a un isomorphisme local $h : T \rrr E$. Restreint \`a l'ouvert $U\times_{E}V \subset T\times_{E}T$, le morphisme $\Delta_{h}$ s'\'ecrit
$$
U \cap V \; \rrr \; U\times_{E}V.\leqno{(\star)}
$$
Mais, $h$ induisant un isomorphisme $V \simeq E$, la projection sur $U$ est un isomorphisme $U\times_{E}V \, \wt{\rrr} \,  U$ qui permet d'identifier l'application $(\star)$ ci-dessus \`a l'immersion ouverte $\s(A_{X}) = U_{0} \rrr U=\s(A)$ ; sa factorisation par adh\'erence sch\'ematique correspond aux morphismes d'anneaux
$$
A \; \rrr \; A/YA \; \rrr \; A_{X}, \qquad {\rm soit} \qquad k[X, Y]/(XY) \; \rrr \; k[X] \; \rrr \; k[X]_{X}.
$$
Il est clair que $A/YA$ n'est pas plat sur $A$.\qed
\end{rque}

Cependant, des hypoth\`eses l\'eg\`erement plus restrictives, et plus maniables, rendent inutile  la condition sur $\Delta_{h}$ ; elles conduisent \`a l'\'enonc\'e suivant : 

\begin{prop}\phantomsection\label{p5.2} Soient $S$ un sch\'ema  int\`egre de point g\'en\'erique $\eta$, et  $f : T \rrr S$ un morphisme plat et quasi-s\'epar\'e.  Consid\'erons les propri\'et\'es suivantes :
\begin{thlist}
\item Le morphisme $f$ se factorise en 
$$
T \; \stackrel{h}{\rrr} \; E \; \stackrel{g}{\rrr} \; S
$$
 o\`u $g$ est  \emph{s\'epar\'e} et o\`u $h$ est un isomorphisme local quasi-compact et  qui induit un isomorphisme $T_{\eta}\; \rt \; E_{\eta}$.
\item La fibre g\'en\'erique $T_{\eta} \rrr \eta$ est s\'epar\'ee, et en d\'esignant par $T_{1} \subset T\times_{S}T$ l'adh\'erence sch\'ematique de la diagonale $\Delta_{f} : T \rrr T\times_{S}T$, les morphismes compos\'es $\xymatrix{T_{1} \ar[r] &T\times_{S}T\ar@<0.5ex>[r] \ar@<-0.5ex>[r]  &T}$  sont plats et  de type fini.
\end{thlist}

Alors on a l'implication (i) $\Rightarrow$ (ii); la réciproque est vraie si, de plus, l'ensemble des composantes irr\'eductibes de $T$ est localement fini.

Si ces propri\'et\'es sont v\'erifi\'ees, le morphisme $h$ de {\rm (i)} est un s\'eparateur i.e. $h$ est aussi quasi-séparé et $\Delta_{h}$ est \scd.
\end{prop}
 
\subsection{D\'emonstration du théorème \ref{p5.1}  et de la proposition  \ref{p5.2}}

\subsubsection {$i) \Rightarrow ii)$} Supposons que $f$ se factorise en $T \; \stackrel{h}{\rrr} \; E \; \stackrel{g}{\rrr} \; S$.
 Comme $h : T \rrr E$ est un isomorphisme local quasi-compact il en est de m\^eme des  projections  $\xymatrix{T\times_{E}T \ar@<0.5ex>[r]\ar@<-0.5ex>[r]& T}$ ; en particulier, elles sont plates et de type fini.

\n Consid\'erons le diagramme suivant o\`u le carr\'e est cart\'esien
$$
\xymatrix{
T \ar[r]^(.4){\Delta_{h}} & T\times_{E}T \ar[d]  \ar[r]^{\Delta'} & T\times_{S} T\ar[d]^{h\times h}\\
&E \ar[r]_{\Delta_{g}} &  E\times_{S} E}
$$
\n Notons que le morphisme diagonal $\Delta_{h}$ est quasi-compact, i.e. $h$ est quasi-s\'epar\'e : c'est une des hypoth\`eses du théorème \ref{p5.1}, et l'hypoth\`ese de quasi-s\'eparation de $f$ dans la proposition \ref{p5.2} implique que le compos\'e $\Delta' \Delta_{h}$ est quasi-compact, donc que $\Delta_{h}$ l'est aussi puisque $\Delta'$ est une immersion (cf. \ref{sA.1.1}).

\n Montrons que l'adh\'erence sch\'ematique $T_{1}$ de la diagonale dans $T\times_{S}T$ s'identifie \`a $T\times_{E}T$, ce qui entra\^inera que les projections  $d_{0}, d_{1} :\xymatrix{T_{1} \ar@<0.5ex>[r]\ar@<-0.5ex>[r]& T}$ sont des isomorphismes locaux quasi-compacts.
\n Comme $g$ est s\'epar\'e, le morphisme diagonal $\Delta_{g}$ est une immersion ferm\'ee ; donc $\Delta'$ est aussi une immersion ferm\'ee ; il reste donc \`a voir que l'application
$$
\oo_{T\times_{E}T} \; \rrr\; (\Delta_{h})_{\star}(\oo_{T}) \leqno{(\star)}
$$ 
est injective, i.e. que $\Delta_{h}$ est \scd ; or, cela fait partie des hypoth\`eses sur $h$ dans le théorème \ref{p5.1}, $i)$. Sous les hypoth\`eses de la proposition \ref{p5.2} $i)$, les morphismes $T \rrr S$ et  $T \rrr E$ sont plats, donc $T\times_{E}T$ est plat sur $S$ ; par suite (lemme \ref{lA.7}), le morphisme $(T\times_{E}T)_{\eta} \rrr  T\times_{E}T$ est \scd\,; par hypoth\`ese, le morphisme g\'en\'erique $T_{\eta} \, \rrr\, E_{\eta}$ est un isomorphisme, donc le morphisme diagonal induit  un isomorphisme $T_{\eta} \;  \wt{\rrr}\;  (T\times_{E}T)_{\eta}$ sur les fibres g\'en\'eriques ; par suite, le morphisme $\Delta_{h}\circ i : T_{\eta} \rrr T\times_{E}T$ est \scd ; ainsi,  le morphisme vertical de gauche dans le  diagramme
$$
\xymatrix{\oo_{T\times_{E}T} \ar[r]^{(\star)} \ar[d] &  (\Delta_{h})_{\star}(\oo_{T}) \ar[d]\\
 (\Delta_{h}\circ i)_{\star}(\oo_{T_{\eta}})\ar@{=}[r] &(\Delta_{h})_{\star}i_{\star}(\oo_{T_{\eta}})}
 $$
est  injectif. L'injectivit\'e de $(\star)$ provient donc de la commutativit\'e de ce diagramme. 

\n On a donc v\'erifi\'e que les deux morphismes compos\'es $\xymatrix{T_{1} \ar[r] &T\times_{S}T\ar@<0.5ex>[r] \ar@<-0.5ex>[r]  &T}$  sont des isomorphismes locaux quasi-compacts, et aussi que $h$ est un s\'eparateur sous les hypoth\`eses de la proposition \ref{p5.2} $i)$.
\bb

\subsubsection{} R\'eciproquement, on suppose maintenant que l'ensemble des composantes irr\'eductibles de $T$ est localement fini, et que les morphismes $d_{0}, d_{1} : \xymatrix{T_{1} \ar@<0.5ex>[r] \ar@<-0.5ex>[r]  &T}$  sont plats et  de type fini, et on va montrer que ce sont des isomorphismes locaux quasi-compacts. Par sym\'etrie, il suffit de le faire pour $d_{0}$, et on peut supposer que $S$ est affine.

Il s'agit d'exhiber un recouvrement de $T_{1}$ par des ouverts $W$ sur lesquels $d_{0}$ induit une immersion ouverte dans $T$, i.e tels que $d_{0}(W)$ soit un ouvert de $T$ et que $d_{0}$  induise un isomorphisme de $W\;  \rt \; d_{0}(W)$.

Or, le sch\'ema $T\times_{S}T$ est recouvert par les ouverts affines de la forme $U\times_{S}V$, pour $U$ et $V$ des ouverts affines de $T$. Consid\'erons la factorisation par adh\'erence sch\'ematique du morphisme diagonal, soit $T \rrr T_{1} \rrr T\times_{S}T$, et sa restriction \`a $U \times_{S}V$ :
$$
U\cap V \; \stackrel{v}{\rrr} \; W  \; \stackrel{u}{\rrr} \; U\times_{S}V ,
$$
o\`u $u$ est une immersion ferm\'ee et o\`u $v$ est une immersion ouverte (lemme \ref{lA.8}), \scd e, et quasi compacte puisque $uv$ est quasi-compact ($T$ est quasi-s\'epar\'e). Le sch\'ema $W$ est affine comme sous-schéma fermé du schéma affine $U\times_{S}V$. En composant avec la projection sur $V$, on obtient les morphismes
$$
U\cap V \; \stackrel{v}{\rrr} \; W  \; \stackrel{d_{0}}{\rrr} \; V .
$$
Par hypothèse, le morphisme $W \rrr V$ est plat et de type fini ; la proposition \ref{l2.2} permet donc de conclure que  $d_{0}$ est une immersion ouverte.
\bb

\n  Puisque les morphismes $d_{0}, d_{1} : \xymatrix{T_{1}\ar@<0.5ex>[r] \ar@<-0.5ex>[r]  &T}$ sont plats,  on peut appliquer la proposition \ref{pB.1}  \`a la relation triviale $T_{\star}$, celle  pour laquelle pour tout $n$, $T_{n} = T$, muni de l'application diagonale dans $T^{n+1}$ ; on en tire que les morphismes $\xymatrix{d_{0}, d_{1} :T_{1} \ar@<0.5ex>[r] \ar@<-0.5ex>[r]& T}$ d\'efinissent une relation d'\'equivalence sur $T$. Mais on vient de voir que ces morphismes sont  des isomorphismes locaux de type fini, et, de plus, le morphisme canonique $\varepsilon : T \rrr T_{1}$ est \scd\, par construction ; la proposition \ref{p4.1} affirme alors qu'il existe un $S$-sch\'ema $E$ et un isomorphisme local quasi-compact et surjectif de $S$-sch\'emas $ h : T \rrr E$ tel que $T_{1} \rrr T\times_{E}T$ soit un isomorphisme.

\n Il reste \`a voir que $E$ est s\'epar\'e sur $S$ ; mais $h$ est fid\`element plat et quasi-compact, et on dispose du carr\'e cart\'esien
$$
\xymatrix{
T_{1} = T\times_{E}T \ar[d]  \ar[r]^>>>>>>{d} & T\times_{S} T\ar[d]^{h\times h}\\
E \ar[r]_{\Delta_{g}} &  E\times_{S} E}
$$
dans lequel $d$ est une immersion ferm\'ee puisque $T_{1}$ est une adh\'erence sch\'ematique ; donc $\Delta_{g}$ est une immersion ferm\'ee \cite[2.3.13]{EGAIV2}.

\n Enfin le morphisme diagonal $\Delta_{h} : T \rrr T\times_{E}T$ est \scd\,  
puisqu'il s'identifie au morphisme $\varepsilon :T \rrr T_{1}$.
\medskip

Dans  la proposition \ref{p5.2} $ii)$, on suppose de plus que $S$ est int\`egre de point g\'en\'erique $\eta$, et que la fibre g\'en\'erique $T_{\eta} \rrr \eta$ est s\'epar\'ee ; on en déduit que l'immersion diagonale $T_{\eta} \rrr T_{\eta}\times_{\eta}T_{\eta}$ est  ferm\'ee, donc que $T_{\eta} = T_{1 \eta}$, ce qui entra\^ine l'isomorphisme $T_{\eta}\;  \wt{\rrr}\;  E_{\eta}$.
\medskip

Ceci conclut les deux démonstrations.\qed

\subsection{Conséquences}

\begin{cor}\phantomsection\label{c5.1} Soit $S$ un sch\'ema  normal dont l'ensemble des composantes irr\'eductibles est localement fini. Soit $f : T \rrr S$ un morphisme \emph{\'etale} quasi-compact et quasi-s\'epar\'e. Alors $f$  admet un s\'eparateur  $ h : T \rrr E$, et pour tout point maximal $\eta$ de $S$, ce morphisme  induit un isomorphisme  $T_{\eta} \; \wt{\rrr}\; E_{\eta}$. Enfin, le morphisme $E\rrr S$ est  étale quasi-compact et quasi-séparé.
\end{cor}

\begin{proof} Comme les anneaux locaux de $S$ sont int\`egres,  les composantes irr\'eductibles de ce sch\'ema sont deux \`a deux disjointes ;  elles sont donc ouvertes puisque leur ensemble est localement fini ; ainsi $S$ est somme de ses composantes irr\'eductibles. On peut donc supposer que $S$ est int\`egre ; soit $\eta$ son point g\'en\'erique. 

V\'erifions que les conditions de  la proposition \ref{p5.2}, $ii)$ sont satisfaites. On note d'abord que, $f$ \'etant \'etale et quasi-compact, la fibre g\'en\'erique $T_{\eta} \rrr \eta$, est discr\`ete (donc s\'epar\'ee) et  a un nombre fini de points, qui sont les points maximaux de $T$.  Comme $T\times_{S}T$ est \'etale sur le sch\'ema normal $S$, c'est un sch\'ema normal (\cite[11.3.13]{EGAIV3}, ou plus directement {\sc Raynaud} \cite[p. 75]{Ray70}). En particulier, les anneaux locaux de ce sch\'ema sont int\`egres. D'autre part,  l'ensemble des  points maximaux de $T\times_{S}T$ est fini puisqu'il est  \'egal \`a l'ensemble des points de la fibre $(T\times_{S}T)_{\eta} = T_{\eta}\times_{\eta}T_{\eta}$, laquelle est un sch\'ema fini sur le corps $\eta$ ; enfin, le morphisme diagonal $\Delta_{f} : T \rrr T\times_{S}T$ est quasi-compact puisque $f$ est quasi-s\'epar\'e. Toutes les hypoth\`eses du lemme  \ref{lA.11} sont satisfaites, et on en tire que l'adh\'erence sch\'ematique  $T_{1}$ de la diagonale est un ouvert et ferm\'e de $T\times_{S}T$ ; chacun des morphismes  $d_{0}, d_{1} : \xymatrix{T_{1}\ar@<0.5ex>[r] \ar@<-0.5ex>[r]  &T}$ est donc plat et de type fini.
\end{proof}

\begin{cor}\phantomsection\label{c5.2} Soit $T$ un sch\'ema quasi-s\'epar\'e dont l'ensemble des composantes irr\'eductibles est localement fini. Alors $T$ admet un s\'eparateur si et seulement si pour tout couple $U, V$ d'ouverts affines de $T$, le sch\'ema $U \cup V$ admet un s\'eparateur. Il suffit m\^eme pour cela qu'il existe un recouvrement de $T$ par des ouverts affines $U_{\lambda}$ tel que chaque r\'eunion $U_{\lambda} \cup U_{\mu}$ de deux d'entre eux admette un s\'eparateur.
\end{cor}

\begin{proof} Montrons que la condition est n\'ecessaire.  Soient $U$ et $V$ deux ouverts affines. Puisque $T$ est quasi-s\'epar\'e, l'intersection d'un ouvert affine avec $U \cup V$ est quasi-compacte, autrement dit, l'ouvert $U \cup V$ est r\'etrocompact dans $T$ ; il suffit donc d'appliquer la proposition \ref{p3.1}, $i)$.

 Montrons que la condition est suffisante.  L'ouvert $(U\cup V) \times (U\cup V)  \subset T \times T$ est r\'eunion des ouverts
$U \times U,\, U\times V,\, V\times U$ et $V \times V$ ; les traces du ferm\'e $T_{1} \rrr T\times T$ sur le premier et le dernier de ces ouverts sont isomorphes, respectivement, \`a  $U$ et \`a $V$ ; elle sont donc plates et de type fini sur $T$ sans hypoth\`eses. Notons $W = T_{1} \cap (U\times V)$ ; l'hypoth\`ese et  l'implication $i) \Rightarrow ii)$ du théorème \ref{p5.1}, montrent que les deux projections $ d_{1} : W \rrr U$ et $d_{0} : W \rrr V$ sont plates et de type fini. Toujours parce que $T$ est quasi-s\'epar\'e, les immersions ouvertes $U \rrr T$ et $V \rrr T$ sont (plates et) de type fini ; enfin, les ouverts  $U\times V$, avec $U$ et $V$ affines, forment un recouvrement de $T \times T$ ; on voit donc que les deux projections de $T_{1}$ sur $T$ sont plates et de type fini ; il suffit alors, pour pouvoir conclure, d'invoquer l'implication  $ii) \Rightarrow i)$ du théorème \ref{p5.1}. 
\end{proof}

\begin{prop}\phantomsection\label{p5.3} Un sch\'ema localement noeth\'erien r\'egulier de dimension $1$ admet un s\'eparateur.
\end{prop}

\begin{proof} Elle utilise  l'implication $ii) \Rightarrow i)$ du théorème \ref{p5.1}. Soit $T$ un sch\'ema localement noeth\'erien r\'egulier de dimension 1. Comme les anneaux locaux de $T$ sont int\`egres, les composantes irr\'eductibles de $T$ sont deux \`a deux disjointes ;  elles sont donc ouvertes puisque leur ensemble est localement fini. On suppose donc d\'esormais que $T$ est int\`egre. Soit $T_{1}$ l'adh\'erence sch\'ematique de la diagonale dans $T\times T$, de sorte qu'on a les morphismes
$$
\xymatrix{ T \ar[r]^{\varepsilon} & T_{1} \ar[r]^d  \ar@/^1pc/[rr]^{d_{0}} \ar@/_1pc/[rr]_{d_{1}} &T\times T  \ar@<0.5ex>[r] \ar@<-0.5ex>[r] &T}
$$
On va montrer directement (i.e. sans utiliser la proposition \ref{l2.2}) que les $d_{i}$ sont des isomorphismes locaux de type fini. Il suffit de le faire pour, disons, $d_{0}$. Soient donc $U$ et $V$ deux ouverts affines non vides de $T$, et consid\'erons la restriction \`a $U\times V$ de la factorisation de la diagonale par adh\'erence sch\'ematique, soit 
$$
U\cap V \; \stackrel{v}{\rrr} \; W  \; \stackrel{u}{\rrr} \; U\times V
$$
o\`u $u$ est une immersion ferm\'ee et o\`u $v$ est une immersion ouverte (lemme \ref{lA.8}), \scd e.  En composant avec la projection sur $V$, on obtient les morphismes
$$
U\cap V \; \stackrel{v}{\rrr} \; W  \; \stackrel{d_{0}}{\rrr} \; V ,
$$
et il s'agit de voir que $d_{0}$  est une immersion ouverte. Les sch\'emas $V$ et  $W$ sont affines. Montrons d'abord que l'ensemble $d_{0}(W)$ est ouvert : comme $V$ est noeth\'erien int\`egre de dimension 1 et que $U\cap V$ est non vide, le ferm\'e compl\'ementaire $V - U\cap V$ est un ensemble fini de points ferm\'es, donc l'ensemble $V - d_{0}(W)$, qui est contenu dans $V - U\cap V$, est ferm\'e.
\medskip

Puisque $v$ est \scd, l'anneau $ B = \Gamma(W, \oo_{W})$ est contenu dans le corps des fractions $K$ de l'anneau de Dedekind  $A = \Gamma(V, \oo_{V})$ ; pour tout id\'eal premier $\mathfrak{p}$ de $A$, $B_{\mathfrak{p}}$ est donc \'egal \`a $K$, ou \`a $A_{\mathfrak{p}}$, et ce dernier cas correspond aux $\mathfrak{p}$ qui se rel\`event \`a $B$, c'est-\`a-dire aux id\'eaux premiers correspondant aux points de l'ouvert $d_{0}(W)$, puisque $W$ est affine. Ainsi,  $d_{0}$ permet d'identifier $W$ \`a l'ouvert $d_{0}(W)$ de $V$; enfin, $d_{0}$ est de type fini puisque c'est une immersion ouverte affine \cite[6.3.4]{EGAG}.
\end{proof}
\section{Ouverts admettant un séparateur}\label{s6}

\subsection{\'Enoncé et premières réductions}

\begin{thm}\phantomsection\label{t6.1} Soit $R$ un anneau noethérien et $T \rrr \s(R)$ un morphisme de type fini. On suppose que $T$ est intègre et normal. Alors il existe un ouvert  $U$ de $T$ contenant tous les points de codimension $1$, et possédant un séparateur.
\end{thm}

Soit $U$ un ouvert non vide. Comme $T$ est intègre, les points de codimension 1 de $T$ qui ne sont pas dans $U$ sont des points maximaux du fermé $T - U$ ; ils sont donc en nombre fini. Par suite, le théorème  est  une conséquence immédiate du résultat suivant :

\begin{prop}\phantomsection\label{p6.3} Gardons  les m\^{e}mes hypothèses sur $R$ et $T$. Soit $U$ un ouvert non vide admettant un séparateur, et soit $s$ un point de $T$ de codimension $1$. Alors, il existe un ouvert $U' \subset U$, contenant tous les points de codimension $1$ de $U$, et un ouvert affine $V$ de $T$ contenant $s$ tels que $U' \cup V$ admette un séparateur.
\end{prop}

La démonstration utilise la conséquence suivante de \ref{p5.1}; désormais, et sauf mention expresse du contraire, tous les produits de schémas seront des produits  sur  $\s(R)$ sans que cette base soit mentionnée. 

\begin{lemme}\phantomsection\label{l6.2} Soit $X = U \cup V$ un schéma réunion de deux ouverts. On suppose que les schémas $U$ et $V$ admettent un séparateur. Alors $X$ admet un séparateur si les deux projections
$$
\xymatrix{& \overline{U\cap V} \ar[dl]_{d_{0}} \ar[d] \ar[dr]^{d_{1}} &\\
V &U\times V \ar[l] \ar[r] & U}
$$
$d_{0}, d_{1}$ sont plates et de type fini, où, par abus de notation, $\overline{U\cap V}$ désigne l'\emph{adhérence schématique} de $U\cap V$ dans $U\times V$.
\end{lemme}

\begin{proof} Soit $X_{1} $ l'adhérence schématique de la diagonale dans $X\times X$. D'après le théorème \ref{p5.1}, le schéma $X$ admet un séparateur si et seulement si les deux projections $\xymatrix{X_{1 } \ar[r] & X\times X  \ar@<0.5ex>[r] \ar@<-0.5ex>[r] & X}$ sont plates et de type fini ; cela se vérifie sur les ouverts du recouvrement 
\[(U\times U, V\times V, U\times V,V\times U)\] 
de $X\times X$ ; ce recouvrement induit sur le fermé $X_{1}$ le recouvrement
$$
(U_{1}, V_{1}, \overline{U\cap V}, \overline{V\cap U}).
$$
Puisque les schémas $U$ et $V$ admettent un séparateur,  les projections de source les schémas $U_{1}$ et $V_{1}$ sont plates et de type fini. 
L'hypothèse sur $\overline{U\cap V}$, et  l'isomorphisme de symétrie $U\times V  \simeq V\times U$, entra\^{i}nent donc la conclusion.
\end{proof}

\subsection{Construction de l'ouvert $U'$} Posons $S = \s(\oo_{T, s})$ ; c'est le spectre d'un anneau de valuation discrète puisque $T$ est normal intègre et noethérien.

\n Notons $F = \overline{U\cap S}$ l'adhérence \emph{schématique} de $U\cap S$ dans $U\times S$. On peut écarter le cas o\`u $s$ est déjà dans $U$, et donc supposer que l'ouvert $U \cap S$ de $S$ est réduit à son point générique $\eta$. Comme $U$ est de type fini sur $R$, le morphisme $d_{0}$ de projection $F \rrr  U\times S \rrr S$ est de type fini ; soit $F_{s} = d_{0}^{-1}(s)$ sa fibre fermée. Pour tout $x \in F_{s}$ l'homomorphisme $\oo_{T, s} \rrr \oo_{F, x}$ est un isomorphisme, en vertu de la maximalité de l'anneau de valutation discrète $\oo_{T, s}$ dans son corps des fractions $K = \oo_{T, \eta}$; de sorte que, si $F_{s} \neq \emptyset$,  le schéma $F$ est la réunion d'un ensemble fini de copies de $S$ recollées en leur point générique. Soit $Z \subset d_{1}(F_{s}) \subset U$ l'ensemble (éventuellement vide) des points $d_{1}(x)$ qui sont de codimension > 1 dans $U$; c'est un ensemble fini, et on pose 
$$
U' \; = \; U \cap (T - \bar{Z}).
$$
Comme la formation de l'adhérence schématique commute à la restriction à l'ouvert $U'\times S \subset U \times S$, le schéma $F' = d_{1}^{-1}(U')$ est l'adhérence schématique de $U'\cap S$ dans $U'\times S$. La projection $d_{1} : F' \rrr U'$ est un morphisme plat puisque pour tout $x \in F'$ tel que $d_{0}(x) = s$, l'anneau $\oo_{U', d_{1}(x)}$ est de valuation discrète.

\subsection{Démonstration de la proposition \ref{p6.3}}En vertu de ce qui précède, quitte à remplacer $U$ par $U'$ on peut faire l'hypothèse supplémentaire

\begin{para} \phantomsection\label{h6.1}
{\it Le morphisme composé  $F = \overline{U\cap S} \rrr U\times S \rrr U$ est plat.}
\end{para}

Nous commencerons par trouver un voisinage ouvert affine $V_{1}$ de $s$ tel que la projection $\overline{U\cap V_{1}} \rrr V_{1}$ soit plate, puis un ouvert $V_{2} \subset V_{1}$ contenant $s$, et tel que la projection $\overline{U\cap V_{2}} \rrr ~ U$ soit plate ; on pourra alors conclure du lemme \ref{l6.2} que l'ouvert $U \cup V_{2}$ admet un séparateur.

Pour la commodité du lecteur, rappelons un résultat général de passage à la limite  énoncé ici sous les hypothèses particulières où nous l'utiliserons pour trouver l'ouvert $V_{1}$.

\begin{vide}[\protect{\cite[11.2.6 (ii)]{EGAIV3}}] \phantomsection\label{t6.2} {\it Soient $S_{0}$ un schéma affine, $(S_{\lambda})$ un système projectif de $S_{0}$-schémas affines, de sorte que la limite projective de ce système est un schéma affine noté $S = \ul S_{\lambda}$. 

\n Soient $X_{0}$ un $S_{0}$ schéma de présentation finie,  $X_{\lambda} = S_{\lambda}\times_{S_{0}}X_{0}$ et $X = \ul X_{\lambda} = S\times_{S_{0}}X_{0}$.

\n Alors, $X$ est plat sur $S$ si et seulement si il existe $\lambda $ tel que le morphisme $X_{\lambda} \rrr S_{\lambda}$ soit plat. $\qed$}
\end{vide}

 Soit $V_{0}$ un ouvert affine de $T$ contenant $s$ ; notons $\mathsf{V}$ l'ensemble des ouverts affines de $T$ contenant $s$ et contenus dans $V_{0}$ ; les systèmes projectifs considérés seront indexés par $\mathsf{V}$ ; en particulier,
$$
\ul_{\mathsf{V}} V =  \s(\mathcal{O}_{T, s}) = S
$$

\n Pour $V \in \mathsf{V}$ on désigne par $F_{V} = \overline{U \cap V}$ l'adhérence schématique de $U\cap V$ dans $U\times V$.

Pour deux ouverts de $\mathsf{V}$, $V' \subset V$,  le carré
$$
\xymatrix{F_{V'} \ar[r] \ar[d] & V' \ar[d]\\
F_{V} \ar[r] &V}
$$
est cartésien puisque la formation de l'adhérence schématique commute au changement de base plat $V' \subset V$ (lemme \ref{lA.7},  \cite[2.3.2]{EGAIV2}). Comme les morphismes $V\rrr V_{0}$ sont des immersions ouvertes affines, il en est de m\^{e}me des morphismes $F_{V} \rrr F_{V_{0}}$ ; ils permettent d'identifier $F_{V}$ à un ouvert de $F_{V_{0}}$, et aussi d'écrire
$$
F \; = \; \bigcap F_{V}.
$$
\medskip

\n Montrons  comment trouver un ouvert $V$ tel que la projection 
$$
F_{V} \rrr U\times V \rrr V
$$
soit plate. Puisque $U$ est de type fini sur le schéma noethérien $\s(R)$, et que $F_{V}$ est fermé dans $U\times V$, $F_{V}$ est de présentation finie sur $V$. On peut donc  utiliser  \ref{t6.2} en rempla\c cant les données $(S_{0}, (S_{\lambda}), X_{0})$ par $(V_{0}, (V)_{V\in \mathsf{V}}, F_{V_{0}})$ ; la limite 
$X = \ul X_{\lambda} = S\times_{S_{0}}X_{0}$ est ici $F = \ul F_{V} = S\times_{V_{0}}F_{V_{0}}$. 

Or, en passant à la limite sur les $V \in \mathsf{V}$, on trouve les morphismes
$$
\ul U\cap V = U \cap S = \eta  \rrr \ul F_{V} = F \rrr \ul V = S
$$
Le morphisme  de type fini $d_{0} : F \rrr S$  est plat puisque $S$ est un trait. D'après \ref{t6.2}  il donc existe $V_{1} \in \mathsf{V}$ tel que le morphisme $d_{0} : F_{V_{1}} \rrr V_{1}$ soit plat ; notons qu'alors pour tout $V \in \mathsf{V}$ tel que $V \subset V_{1}$, le morphisme $F_{V}\rrr V$ est plat.  \medskip

Considérons maintenant la projection sur l'autre facteur $d_{1} : F_{V_{1}} \rrr U$. Soit $W \subset F_{V_{1}}$ son ouvert (non vide) de platitude (\cite[11.1.1]{EGAIV3}). La platitude du morphisme $F \rrr U$  (hypothèse \ref{h6.1}), se traduit par l'inclusion $F \subset W$ ; en identifiant les $F_{V}$ à des ouverts  de $F_{V_{1}}$, elle s'écrit aussi
$$
\bigcap F_{V} \; \subset \; W,
$$
pour $V$ parcourant la famille des ouverts affines de $T$, contenant $s$ et contenus dans $V_{1}$. Le théorème 7.2.5 de \cite{EGAG}  montre alors qu'il existe $V_{2} \subset V_{1}$ tel que $F_{V_{2}}$ soit contenu dans $W$, donc tel le morphisme $F_{V_{2}} \rrr U$ soit plat. Mais plutôt que d'invoquer le résultat très général de loc.cit., on peut, sous nos hypothèses très particulières,  esquisser une idée de la démonstration : soit $X = F_{V_{1}} - W$ le schéma réduit induit sur ce fermé ; posons $Y_{V} = F_{V} \cap X$ ; on obtient une famille filtrante d'ouverts de $X$ d'intersection vide, et il faut voir que l'un des $Y_{V}$ est vide ; comme $X$ est quasi-compact et que la famille est filtrante on peut supposer de plus que $X$ est affine ; puisque les immersions ouvertes $F_{V} \subset F_{V_{1}}$ sont affines, les ouverts $Y_{V}$ sont alors affines ; mais une limite inductive d'anneaux ne peut être nulle que si l'un d'eux l'est déjà. \qed

\subsection{Un exemple}

\begin{ex}\phantomsection\label{ex5.1} Soit $Y = \s(R)$ le spectre d'un anneau local r\'egulier de dimension 2 ; notons $y$ son point ferm\'e et $Y'$ l'ouvert compl\'ementaire.
Soit $p : X \rrr Y$ le morphisme d'\'eclatement de $Y$ en l'idéal maximal $\mathfrak{m}$ de $R$, de sorte que $p$ induit un isomorphisme $p^{-1}(Y') \; \widetilde{\rrr}\; Y'$ ; on note $x$ le point générique de $p^{-1}(y)$. Soit $T$ le sch\'ema 
obtenu par recollement de $X$ et de $Y$ le long de $Y'$ ; plus pr\'ecis\'ement, $T$ est d\'efini par l'exactitude (dans la cat\'egorie des sch\'emas) de la suite
$$
\xymatrix{Y'\simeq p^{-1}(Y') \ar@<0.5ex>[r]^>>>>>{j} \ar@<-0.5ex>[r]_>>>>>{p} & X\sqcup Y \ar[r] &T}  
$$
o\`u $j : Y' \rrr X$ d\'esigne l'immersion ouverte $p^{-1}(Y') \rrr X$.

\n On désignera par le même symbole les éléments ou parties de $X$ et de $Y$ et leur image dans $T$.\medskip

Alors,

\n (a)\;  Le sch\'ema $T$ est r\'egulier de dimension 2. 

\n (b)\;  L'ouvert $X$ de $T$ est s\'epar\'e et $T - X = \{y\}$; en particulier $X$ contient les points de $T$ de codimension 1.

\n (c) \; Aucun ouvert de $T$ contenant $x$ et $y$ n'admet de séparateur.

\n (d)\; L'ouvert (séparé) $Y$ de $T$ est maximal parmi les ouverts admettant un séparateur, mais il ne contient pas le point $x$ qui est de codimension 1.
\end{ex}

Vérifions  (c). Soit $U$ un ouvert de $T$ contenant $x$ et $y$ ; il contient donc l'ouvert  $Y$ ; par suite $U$ est le recollement de $U\cap X$ et de $Y$  le long de l'ouvert $Y'$ ; appliquons la proposition \ref{p3.2}, avec ici $(Y' \subset U\cap X \stackrel{p}{\rrr} Y)$ à la place des données  $(U_{0} \subset U \stackrel{\theta}{\rrr} V)$ de loc.cit. ; comme $U$ contient le point générique $x$ du diviseur exceptionnel $p^{-1}(y) \simeq \mathbf{P}_{1}$, cet ouvert contient aussi un point fermé de $x' \in p^{-1}(y) \subset X$, si bien que l'homomorphisme $\oo_{Y, y} \rrr \oo_{X, x'}$ n'est pas un isomorphisme ; cela montre que la restriction de $p$ à $U \cap X$ n'est pas un isomorphisme local, et que donc $U$ n'admet pas de séparateur (proposition \ref{p3.2}). 

Cet exemple montre que dans la proposition  \ref{p6.3}, on ne peut éviter la restriction de $U$ à $U'$, ici la restriction de $Y$ à $Y'$ : dans $T$, $Y$ est un ouvert séparé, et le point $x$ est de codimension 1, bien que $T$ ne contienne  pas d'ouvert admettant un séparateur et  contenant $x$ et $Y$ ; par contre, l'ouvert séparé $X$ contient $x$ et $Y' = Y - \{y\}$. 
\medskip

Par ailleurs, le morphisme $\varphi : T \rrr Y$ dont la restriction à $X$ est égale à $p$, et dont la restriction à $Y$ est l'application identique, est une enveloppe séparée de $T$, et aussi d'ailleurs son enveloppe affine (proposition \ref{p3.2}).  Soit $q : T' \rrr T$ le morphisme d'éclatement de $T$ en le point fermé $y$ ; en identifiant, comme plus haut $Y$ à son image dans $T$, on voit que la restriction de $q$  à $q^{-1}(Y)$ redonne le morphisme $p : X \rrr Y$, et sa restriction à l'ouvert $X \subset T$ est l'identité ; ainsi, $T'$ est-il  isomorphe au schéma obtenu en recollant deux exemplaires de $X$ le long de l'ouvert $p^{-1}(Y')$ ; le morphisme évident $h' : T' \rrr X$ est donc un séparateur de $T'$.  Finalement, le diagramme suivant résume la situation.
$$
\xymatrix{T' \ar[r]^{h'} \ar[d]_{q} & X \ar[d]^{p}\\
T \ar[r]_{\varphi} & Y}
$$

Cet exemple est compatible avec la conjecture \ref{cj3.1}.

\section{Anneaux locaux apparent\'es}\label{s07}

 Nous allons relier les r\'esultats  qui pr\'ec\`edent aux \og Crit\`eres de s\'eparation d'un sch\'ema int\`egre\fg \, qui font l'objet du \S 8.5 de \cite{EGAG}. Notons que dans la premi\`ere \'edition de EGA  I, les m\^emes d\'eveloppements  \'etaient annonc\'es sous le titre  \og \S 8. Les sch\'emas de Chevalley\fg ; les sch\'emas en question, introduits dans le s\'eminaire Cartan-Chevalley de 1955, \'etaient juste \'evoqu\'es en fin de paragraphe, et cette \'evocation elle-m\^eme a disparu dans la nouvelle \'edition ; mais on la trouve toujours dans \cite[p.125]{God58}.
 
 Le concept d'anneaux locaux \og apparentés\fg, dégagé par {\sc Chevalley}, plutôt abandonné aujourd'hui, permet cependant d'éclairer un peu la démarche suivie ici ; mais il ne semble pas qu'on puisse l'étendre à des schémas qui ne seraient pas intègres.
 
 \subsection{Le langage}
 
\begin{para}\phantomsection\label{6.1} Reprenons \cite[8.5.2]{EGAG}. Soit $K$ un anneau int\`egre (par exemple un corps). Deux sous-anneaux locaux $M$ et $N$ de $K$ sont dits {\it apparent\'es} s'il existe un sous-anneau local $Q$ de $K$ dominant \`a la fois $M$ et $N$. Cette propri\'et\'e est \'equivalente \`a la suivante : soit $P \subset K$ le sous-anneau engendr\'e par $M \cup N$, c'est-\`a-dire l'image de l'homomorphisme canonique
$$
M \otimes N \; \rrr \; K \, ;
$$
alors il existe un id\'eal premier $\mathfrak{p}$ de $P$ tel que les homomorphismes  $M \rrr P_{\mathfrak{p}}$  et $N \rrr P_{\mathfrak{p}}$ soient locaux. 

\bb 

Soit $T$ un sch\'ema int\`egre de point g\'en\'erique $\xi$; désignons par $K = \oo_{T, \xi}$ le corps des fonctions rationnelles sur $T$. Pour tout point $t \in T$, l'homomorphisme de restriction induit un homomorphisme injectif
$$
j_{t} : \oo_{t} \; \rrr\; K .
$$
Contrairement à \cite[8.5]{EGAG}, on n'identifiera pas, dans ce qui suit, l'anneau $\oo_{T, t}$ à son image $j_{t}(\oo_{t}) \subset K$.\bb

Considérons les propri\'et\'es suivantes portant sur un couple $x, y$ de points de $T$:
\medskip

a) $x = y$ ;
\medskip

b) $j_{x}(\oo_{x}) = j_{y}(\oo_{y})$;
\medskip

c) $j_{x}(\oo_{x})$ et  $j_{y}(\oo_{y})$ sont apparentés.
\medskip

\n La séparation de $T$ signifie que pour tout couple $x, y$ de points de $T$ les propriétés a)  et c)  sont équivalentes (corollaire 6.3). On va montrer que, sous une condition usuelle de finitude, le schéma $T$ admet un séparateur si pour tout couple $x, y$ les propriétés b) et c)  sont équivalentes.\bb

Le sch\'ema $T$ \'etant suppos\'e int\`egre, son morphisme diagonal $\Delta_{T} : T \rrr T\times T$  admet une factorisation par adh\'erence sch\'ematique \cite[6.10.5]{EGAG}:
$$
T \; \rrr \; T_{1} \; \rrr \; T\times T .
$$
On note encore  $d_{0}, d_{1} : \xymatrix{T_{1}\ar@<0.5ex>[r] \ar@<-0.5ex>[r]  &T}$ les deux projections sur $T$.
\medskip

\n Puisque l'immersion ouverte $T \rightarrow T_{1}$ est \scd e, et que le morphisme d'inclusion du point générique  $\xi \rrr T$ est aussi \scd, le morphisme composé $\xi \rrr T \rrr T_{1}$ est \scd ; par suite, pour tout  $z \in T_{1}$, on a l'homomorphisme de restriction  $j_{z} : \oo_{T_{1}, z} \rrr \oo_{T, \xi}=K$. 
Comme les morphismes composés $\xymatrix{T \ar[r] &T_{1}\ar@<0.5ex>[r] \ar@<-0.5ex>[r]  &T}$ sont l'identité,  pour $z \in T_{1}$, l'homomorphisme compos\'e
$$
\oo_{T, d_{0}(z)} \rrr \oo_{T_{1}, z} \rrr K 
$$
est \'egal \`a l'injection $j_{d_{0}(z)}$, et idem pour $j_{d_{1}(z)}$.

\end{para}

\begin{prop}\phantomsection\label{p6.1}
Gardons les notations qui pr\'ec\`edent. Soient $x$ et $y$ des points de $T$. Alors  $j_{x}(\oo_{x})$ et $j_{y}(\oo_{y})$ sont apparent\'es si et seulement si il existe $z \in T_{1}$ tel que $x = d_{0}(z)$  et $y = d_{1}(z)$. Autrement dit, la relation d\'efinie sur les points de $T$ par le ferm\'e $T_{1}$ de $T\times T$ est exactement la relation d'apparentement.
\end{prop}

\begin{proof}La condition est suffisante, car l'anneau local $j_{z}(\oo_{T_{1}, z})$ est un sous-anneau local du corps $K$ des fonctions rationnelles sur $T$, et par hypoth\`ese il domine $j_{x}(\oo_{T, x})$ et $j_{y}(\oo_{T, y})$.

R\'eciproquement, le diagramme cart\'esien
\[
\xymatrix{T_{1} \ar[r] & T\times T\\
W \ar[u]^{\psi} \ar[r] & \s(\oo_{T, x}) \times \s(\oo_{T, y}) \ar[u]_{\varphi}}
\]
d\'efinit un sous-sch\'ema ferm\'e $W$ du sch\'ema affine  $\s(\oo_{T, x}) \times \s(\oo_{T, y})$ ; c'est donc un sch\'ema affine, et il est int\`egre puisque $\varphi$, et donc aussi $\psi$, sont des monomorphismes plats. L'anneau $P$ de $W$ est engendr\'e par les sous-anneaux  $M = j_{x}( \oo_{T, x})$ et $N = j_{y}(\oo_{T, y})$  de $K$. Puisque $M$ et $N$ sont suppose\'s apparent\'es il existe un id\'eal premier $\mathfrak{p}$ de $P$ tel que $P_{\mathfrak{p}}$ domine $M$ et $N$ ; soit $w \in W = \s(P)$ le point correspondant \`a $\mathfrak{p}$ ; en prenant $ z = \psi(w)$ on a $x = d_{0}(z)$  et $y = d_{1}(z)$.
\end{proof}

\n Du r\'esultat qui pr\'ec\`ede on déduit immédiatement le  crit\`ere de s\'eparation de \cite[8.5.5]{EGAG} :

\begin{cor}\phantomsection\label{c6.1} Pour qu'un sch\'ema int\`egre $T$ soit s\'epar\'e, il faut et il suffit que la relation \og $j_{x}(\oo_{x})$ et $j_{y}(\oo_{y})$ sont apparent\'es \fg \, entre points $x$ et $y$ de $T$ implique $x = y$.
\end{cor}

\begin{proof} En effet, le morphisme $T_{1} \rrr T\times T$ \'etant une immersion ferm\'ee, la diagonale de $T$ est ferm\'ee si et seulement si le morphisme $T \rightarrow T_{1}$ est surjectif ; il est alors un isomorphisme.
\end{proof}

\subsection{Platitude et apparentement}

\begin{cor}\phantomsection\label{c6.2} Soit $T$ un sch\'ema int\`egre. Gardons les notations de \ref{6.1}. Les propri\'et\'es suivantes sont \'equivalentes :
\begin{thlist}
\item Les morphismes $d_{0}, d_{1} : \xymatrix{T_{1}\ar@<0.5ex>[r] \ar@<-0.5ex>[r]  &T}$ sont plats.

\item Pour tout couple $x, y$ de points de $T$, si les anneaux locaux  $j_{x}(\oo_{x})$ et $j_{y}(\oo_{y})$ sont apparent\'es, alors ils sont \'egaux.
\end{thlist}
\end{cor}

\begin{proof} On a vu plus haut que pour $z \in T_{1}$, l'homomorphisme compos\'e
$$
\oo_{T, d_{0}(z)} \rrr \oo_{T_{1}, z} \rrr K
$$
est \'egal \`a l'injection canonique $j_{d_{0}(z)}$, et idem pour $j_{d_{1}(z)}$.

 Montrons que (i) implique (ii). D'apr\`es la proposition, si les anneaux locaux  $j_{x}(\oo_{x})$ et $j_{y}(\oo_{y})$ sont apparent\'es il existe $z \in T_{1}$ tel que  $x = d_{0}(z)$  et  $y = d_{1}(z)$ ; l'hypoth\`ese (i) implique que les homomorphismes locaux $\oo_{T, x} \rrr \oo_{T_{1},z}$ et $\oo_{T, y} \rrr \oo_{T_{1}, z}$ sont fid\`element plats ; comme ils sont birationnels, ce sont des isomorphismes, et les sous-anneaux $j_{x}(\oo_{x})$ et $j_{y}(\oo_{y})$ de $K$ sont \'egaux \`a l'image de $\oo_{T_{1}, z}$ dans $K$.

 Montrons que (ii) implique (i). Considérons un point $z \in T_{1}$, d'images $x$ et $y$ dans $T$ ; on va montrer que les homomorphismes $u : \oo_{T, x} \rrr \oo_{T_{1}, z}$ et $v : \oo_{T, y} \rrr \oo_{T_{1}, z}$ sont surjectifs, donc bijectifs puisqu'ils sont birationnels; cela entra\^{i}nera la platitude annonc\'ee. 

\n La proposition \ref{p6.1} dit que les sous-anneaux locaux $j_{x}(\oo_{x})$ et $j_{y}(\oo_{y})$ de $K$ sont apparentés, donc égaux d'après l'hypothèse (ii). Comme les injections $j_{x}$  et $j_{y}$ se factorisent respectivement par $u$ et par $v$, on voit que les images ${\rm Im}(u)$  et  ${\rm Im}(v)$ sont des sous-anneaux {\it \'egaux} de $\oo_{T_{1},z}$. Considérons de nouveau le sous-anneau $P \subset K$ engendré par  $j_{x}(\oo_{x})$ et $j_{y}(\oo_{y})$ ; puisque $T_{1}$ est un sous-schéma fermé de $T\times T$, l'anneau local $\oo_{T_{1},z}$ est un localisé de $P$, soit $P_{\mathfrak{p}} = \oo_{T_{1},z}$. Consid\'erons le diagramme suivant, o\`u $w$ est surjectif 
$$
\xymatrix{\oo_{T, x} \ar[d] \ar[dr]_{u'} \ar[drr]^u &&\\
\oo_{T, x}\otimes \oo_{T, y} \ar[r]^>>>>>w & P \ar[r] &P_{\mathfrak{p}}\\
\oo_{T, y} \ar[u] \ar[ur]^{v'} \ar[urr]_{v} &&}
$$
Comme $P$ est int\`egre, l'homomorphisme $P \rrr P_{\mathfrak{p}}$ est injectif ; d'o\`u l'\'egalit\'e ${\rm Im}(u') = {\rm Im}(v')$ ; mais $w$ est surjectif par définition de $P$ ;  on en déduit les égalités $P = {\rm Im}(w) = {\rm Im}(u') = {\rm Im}(v')$ ; par suite, $P$ est un anneau local, donc $P = P_{\mathfrak{p}}$, et cela permet de conclure que $u$ et $v$ sont surjectifs.
\end{proof}

\begin{prop}\phantomsection\label{p6.2} 
Soit $T$ un schéma intègre et quasi-séparé. Alors $T$ admet un séparateur si et seulement si les morphismes $d_{0}, d_{1} : \xymatrix{T_{1}\ar@<0.5ex>[r] \ar@<-0.5ex>[r]  &T}$ sont de type fini, et si pour tout couple $x, y$ de points de $T$, les propri\'et\'es b) et c) de \ref{6.1} sont équivalentes, autrement dit: les anneaux locaux  $j_{x}(\oo_{x})$ et $j_{y}(\oo_{y})$ sont \'egaux d\`es qu'ils sont apparent\'es.
\end{prop}

\begin{proof}
D'après le théorème \ref{p5.1}, l'existence d'un séparateur est équivalente à la propriété pour les morphismes $d_{i}$ d'\^{e}tre plats et de type fini, et d'après le corollaire \ref{c6.2} la condition portant sur les couples de points équivaut à la platitude des $d_{i}$. La proposition découle donc de ce qui précède.

\medskip

Mais on peut aussi montrer que les conditions de l'énoncé sont suffisantes sans utiliser le théorème \ref{p5.1}, en décrivant explicitement un séparateur $h : T \rrr E$; on verra que le passage de $T$ à $E$ consiste simplement à \emph{identifier les points $x$ et $y$ de $T$ lorsque les anneaux locaux $j_{x}(\oo_{x})$ et $j_{y}(\oo_{y})$ sont \'egaux}.

\n Notons d'abord que le corollaire \ref{c6.2} entra\^{i}ne que les morphismes $d_{i}$ sont plats (et de type fini) ; on peut donc invoquer la proposition \ref{l2.2} pour conclure que les $d_{i}$ sont des isomorphismes locaux.

\n On note encore par $K$ le corps des fonctions rationnelles sur $T$, et par $j_{t} : \oo_{t}\rrr R$ l'injection canonique dans $K$ de l'anneau local en un point de $T$.
\medskip

 On d\'efinit donc  $E$ comme l'ensemble des sous-anneaux locaux de $K$ de la forme $j_{t}(\oo_{t})$, pour $t$ parcourant $T$ ; on a donc une application surjective d'ensembles
 $$
 h : T \; \rrr \; E.
 $$
 Il r\'esulte de la proposition  \ref{p6.1} et du corollaire \ref{c6.1} que la restriction de l'application $h$ \`a un ouvert s\'epar\'e, p.ex. affine, est injective.
 \medskip
 
 On munit $E$ de la topologie engendr\'ee par les ensembles $h(U)$, pour $U$ parcourant l'ensemble des ouverts affines de $T$. Pour voir que l'application $h$ est continue il faut v\'erifier que, pour tout ouvert affine $U$ de $T$, 
l'ensemble $h^{-1}h(U)$ est ouvert; or, c'est l'ensemble des $x \in T$ pour lesquels il existe $y \in U$ tel que $h(x) = h(y)$, i.e. $j_{x}(\oo_{x}) = j_{y}(\oo_{y})$ ; par hypoth\`ese cette \'egalit\'e  est \'équivalente \`a : $j_{x}(\oo_{x})$ et $j_{y}(\oo_{y})$ sont apparent\'es; d'apr\`es la proposition \ref{p6.1},  cette derni\`ere relation \'equivaut \`a l'existence d'un \'el\'ement $z \in T_{1}$ tel que $x = d_{0}(z)$ et $y = d_{1}(z)$ ; bref, $h^{-1}h(U) = d_{0}d_{1}^{-1}(U)$ ; c'est bien un ouvert puisque $d_{0}$ est une application ouverte.
\medskip

  Finalement, on d\'efinit un faiceau $\oo_{E}$ d'anneaux sur $E$ en associant \`a un ouvert $V$ de $E$ l'intersection (dans $K$) des anneaux locaux \'el\'ements de $V$ ; mais, d'après \cite[8.5.1.1]{EGAG}, ou \cite[pp. 124--126]{God58}, cette intersection est aussi \'egale \`a 
$$ 
\bigcap_{x \in h^{-1}(V)} j_{x}(\oo_{T, x}) \; = \; \Gamma(h^{-1}(V), \oo_{T}) .
$$
En d'autres termes, l'homomorphisme de faisceaux sur $E$
$$
\oo_{E} \; \rrr \; h_{\star}(\oo_{T})
$$
est bijectif, et il aurait permis de d\'efinir plus simplement  $\oo_{E}$ comme cette image directe.
Il reste \`a voir que l'espace annel\'e $(E, \oo_{E})$ est un sch\'ema. Or, si $U$ est un 
ouvert affine de $T$, le morphisme $h$ induit un isomorphisme d'espaces annel\'es de $(U, \oo_{U})$ sur $(h(U), \oo_{j(U)})$.

Cela montre que $E$ est un schéma  et que $h$ est un isomorphisme local de schémas, quasi-compact et quasi-séparé ; $E$ est séparé en vertu du corollaire \ref{c6.1}.
\end{proof}

\subsection{Un exemple}

\begin{ex}\phantomsection\label{ex6.1} Il existe un sch\'ema int\`egre noeth\'erien $T$, de dimension 1, n'ayant que 2 points ferm\'es, et qui n'admet pas de s\'eparateur. Cependant, son morphisme  d'enveloppe affine \cite[9.1.21]{EGAG} 
$$
i_{T} : T \rrr T^{0} = \s(\Gamma(T))
$$
est universel pour les morphismes de $T$ vers un schéma séparé.
\medskip

Pour le construire, partons d'un anneau de valuation discr\`ete $P$ de corps des fractions $K$ ; notons $a \mapsto \bar{a}$ le morphisme vers le corps r\'esiduel  : $P \rrr P/\mathfrak{m} = k$.
On suppose qu'il existe deux sous-corps \emph{distincts}  $k_{1}$ et $k_{2}$ de $k$ tels que $[k : k_{1}] = [k : k_{2}] = 2.$
\medskip

\n Soit $A_{i} \subset P$, pour $i = 1, 2$,  les sous-anneaux form\'es des $a \in P$ tels que, respectivement,  $\bar{a} \in k_{1}$ et $\bar{a} \in k_{2}$ ; on v\'erifiera plus bas  les points suivants.

$a)$ La cl\^oture int\'egrale de $A_{i}$ est \'egale \`a $P$.

$b)$ L'homomorphisme $A_{1} \otimes A_{2} \; \rrr \; P$, d\'efini par  $a\otimes b \mapsto ab$ est surjectif.
\medskip

Supposons ces points acquis. On définit le sch\'ema $T$ par recollement des sch\'emas $U_{i} = \s(A_{i})$ le long de l'ouvert $U_{0} = \s(K)$. En notant $x_{1}$ et $x_{2}$ les deux points ferm\'es, on a $j_{x_{i}}(\oo_{x_{i}}) = A_{i}$ ; ces deux anneaux locaux sont apparent\'es (par $P$), et ils ne sont pas \'egaux. On en tire, via le corollaire  \ref{c6.2},  que les morphismes $d_{0}, d_{1} : \xymatrix{T_{1}\ar@<0.5ex>[r] \ar@<-0.5ex>[r]  &T}$ ne sont pas plats ; en particulier, le sch\'ema $T$ n'admet pas de s\'eparateur. 

\medskip

Vérifions que l'enveloppe affine de $T$ est aussi son enveloppe séparée. La suite exacte (dans la catégorie des schémas)
$$
\xymatrix{U_{0} \ar@<0.5ex>[r]^<<<<{u_{1}} \ar@<-0.5ex>[r]_<<<<{u_{2}} &U_{1} \sqcup U_{2} \ar[r]   &T}
$$
montre que $\Gamma(T) = A_{1}\cap A_{2} = \{b \in P \, | \, \bar{b} \in k_{1} \cap k_{2} \}$. Posons $V = \s(P)$ et désignons par $v_{i} : V \rrr U_{i}$ les morphismes locaux associés aux inclusions $A_{i} \subset P$.

Considérons un morphisme $u : T \rrr S$ vers un schéma séparé ; il faut montrer l'existence d'un unique morphisme $u^{0} : T^{0} \rrr S$ tel que $u = u^{0}\circ i_{T}$. Désignons par $u_{i} : U_{i} \rrr S$ les restrictions de $u$. Montrons que $u_{1}v_{1} = u_{2}v_{2}$, autrement dit que le carré suivant est commutatif. 
$$
\xymatrix{V \ar[r]^{v_{1}} \ar[d]_{v_{2}} & U_{1} \ar[d]^{u_{1}}\\
U_{2} \ar[r]_{u_{2}} & S} 
$$
Or, les morphismes  $u_{1}v_{1}$ et $u_{2}v_{2}$ sont, par hypothèse, égaux sur l'ouvert $U_{0}$ de $V$; puisque $S$ est séparé, ils sont égaux sur un fermé de $V$, et ce fermé contient $U_{0}$ : il est donc égal à $V$, et on a bien $u_{1}v_{1} = u_{2}v_{2}$. On en déduit, les morphismes $v_{i}$ étant locaux, que les points fermés $x_{i}$ de $U_{i}$ ont la m\^{e}me image $s \in S$, et que les morphismes $u_{i}$ se factorisent par $\s(\oo_{S,s}) \rrr S$ ; 
on a ainsi un carré commutatif d'homomorphismes d'anneaux
$$
\xymatrix{P  & A_{1} \ar[l]\\
A_{2} \ar[u] & \oo_{S,s} \ar[l] \ar[u]}
$$
Cette commutativité montre qu'il existe un unique homomorphisme d'anneaux 
$$
\oo_{S, s} \rrr A_{1}\cap A_{2} ,
$$
 d'o\`u l'on tire le morphisme composé   $u^{0} : T^{0} = \s(A_{1}\cap A_{2})  \rrr \s(\oo_{S,s}) \rrr S$, et on a   $u = u^{0}\circ i_{T}$.
\bb

Il reste donc \`a v\'erifier les propriétés $a)$ et $b)$. L'id\'eal maximal $\mathfrak{m}$ de $P$ est aussi un id\'eal maximal de chacun des anneaux $A_{i}$. Pour un \'el\'ement non nul $s \in \mathfrak{m}$ on a $sP \subset \mathfrak{m} \subset A_{i}$, donc les anneaux de fractions $(A_{i})_{s}$ et  $P_{s}$ sont \'egaux au corps des fractions  $K$. Les homomorphismes $A_{i}/\mathfrak{m} \rrr P/\mathfrak{m}$ s'identifient aux l'homomorphismes fini $k_{i} \rrr k$ ; donc $P$ est fini sur $A_{i}$ ; comme $P$ est int\'egralement clos, c'est la cl\^oture int\'egrale de $A_{i}$.

Pour vérifier la propri\'et\'e $b)$, choisissons  un \'el\'ement  $t \in A_{2}$ qui ne soit pas dans $A_{1}$, c'est-\`a-dire tel que $\bar{t} \in k_{2}$ et  $\bar{t} \notin k_{1}$. Puisque $k$ est de degr\'e 2 sur $k_{1}$, on a $k = k_{1}[\bar{t}]$, donc aussi $P = A_{1}[t]$ ; d'o\`u sa surjectivit\'e de l'homomorphisme $A_{1} \otimes A_{2} \; \rrr \; P$. 
\medskip

On n'a que l'embarras du choix pour exhiber des exemples d'anneaux de valuation discr\`ete munis des donn\'ees requises. Le plus simple est de prendre une extension galoisienne de corps, $k_{0} \subset k$ de groupe de Galois le groupe sym\'etrique $\mathfrak{S}_{3}$, et de choisir deux transpositions distinctes $\sigma$ et $\tau$ dans  $\mathfrak{S}_{3}$ ; les sous-corps d'invariants  $k_{1} = k^{\sigma}$ et $k_{2} = k^{\tau}$ conviennent ; enfin, on peut prendre pour $P$ l'anneau local \`a l'origine de $k[X]$.
 \qed
\medskip

Malgré sa propriété universelle, le morphisme $i_{T} : T \rrr T^{0}$ n'est pas un séparateur au sens donné au début du texte : ce n'est pas un isomorphisme local, et il n'est m\^{e}me pas plat.
\end{ex}

\section{Sur les schémas réunion de deux ouverts affines}\label{s7}

Dans ce paragraphe on examine l'existence d'un s\'eparateur pour les sch\'emas qui sont r\'eunion de deux ouverts affines ; comme il est indiqué en \ref{c5.2},  c'est la situation d\'ecisive. Commen\c cons par traduire le critère du théorème \ref{p5.1} pour ces schémas élémentaires.

 Lorsqu'un sch\'ema de base n'est pas mentionn\'e c'est qu'il s'agit $\s({\bf Z})$. Lorsque le contexte le permettra on \'ecrira encore $\Gamma(U)$ \`a la place de $\Gamma(U, \oo_{U})$.
 
\subsection{Principe des constructions}
 
\begin{scholie}\phantomsection\label{sc7.1}  Soit $T = U \cup V$ un sch\'ema r\'eunion de deux ouverts affines, dont l'intersection est quasi-compacte, de sorte que $T$  est un sch\'ema quasi-compact et quasi-s\'epar\'e ; on suppose aussi que l'ensemble des composantes irréductibes de $T$ est fini. Notons 
$$
\varphi_{UV} : \Gamma(U)\otimes\Gamma(V) = \Gamma(U\times V) \; \rrr \; \Gamma(U\cap V)
$$
l'homomorphisme d'anneaux induit par les restrictions à l'ouvert $U\cap V$.
Alors $T$ admet un séparateur si et seulement si  l'anneau ${\rm Im}(\varphi_{UV})$ est plat et de type fini sur $\Gamma(U)$ et sur $\Gamma(V)$.
\end{scholie}

\begin{proof} Notons, comme dans le théorème \ref{p5.1}, $T \rightarrow T_{1} \rightarrow T\times T$ la factorisation de la diagonale par adhérence schématique ; il s'agit de traduire, en termes de ${\rm Im}(\varphi_{UV})$, la propriété pour les deux projections de $T_{1}$ sur $T$ d'être plates et de type fini. Puisque $U$ et $V$ sont séparés, les restrictions de $T_{1}$ aux ouverts $U\times U$ et $V\times V$ de $T\times T$ sont égales à $U$ et $V$ respectivement, et ne sont donc pas pertinentes ; par contre, ce qui est significatif est la restriction de l'adhérence schématique à l'ouvert $U\times V$ ; on trouve
$$
U\cap V \stackrel{v}{\rrr} W \rrr U\times V
$$
où $W$ est un schéma affine puisque c'est un sous-schéma fermé de $U\times V$ ; par ailleurs, comme $v$ est \scd, l'homomorphisme
$\Gamma(W) \rrr \Gamma(U\cap V)$ est injectif ; cela montre que $\Gamma(W)$ n'est autre que l'anneau ${\rm Im}(\varphi_{UV})$ de l'énoncé. Par suite, en utilisant \cite[6.2.5  et 6.3.1]{EGAG}, on voit que les morphismes de projection $ U \longleftarrow W \rrr V$ sont plats et de type fini si et seulement si les homomorphismes
$\Gamma(U) \rrr {\rm Im}(\varphi_{UV}) \longleftarrow \Gamma(V)$ le sont.
\end{proof}

\begin{para}\phantomsection\label{s7.2} Nous donnons ci-dessous deux exemples de sch\'emas qui n'admettent pas de s\'eparateur, \`a savoir :
\begin{itemize}
\item un sch\'ema r\'egulier int\`egre de dimension 2 (exemple \ref{ex7.1}) ;
\item  un sch\'ema {\it \'etale}  sur une base affine de dimension 1 (exemple \ref{ex7.2}).
\end{itemize}

\n  Ces deux constructions utilisent le m\^eme proc\'ed\'e. Partant  d'un sch\'ema affine $U$, d'une immersion  ouverte  $i : U_{0} \rrr U$ et d'un automorphisme $\tau : U_{0} \simeq U_{0}$, le sch\'ema  annoncé sera la somme amalgam\'ee
$$
T = (U, i) \sqcup_{U_{0}}(U, i\circ \tau)
$$
 obtenue par r\'eunion de deux copies de $U$, recoll\'ees le long de $U_{0}$, l'une munie de l'injection canonique $i : U_{0} \rrr U$, et l'autre munie de l'injection, \emph{tordue par} $\tau$,\, $i\circ \tau : U_{0} \rrr U$. Autrement dit, la suite
$$
\xymatrix{U_{0} \ar@<0.5ex>[r]^<<<<<{i} \ar@<-0.5ex>[r]_<<<<<{i\circ \tau} & U\sqcup U \ar[r] &T}
$$
est exacte dans la cat\'egorie des sch\'emas.
\medskip  


 On notera que dans chacun de ces exemples, le sch\'ema affine  $U$ est noeth\'erien r\'eduit et que l'ouvert de recollement $U_{0}$ est lui aussi affine, et sch\'ematiquement dense dans $U$. 
 
\end{para}

\subsection{Un contre-exemple lisse}

\begin{ex}\phantomsection\label{ex7.1} Il existe un morphisme $T \rrr S$  lisse, quasi-compact et qui n'admet pas de s\'eparateur.
 
 \n On peut choisir $S$ affine r\'egulier de dimension 1, $T$ est alors r\'egulier de dimension 2.
\end{ex} 

Soit $S = \s(R)$ le spectre d'un anneau noeth\'erien int\`egre, poss\'edant une suite r\'eguli\`ere $(s, t)$ form\'ee d'\'el\'ements non inversibles (par exemple $R = k[X]$, o\`u $k$ est un corps, avec la suite $(X, 1-X)$). Posons $A = R[Z]$, $U = \s(A)$ et $U_{0} = \s(A_{st})$ ; l'automorphisme $\tau$ de $U_{0}$ est ici associé à l'automorphisme $ \sigma $ de la $R$-algèbre $A_{st} = R_{st}[Z]$ défini par 
$$
Z \longmapsto sZ / t
$$
 Finalement, soit $T$ le schéma défini par recollement à partir des données $(U_{0} \subset U, \tau)$ ; il est donc caractérisé par l'exactitude de la suite
$$
\xymatrix{U_{0} \ar@<0.5ex>[r]^<<<<<{i} \ar@<-0.5ex>[r]_<<<<<{i\circ \tau} & U\sqcup U \ar[r] &T}
$$
Pour appliquer le critère \ref{sc7.1}, il faut expliciter l'anneau  $C \subset R_{st}[Z]$ image du morphisme de $R$-alg\`ebres 
$$
\varphi : \Gamma(U)\otimes_{R}\Gamma(U) = A\otimes_{R}A \; \rrr \; A_{st} = \Gamma(U_{0}), \hspace{1cm} a \otimes b \longmapsto \sigma(a)b .
$$

On utilisera le 
\medskip

\begin{lemme}\phantomsection\label{l7.1} Soit $(s, t)$ une suite r\'eguli\`ere dans un anneau $A$. Alors,  le morphisme de $A$-alg\`ebres 
$$
A[T]/(sT-t) \rrr A_{s}, \qquad  T \longmapsto t/s 
$$
 est injectif. Son image $A[t/s]$ est plate sur $A$ si et seulement si l'id\'eal $sA+tA$ est \'egal \`a $A$.
 \end{lemme}

\begin{proof} Soit $F(T) = a_{n}T^n + a_{n-1}T^{n-1} + \cdots + a_{0} \in A[T]$ un polyn\^ome tel que, dans $A_{s}$, on ait $F(t/s) = 0$, c'est-\`a-dire 
$ a_{n}t^n + sa_{n-1}t^{n-1} + \cdots + s^na_{0} = 0$. Comme $s$ est r\'egulier, le morphisme $A \rrr A_{s}$ est injectif, donc l'\'egalit\'e pr\'ecedente est d\'ej\`a vraie dans $A$. Comme la suite $(s, t)$ est r\'eguli\`ere, on voit que $a_{n} = sb$, avec $b \in A$, d'o\`u l'on tire que 
$F(T) = bT^{n-1}(sT-t) + G(T)$ avec $\deg(G) < n$. Cela permet de raisonner par r\'ecurrence sur le degr\'e de $F$ pour conclure que ce polyn\^ome est un multiple de $sT-t$. 

\n Montrons la seconde assertion. Posons $I = sA+tA$ et $J = (sT-t)A[T]$, de sorte que $J \subset IA[T]$. Si le quotient $A[T]/J$ est plat sur $A$, l'application
$J/IJ \rrr (A/I)[T]$ est injective, et nulle ; on en d\'eduit  que  $J = IJ$.En notant $P(T) = sT-t$, la relation $J = IJ$ s'\'ecrit aussi dans $A[T]$,
$P = PQ$, o\`u $Q$ est un polyn\^ome \`a coefficients dans $I$. Le polyn\^ome $P$ est r\'egulier dans $A[T]$ puisqu'il l'est dans $A_{s}[T]$, et que $A \rrr A_{s}$ est injectif.   On en tire que $1 = Q$, soit $1 = Q(0) \in I$. 

R\'eciproquement, s'il existe $a, b \in A$ tels que $1 = as +bt $, alors, dans $A_{s}$, on a $1/s = a + b(t/s)$, donc $A_{s} = A[t/s]$.
\end{proof}

\n Pour expliciter l'anneau $C = {\rm Im}(\varphi)$, on identifie $A\otimes_{R}A $ et $R[Z_{0}, Z_{1}]$, ce qui permet d'écrire  l'homomorphisme $\varphi$ sous la forme
$$
R[Z_{0}, Z_{1}] \rrr R_{st}[Z], \hspace{1cm} Z_{0} \longmapsto sZ/t, \quad Z_{1} \longmapsto Z
$$
Il est alors clair qu'il se factorise par le quotient
$$
R[Z_{0}, Z_{1}]/(tZ_{0}-sZ_{1}) \rrr R_{st}[Z]
$$
Le lemme ci-dessus, appliqué en rempla\c cant $(A, s, t)$ par $(R[Z_{1}], t, sZ_{1})$, montre que cet homomorphisme est injectif, et qu'on peut donc identifier l'anneau $C$ au quotient  $R[Z_{0}, Z_{1}]/(tZ_{0}-sZ_{1})$. 
Les homomorphismes $\varphi_{0} : A \rrr C$  et $\varphi_{1} : A \rrr C$, dont il faut montrer qu'ils ne sont pas plats, s'identifient alors aux homomorphismes ``naturels'' 
$$
\varphi_{0} : R[Z_{0}] \rrr R[Z_{0}, Z_{1}]/(tZ_{0}-sZ_{1}), \hspace{1cm} \varphi_{1} : R[Z_{1}] \rrr R[Z_{0}, Z_{1}]/(tZ_{0}-sZ_{1})
$$

Ils ne sont pas plats d'après le lemme \ref{l7.1} appliqué en rempla\c cant $(A, s, t)$ respectivement par $(R[Z_{0}], s, tZ_{0})$ et $(R[Z_{1}], t, sZ_{1})$.

Ceci conclut la construction de l'exemple \ref{ex7.1}.\qed

\begin{rque}\phantomsection Déterminons l'enveloppe affine $i_{T}: T \rrr T^{0} = \s(\Gamma(T))$  de $T$ \cite[9.1.21]{EGAG}, en introduisant les morphismes
$$
\xymatrix{U_{0} \ar[r]^{i} \ar[d]_{i\tau} & U \ar[d]^{u} \ar[ddr]^{u^{0}}&\\
U\ar[r]_{v} \ar[drr]_{v^{0}} & T\ar[dr]^<<<{i_{T}}&\\
&&T^{0}}
$$
On trouve que $T^{0} = \s(R[Z])$, les morphismes $u^{0}$ et $v^{0}$ étant associés respectivement aux homomorphismes $Z \mapsto tZ$  et $Z \mapsto sZ$. Introduisons les ouverts $S' \subset S'' \subset S = \s(R)$ définis par $S' = D(s)\cap D(t)$ et $S'' = D(s) \cup D(t)$. 

Le critère \ref{sc7.1} montre qu'au dessus de l'ouvert $S'$, c'est-à-dire si $s$ et $t$ sont inversibles,  la restriction  
$$
1\times i_{T} : S'\times_{S}T \rrr S'\times_{S}T^{0}
$$
 est un séparateur puisque les homomorphismes $\varphi_{0}$  et $\varphi_{1}$ sont alors des isomorphismes.
 
 Au dessus de $S''$, c'est-à-dire si $s$ ou $t$ est inversible, alors  $v^{0}$ ou $u^{0}$ est un isomorphisme, et 
 $$
1\times i_{T} : S''\times_{S}T \rrr S''\times_{S}T^{0}
$$
 est une enveloppe séparée (proposition \ref{p3.2}).
\end{rque}

\subsection{Un contre-exemple étale}

\begin{ex}\phantomsection\label{ex7.2} Il existe un sch\'ema $S$ local noeth\'erien int\`egre de dimension 1, et un morphisme \emph{\'etale} 
$f : T \rrr S$, tels que $T$ n'admette pas de s\'eparateur.
\end{ex}

Le sch\'ema $S$ introduit ici n'est pas normal. D'ailleurs on a vu (corollaire \ref{c5.1}) que si $T$ est \'etale sur un sch\'ema affine int\`egre \emph{et  normal}, il admet un s\'eparateur.

La construction \'etant un peu technique, en voici d'abord un r\'esum\'e. On commence par d\'efinir un rev\^etement \'etale galoisien {\it connexe} de rang 3, $U \rrr S$ dont la fibre g\'en\'erique $U_{\eta}$ est constitu\'ee de 3 points isomorphes au point générique  $\eta$ de $S$ ; cette fibre admet donc un $S$-automorphisme $\tau$ d'ordre 2. On d\'efinit $T$ par recollement à partir des données $(U_{\eta} \subset U, \tau)$, c'est-à-dire le recollement de deux copies de $U$ le long des immersions ouvertes $i : U_{\eta} \rrr U$ et $i\circ \tau : U_{\eta} \rrr U$ ; on  a donc une suite exacte
$$
\xymatrix{U_{\eta} \ar@<0.5ex>[r]^<<<<<{i} \ar@<-0.5ex>[r]_<<<<<{i\circ \tau} & U\sqcup U \ar[r] &T} .
$$  
Notons $A$ l'anneau (local) de $S$, $K$ son corps des fractions, et $\bar{A}$ le normalis\'e de $A$. On v\'erifiera que l'image du morphisme 
$$
\varphi : \Gamma(U) \otimes\Gamma(U) \rrr \Gamma(U_{\eta}) = K\times K \times K
$$
est égal au  sous-anneau $ \bar{A}\times \bar{A}\times \bar{A}$ ; il est donc entier sur $A$. Comme $A$ n'est pas normal, $\bar{A}$ n'est pas plat sur $A$, et {\it a fortiori} ${\rm Im}(\varphi) $ n'est pas plat sur $\Gamma(U)$, qui est étale sur $A$.

On reconna\^itra en $S$  (le localisé à l'origine de) la cubique \`a point double et en $U$  le recollement de trois copies de sa normalis\'ee (une droite) obtenu en identifiant de fa\c con circulaire les deux points 
ferm\'es au dessus de l'origine, soient $\{a_{0}, a_{1}\} ,\, \{b_{0}, b_{1}\}, \, \{c_{0}, c_{1}\}$, identifi\'es via $a_{0}= b_{1}, \; b_{0}= c_{1}, \; c_{0}= a_{1}$.
\bb

{\footnotesize \n Voici les d\'etails.
\medskip

\n Soient $k$ un corps, et $\bar{A}$ l'anneau semi-local obtenu en localisant $k[X]$ en les points correspondants \`a $X = -1$ et $X = 1$. Soit $I \subset \bar{A}$ l'id\'eal engendr\'e par $X^2-1$, c'est donc le noyau d'un morphisme surjectif
$$
p : \bar{A} \; \rrr\;  k\times k.
$$

\n Soit $A = k + I  \subset \bar{A}$ le sous-anneau obtenu en \emph{pin\c cant} les deux points\footnote{Si, pour mieux voir la cubique \`a point double, on pr\'ef\`ere les \'equations, on d\'efinit $A$ comme le localis\'e \`a l'origine de l'anneau $k[U, V]/(U^3+U^2-V^2)$, le morphisme de cet anneau vers $k[X]$ \'etant d\'efini par $U \mapsto X^2-1,\, V \mapsto X^3-X$} ; un \'el\'ement $a \in \bar{A}$ est donc dans $A$ si son image $p(a) = (a_{0}, a_{1}) \in k\times k$ a ses deux composantes \'egales.
\medskip

\n On d\'efinit une $A$-alg\`ebre finie \'etale de rang 3, \, $B \subset \bar{A}^3$  par la condition suivante portant sur les triplets $(a, b, c) \in \bar{A}^3$ :
\begin{equation}\label{eq7.1}
(a, b, c) \; \in \; B \quad \Longleftrightarrow \quad a_{0}= b_{1}, \; b_{0}= c_{1}, \; c_{0}= a_{1}
\end{equation}
On peut voir $B$ comme le produit cart\'esien dans le diagramme suivant.
$$
\xymatrix{B \ar[dd]_{q} \ar[r] & \bar{A}\times \bar{A}\times\bar{A} \ar[d]^{p\times p \times p}\\
& k^2\times k^2 \times k^2 \ar[d]^{\theta} \\
k^3 \ar[r]_>>>>>{\delta \times \delta \times \delta} &k^2\times k^2 \times k^2 }
$$
Dans cette figure, $\delta : k \rrr k^2$ d\'esigne l'injection diagonale, et $\theta$ d\'esigne l'automorphisme d'anneaux d\'efini par
$$
\theta((a_{0},a_{1}), \, (b_{0}, b_{1}), \,  (c_{0}, c_{1})) = ((a_{0},b_{1}), \, (b_{0}, c_{1}), \,  (c_{0}, a_{1}))
$$
Le sch\'ema $U = \s(B)$ ressemble \`a ceci, o\`u les symboles $\bullet$ d\'esignent les points ferm\'es, tandis que $\circ$ d\'esignent les points g\'en\'eriques.
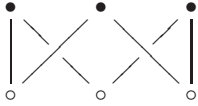
$$
\xymatrix{\bullet \ar@{-}[d] \ar@{-}[dr]|\hole &\bullet \ar@{-}[dl] \ar@{-}[dr] & \bullet \ar@{-}[d] \ar@{-}[dl]|\hole\\
\circ &\circ &\circ}
$$
Chacune des trois composantes irr\'eductibles est isomorphe \`a $\s(\bar{A})$ ; le recollement est indiqu\'e par l'isomorphisme $\theta$, c'est-\`a-dire par les \'egalit\'es \eqref{eq7.1}.

\n Pour voir que le morphisme $A \rrr B$ est fini \'etale de rang 3, on peut éviter une vérification directe en utilisant \cite[2.2  et 5.6]{Fer03}.
\medskip

 Notons $K = k(X)$ le corps des fractions de $A$, de sorte que $\eta = \s(K)$ et que $U_{\eta} = \s(K^3)$.
 Soit $\tau$ l'automorphisme de $U_{\eta}$ qui permute les deux premiers facteurs de $K^3$.
  Comme annonc\'e au d\'ebut, on d\'efinit le sch\'ema $T$ comme la r\'eunion de deux copies de $U$ recoll\'ees par les immersions $i : U_{\eta} \rrr U$  et $i \circ \tau : U_{\eta} \rrr U$. Explicitons le morphisme 
  $$
  \varphi : \Gamma(U)\otimes \Gamma(U) = B \otimes B \; \rrr \; \Gamma(U_{\eta}) = K^3.
  $$
Comme $B$ est un sous-anneau de $\bar{A}^3$, on peut repr\'esenter chacun de ses \'el\'ements par un triplet $(a, b, c)$ d'\'el\'ements de $\bar{A}$ (soumis à certaines conditions) ; avec cette convention, l'homomorphisme $\varphi$ s'\'ecrit
\begin{equation}\label{eq7.2}
\varphi : B\otimes B \rrr K^3, \qquad  (a, b, c) \otimes (a', b', c') \longmapsto (ab', ba', cc')
\end{equation}
On va v\'erifier que le sous-anneau $ C = {\rm Im}(\varphi) \subset K^3$ est \'egal \`a $\bar{A}^3$, ce qui montrera qu'il n'est pas plat sur $B = \Gamma(U)$, et que donc le sch\'ema $T$ n'admet pas de s\'eparateur. 

En effet, suivant l'argument d\'ej\`a \'evoqu\'e, si l'homomorphisme $B \rrr \bar{A}^3$ \'etait plat, alors $\bar{A}$ serait plat sur $A$ puisque $B$ est étale sur $A$, donc $A$ serait normal, une contradiction.
\medskip

Comme le carr\'e ci-dessus qui d\'efinit $B$ est cart\'esien, le noyau de $q$ est \'egal au noyau de l'homomorphisme vertical de droite, soit l'id\'eal
$I^{\times 3} = I \times I \times I$ de l'anneau $\bar{A}^3$.

Puisque $B/I^{\times3}$ est isomorphe \`a  $k^3$, un \'el\'ement  de $B$ est d\'etermin\'e, modulo $I^{\times 3}$, par son image dans $k^3$, c'est-\`a-dire par un triplet $(\alpha, \beta, \gamma) \in k^3$ ; si cet \'el\'ement de $B$, vu dans $\bar{A}^3$, s'\'ecrit $(a, b, c)$,  compte-tenu de $\theta$, on a, dans $k^2\times k^2 \times k^2$, la relation
\begin{equation}\label{eq7.3}
(p(a), p(b), p(c)) \; = \; ((\alpha, \gamma),\, (\beta, \alpha), \, (\gamma, \beta))
\end{equation}
Notons, enfin, que l'automorphisme $\rho : \bar{A}^3 \rrr \bar{A}^3$ de permutation circulaire : $\rho(a, b, c) = (c, a, b)$, est compatible avec $\theta$ et induit donc un automorphisme de $B$ qui, modulo $I^{\times 3}$, s'\'ecrit 
$$
(\alpha, \beta, \gamma) \longmapsto (\gamma, \alpha, \beta).
$$
 Ceci \'etant pr\'ecis\'e, montrons que $C$ contient les idempotents $(1, 0, 0), (0, 1, 0)$ et $(0, 0, 1)$.
\medskip

\n Soit $x$ un \'el\'ement de $B$ tel que $q(x) = (1, 0, 1)$ ; \'ecrit dans $\bar{A}^3$ sous la forme $(a, b, c)$ on a $(p(a), p(b), p(c)) = ((1, 1),\, (0, 1),\, (1, 0))$ ; pour $\rho(x)$, on a $((1, 0),\, (1, 1), \, (0, 1))$ ; par suite \eqref{eq7.2}
$$
\varphi(x\otimes \rho(x)) \equiv ((1,1), \, (0, 0), \, (0, 0)) \quad {\rm mod.} I^{\times3}
$$
cela montre que l'idempotent $(1, 0, 0)$ est bien dans $C$.  De m\^eme, on trouve
$$
\varphi(\rho(x) \otimes x) \equiv ((0, 0), \, (1, 1), \, (0, 0)) \quad {\rm mod.} I^{\times 3}
$$
$$
\varphi(\rho^2(x)\otimes \rho^2(x)) \equiv ((0,0),\, (0, 0),\, (1, 1)) \quad {\rm mod.} I^{\times3}.
$$
Ainsi $C$ est-il l'anneau produit de ses images dans chaque facteur de $K^3$, ce qu'on note $C = C_{a} \times C_{b} \times C_{c}$.
\medskip

Choisissons un \'el\'ement $t \in \bar{A}$ tel que $p(t) =(1, 0)$, donc tel que  $\bar{A} = A[t]$, puisque $(1, 0) \in k^2$ est un g\'en\'erateur de cette $k$-alg\`ebre. Montrons que l'\'el\'ement $(t, 0, 0) \in \bar{A}^3$ est dans $C$. Soit $y \in B$ un \'el\'ement tel que $q(y) = (1, 0, 0) \in k^3$ ; \'ecrit dans $\bar{A}^3$ sous la forme $y = (a, b, c)$, on a $(p(a), p(b), p(c)) = ((1, 0),\, (0, 1), \, (0, 0))$ \eqref{eq7.3}, et on trouve 
$$
\varphi(y\otimes \rho(y)) = ((1, 0), \,(0, 0),\, (0, 0))
$$
cela montre que $(t, 0, 0)$ est dans $C$ mod. $I^{\times3}$, donc que $t \in C_{a}$, d'o\`u l'\'egalit\'e $\bar{A} = A[t] = C_{a}$. De m\^eme, le calcul de $\varphi(y\otimes \rho^2(y))$ montre que l'on a $t \in C_{b}$, et le calcul de $\varphi(\rho^2(y)\otimes \rho^2(y))$ permet de conclure que $t \in C_{c}$. 
Cela ach\`eve la v\'erification de l'\'egalit\'e $C = {\rm Im}(\varphi) = \bar{A}^3$, et ainsi la d\'emonstration de ce que $T$ n'admet pas de s\'eparateur.\qed

}


\section{Cas des espaces algébriques}\label{s8}

\subsection{S\'eparation d'espaces alg\'ebriques}

\begin{lemme}\phantomsection\label{l8.1} Soient $S$ un sch\'ema, $X$ un $S$-espace alg\'ebrique, et  $U \rrr X$ une pr\'esentation \'etale par un $S$-sch\'ema ; enfin, notons $R = U\times_{X}U $ la relation d'\'equivalence qui d\'efinit $X$ comme quotient $U/R$. On suppose que le morphisme canonique (de sch\'emas) $\delta : R \rrr U\times_{S}U$ est une immersion quasi-compacte, i.e. que $X$ est ``localement s\'epar\'e'' avec les d\'efinitions de \cite[p. 97]{Knu71}. Consid\'erons la factorisation de $\delta$ par adh\'erence sch\'ematique 
$$
R \stackrel{v}{\rrr} \overline{R} \stackrel{\bar{\delta}}{\rrr} U\times_{S}U
$$
de sorte que $v$ est sch\'ematiquement dominant et que  $\bar{\delta}$ est une immersion ferm\'ee.

\n i) Supposons que  les morphismes compos\'es $\xymatrix{\overline{R} \ar@<0.5ex>[r] \ar@<-0.5ex>[r]& U}$ soient \'etales.
Alors $\overline{R}$ est  une relation d'\'equivalence \'etale sur $U$, l'espace alg\'ebrique quotient $Y = U/\overline{R}$ est s\'epar\'e, et le morphisme canonique $h : X \rrr Y$ est \'etale surjectif.

\n ii) Supposons de plus que $U \rrr S$ soit plat quasi-compact et quasi-s\'epar\'e, et qu'il existe un morphisme plat  $S' \rrr S$ tel que l'espace alg\'ebrique $S'\times_{S}X$ soit  s\'epar\'e sur $S'$. Alors $h$ induit un isomorphisme $S'\times_{S}X\, \wt{\rrr}\, S'\times_{S}Y$.
\end{lemme}

\begin{proof} L'\'enonc\'e $i)$ est une cons\'equence imm\'ediate de la proposition \ref{pB.1} appliqu\'ee \`a la relation $T_{\star} = R$. Rappelons que la s\'eparation sur $S$ de $U/\overline{R}$ est synonyme de ce que le morphisme $\bar{\delta} : \overline{R} \rrr U\times_{S}U$ soit une immersion ferm\'ee.

\n $ii)$  La s\'eparation du $S'$-espace alg\'ebrique $S'\times_{S}X$ signifie que le morphisme 
$$
S'\times_{S}R \rrr S'\times_{S}(U\times_{S}U)
$$
 est une immersion ferm\'ee ; puisque $U\times_{S}U$ est plat quasi-compact et quasi-s\'epar\'e sur $S$, la formation de l'adh\'erence sch\'ematique commute au changement de base plat $S' \rrr S$ (lemme \ref{lA.7}); donc  $S'\times_{S}R \rrr S'\times_{S}\overline{R}$ est un isomorphisme ; par suite, le morphisme $S'\times_{S}X\rrr S'\times_{S}Y$ est bien un isomorphisme.
 \end{proof}

\begin{prop}\phantomsection\label{p8.1} Soit $S$ un sch\'ema  normal int\`egre, de point g\'en\'erique $\eta$. Soit $X$ un $S$-espace alg\'ebrique 
\'etale et quasi-compact. Alors il existe un $S$-\emph{sch\'ema} \'etale et s\'epar\'e $Y$, et un morphisme $h : X \rrr Y$ qui est \'etale surjectif et qui induit un isomorphisme sur les fibres g\'en\'eriques : $
X_{\eta}\, \rt \, Y_{\eta}$. On notera $Y = X^{\rm sep}$.
\end{prop}

\begin{proof} Choisissons une pr\'esentation \'etale $U \rrr X$ par un sch\'ema $U$, \'etale et quasi compact  sur $S$. Le sch\'ema $U\times_{S}U$, \'etant \'etale sur $S$, est normal. Puisque $X$ est \'etale, le morphisme diagonal $X \rrr X\times_{S}X$, qui est repr\'esentable, est une immersion ouverte ; par produit fibr\'e on en d\'eduit que le morphisme
$R = U\times_{X}U \rrr U\times_{S}U$ est aussi une immersion ouverte ; soit $R \rrr \overline{R} \rrr U\times_{S}U$ sa factorisation par adh\'erence sch\'ematique. Comme $U$ est \'etale et quasi-compact sur $S$ l'ensemble des points maximaux de $U\times_{S}U$  s'identifie à $(U\times_{S}U)_{\eta}$, et est par suite fini. Le lemme \ref{lA.11} permet alors de conclure que  $\overline{R} \rrr U\times_{S}U$ est une immersion ouverte et ferm\'ee ; a fortiori, les morphismes $\xymatrix{\overline{R} \ar@<0.5ex>[r] \ar@<-0.5ex>[r]& U}$ sont \'etales, et on peut appliquer le lemme  \ref{l8.1}; il montre que le faisceau quotient $U/\overline{R}$ est un espace alg\'ebrique
 qui est \'etale et s\'epar\'e ; mais alors c'est un sch\'ema, d'apr\`es \cite[II 6.17]{Knu71} ou \cite[Th. A2, p.198]{LMB00}. 
 
 Enfin, la fibre $X_{\eta} \rrr \eta$ est un espace alg\'ebrique \'etale et quasi-compact sur un corps ; c'est donc un sch\'ema affine discret donc s\'epar\'e. La deuxi\`eme partie du lemme ci-dessus, appliqu\'ee au morphisme plat $\eta \rrr S$, montre que le morphisme $X_{\eta} \rrr Y_{\eta}$ est un isomorphisme.
\end{proof}

\section{Un adjoint \`a gauche}\label{s9}

La cat\'egorie des $S$-sch\'emas \'etales de pr\'esentation finie est une sous-cat\'egorie pleine de la cat\'egorie des $S$-sch\'emas plats de pr\'esentation finie. Ce paragraphe aborde la question suivante : cette inclusion de cat\'egories admet-elle un adjoint \`a gauche ?  Autrement dit, peut-on associer fonctoriellement \`a tout tel $S$-sch\'ema plat $T$ un sch\'ema $E$, \'etale et de pr\'esentation finie sur $S$, muni d'un $S$-morphisme $h : T \rrr E$ qui soit universel pour les morphismes de $T$  vers un $S$-sch\'ema \'etale du m\^eme type ?

Puisque nous suivons les conventions de \cite{EGAG},  le terme ``de pr\'esentation finie'' signifiera, comme dans cet ouvrage : \emph{localement de pr\'esentation finie, quasi-compact et quasi-s\'epar\'e} \cite[6.3.7]{EGAG}.

\subsection{Une première propriété universelle} Commen\c cons par rappeler deux r\'esultats classiques d'existence de cet adjoint.

\begin{prop}\phantomsection Soient $S$ un sch\'ema noeth\'erien et $f : T \rrr S$ un morphisme propre et lisse. Dans la factorisation de Stein
$$
T \; \stackrel{h}{\rrr} \; {\rm Spec}(f_{\star}(\oo_{T})) \; \stackrel{g}{\rrr} \; S
$$
le morphisme $g$ est \'etale fini, et les fibres de $h$ sont g\'eom\'etriquement connexes.
\end{prop}

Pour une d\'emonstration voir \cite[4.3.4, 7.8.7 et 7.8.10]{EGAIII}, ou bien \cite[X 1.2]{SGA1}, ou enfin \cite[8.5.16]{FAG}. La propri\'et\'e universelle de $h$ ne semble pas explicit\'ee dans ces r\'ef\'erences, mais c'est en tout cas  une cons\'equence de la proposition \ref{p9.1} ci-dessous. 

\begin{prop}\phantomsection\label{p9.2} Soient $S = \s(k)$ le spectre d'un corps et $f : T \rrr S$ un morphisme localement de type fini. Alors, il existe un $S$-sch\'ema \'etale, not\'e $\pi_{0}(T/S)$, et une factorisation de $f$ en
$$
T \; \stackrel{h}{\rrr} \; \pi_{0}(T/S) \; \stackrel{g}{\rrr} \; S.
$$
Le morphisme $h$ est universel pour les $S$-morphismes de $T$ vers un $S$-sch\'ema \'etale, et il reste universel apr\`es toute extension de corps $k \rrr k'$ ; autrement \'ecrit, si $S' \rrr S$ d\'esigne le morphisme de sch\'emas associ\'e \`a cette extension, on a un isomorphisme
$$
S'\times_{S}\pi_{0}(T/S) \quad \wt{\rrr}\quad \pi_{0}(S'\times_{S}T / S').
$$
Si $T'\rrr S$ est un second $S$-sch\'ema localement de type fini, on a un isomorphisme 
$$
\pi_{0}(T\times_{S}T') \; \rrr \; \pi_{0}(T/S)\times_{S}\pi_{0}(T'/S).
$$
 Enfin, les fibres de $h$ sont g\'eom\'etriquement connexes.
 \end{prop}

 Voir, par exemple,  \cite[I, \S4, n$^o$ 6, p.122-126]{DG70}.
 \medskip
 
 Dans ces deux cas, les fibres du morphisme $h$ sont g\'eom\'etriquement con\-nexes. En fait cette propri\'et\'e est caract\'eristique de l'adjonction, en vertu du r\'esultat suivant. 

\begin{prop}\phantomsection\label{p9.1} Soient $f : T \rrr S$ un morphisme plat de pr\'esentation finie et $T \stackrel{h}{\rrr} E \stackrel{g}{\rrr} S$ une factorisation de $f$, o\`u $g$ est \'etale de pr\'esentation finie. Consid\'erons les propri\'et\'es suivantes :
\begin{thlist}
\item Le morphisme $h$ est surjectif et ses fibres sont g\'eom\'etriquement con\-nexes.

\item Le morphisme $h$ est universel pour les $S$-morphismes de $T$ vers un $S$-sch\'ema \'etale et de pr\'esentation finie, et il reste universel apr\`es tout changement de base $S' \rrr S$.

\item Comme {\rm (ii)}, mais en se restreignant aux changements de base \'etales de pr\'esentation finie.
\end{thlist}
Alors on a les implications {\rm (i)}$ \Rightarrow  {\rm (ii)} \Rightarrow {\rm (iii)}$, et si $S$ est un sch\'ema noeth\'erien ces trois propri\'et\'es sont \'equivalentes.
\end{prop}

\begin{proof} Rappelons  \cite[7.3.10]{EGAG} qu'un morphisme plat et de pr\'esentation finie est  universellement ouvert.
\medskip

\n ${\rm (i)} \Rightarrow {\rm (ii)}$\, Supposons que $h$ soit surjectif \`a fibres g\'eom\'etriquement connexes. Ces hypoth\`eses \'etant stables par changement de base, il suffit  de montrer que $h$ est universel au-dessus de $S$. Soit donc $h' : T \rrr E'$ un $S$-morphisme vers un \'etale de pr\'esentation finie ; il faut exhiber un $S$-morphisme $E \rrr E'$ compatible avec $h$ et $h'$. Notons d'abord que le morphisme $(h, h') : T \rrr E\times_{S}E'$ est ouvert, puisque le morphisme $T\rrr S$ est universellement ouvert, et que $E\times_{S}E'$ est \'etale sur $S$, donc \`a diagonale ouverte (cf. \ref{sA.1}); notons $F \subset E\times_{S}E'$ l'ouvert image de $(h, h')$, et  $h'' : T \rrr F$ le morphisme surjectif  d\'eduit de $(h, h')$. Le morphisme $h$ se factorise donc en
$$
T \; \stackrel{h''}{\rrr} \; F \; \stackrel{u}{\rrr} \; E ,
$$
o\`u $u$ désigne la restriction à $F$ de la première projection. Il s'agit de v\'erifier que $u$ est un isomorphisme. Comme le morphisme $h = uh''$ est surjectif par hypoth\`ese, $u$ est surjectif ; puisque $F$ est un ouvert du sch\'ema $E\times_{S}E'$ qui est \'etale sur $E$, le morphisme $u$ est \'etale. Montrons  que les fibres de $u$ sont des isomorphismes : les morphismes $h''$ et $u$ induisent des morphismes (d\'esign\'es par les m\^emes lettres)
sur les fibres en $x \in E$ (o\`u les symboles $h^{-1}(x)$ et $u^{-1}(x)$ et $x$ d\'esignent les sch\'emas obtenus par le changement de base $\s(\kappa(x)) \rightarrow E$)
$$
h^{-1}(x) \; \stackrel{h''}{\rrr}\; u^{-1}(x)\; \stackrel{u}{\rrr} \; x
$$
Comme le morphisme $h^{-1}(x) \rrr x$ est g\'eom\'etriquement connexe par hypoth\`ese, et que $h''$ est surjectif, on voit que $u^{-1}(x) \rrr x$ est g\'eom\'etriquement connexe ; mais, $u$ \'etant \'etale de pr\'esentation finie, $u^{-1}(x)$ est le spectre d'un produit fini de corps extensions finies s\'eparables de $\kappa(x)$ ; la connexit\'e g\'eom\'etrique entra\^ine qu'il n'y a qu'un seul corps, et qu'il est isomorphe \`a $\kappa(x)$. On en d\'eduit imm\'ediatement que le morphisme $\Delta_{u} : F \rrr F\times_{E}F$ est bijectif ; comme c'est une immersion ouverte, $\Delta_{u}$ est un isomorphisme ; on voit donc que $u$ est un monomorphisme fid\`element plat ; pour pouvoir conclure, par descente fpqc, il reste \`a v\'erifier que $u$ est quasi-compact, ou que $h = uh''$ l'est, puisque $h''$ est surjectif. Or, le compos\'e $T \stackrel{h}{\rrr} E \stackrel{g}{\rrr} S$ est quasi-compact par hypoth\`ese, et $\Delta_{g}$ est quasi-compact puisque $g$ est quasi-s\'epar\'e ; il suffit donc, une fois de plus, d'utiliser le proc\'edé de \ref{sA.1}.
\bb

\n L'implication (ii) $\Rightarrow$ (iii) est claire.
\bb

\n (iii) $\Rightarrow$ (i)\,  Puisque le compos\'e $T \stackrel{h}{\rrr} E \stackrel{g}{\rrr} S$ est universellement ouvert, ainsi que le morphisme diagonal $\Delta_{g}$ (car $g$ est \'etale), on d\'eduit de \ref{sA.1} que le morphisme $h$ est ouvert. Le sch\'ema induit par $E$ sur l'ouvert $h(T)$ est donc \'etale  et de pr\'esentation finie sur $S$. La propri\'et\'e universelle de $h$ entra\^ine donc que $h(T) = E$. 

Soit $T_{x} = h^{-1}(x) \rrr x$ une fibre de $h$ ; il s'agit de montrer que pour toute extension finie s\'eparable $k$ de $\kappa(x)$, le sch\'ema $k\times_{\kappa(x)}T_{x}$ est connexe. Or, il existe un morphisme \'etale de pr\'esentation finie $S' \rightarrow S$ et un point $x' \in S'\times_{S}E$ qui se projette sur $x$ et dont le corps r\'esiduel $\kappa(x')$ est isomorphe \`a $k$, et il s'agit de montrer  que la fibre $T_{x'} \rrr x'$ est connexe ; comme les hypoth\`eses sur $h$ sont pr\'eserv\'ees par le changement de base \'etale $S' \rrr S$, il suffit de montrer que les fibres sont connexes.

 On va proc\'eder par \'etapes, en montrant d'abord que pour tout ferm\'e connexe $F \subset E$, $h^{-1}(F)$ est connexe. Consid\'erons deux ferm\'es   $T_{1}$ et $T_{2}$ de $h^{-1}(F)$, donc de $T$, tels que
$$
h^{-1}(F) \, = \, T_{1} \cup T_{2}, \qquad T_{1} \cap T_{2}  = \emptyset .
$$
On va montrer que l'un des deux est vide. Les ouverts compl\'ementaires $U_{i} = T - T_{i}$  ont les propri\'et\'es suivantes :
$$
U_{1} \cap U_{2} = T - h^{-1}(F), \qquad U_{1} \cup U_{2} = T.
$$
Soit $E'$ le sch\'ema obtenu en recollant les sch\'emas  $h(U_{i})$ (des ouverts de $E$) le long de l'ouvert  $h(U_{1}\cap U_{2}) = E - F$. Ce sch\'ema $E'$ est \'etale sur $S$, tout comme les ouverts $h(U_{i})$ ; mais on a suppos\'e que $S$ est noeth\'erien (c'est ici le seul passage o\`u cette hypoth\`ese est utilis\'ee) ; comme $E$ est de pr\'esentation finie sur $S$ ses ouverts $h(U_{i})$ sont noeth\'eriens, et, en particulier, quasi-compacts et quasi-s\'epar\'es ; ainsi, $E'$ est lui aussi \'etale et de pr\'esentation finie sur $S$. Comme $h$ se factorise en $T \rrr E' \rrr E$, o\`u  $T \rrr E'$ est surjectif, la propri\'et\'e universelle de $h$ implique que $E' \rrr E$ est un isomorphisme, i.e. que l'ouvert de recollement est le m\^eme pour $E'$ et pour $E$, soit 
$$
h(U_{1} \cap U_{2}) \, = \, h(U_{1}) \cap h(U_{2}).
$$
On v\'erifie que $F \cap h(U_{1}) = h(T_{2})$  et $F \cap h(U_{2}) = h(T_{1})$; ce sont donc des ouverts de $F$ ; comme $h(U_{1} \cap U_{2}) = E - F$, l'\'egalit\'e ci-dessus montre que $\{h(T_{1}), h(T_{2})\}$ forme une partition ouverte de $F$ ; mais on a suppos\'e que $F$ est connexe ; donc l'un de ces ouverts est vide, ce qui entra\^ine que $T_{1}$ ou $T_{2}$ est vide ; bref, $h^{-1}(F)$ est connexe.
\medskip

Revenons \`a la connexit\'e des fibres. Soit $x$ un point de $E$, et soit $s$ son image dans $S$. Si $s$ est un point ferm\'e de $S$, alors $x$ est un point ferm\'e de $E$, et l'argument pr\'ec\'edent permet de conclure ; sinon on doit utiliser le ``passage \`a la limite'' suivant. Soit $S_{0}$ un ouvert affine de $S$ contenant $s$, et soit $S_{\lambda} \subset S_{0}$ la famille des ouverts affines de $S_{0}$ contenant $s$, de sorte que l'on a $\bigcap_{\lambda} S_{\lambda} = \s(\oo_{S, s})$. Enfin, notons 
$$
T_{\lambda} \; \stackrel{h_{\lambda}}{\rrr}\; E_{\lambda} \; \stackrel{g_{\lambda}}{\rrr}\; S_{\lambda}
$$
les morphismes obtenus par image r\'eciproque. On a $\bigcap_{\lambda} E_{\lambda} = g^{-1}(\s(\oo_{S, s}))$ ; par suite, $g^{-1}(s)$ est ferm\'e dans $\bigcap_{\lambda} E_{\lambda}$ ; mais, $g$ \'etant \'etale, $x$ est ferm\'e dans $g^{-1}(s)$ ; il est donc aussi ferm\'e dans $\bigcap_{\lambda} E_{\lambda}$ ; on en tire qu'en posant $F_{\lambda} = E_{\lambda} \cap \overline{\{x\}}$, on a
$$
\bigcap_{\lambda} F_{\lambda} = \{x\},\qquad {\rm et} \qquad  h^{-1}(x) = \bigcap_{\lambda} h^{-1}(F_{\lambda}) .
$$

\n Comme $S_{\lambda} \rrr S_{0}$ est \'etale de pr\'esentation finie, l'hypoth\`ese (iii) implique que le morphisme $h_{\lambda} : T_{\lambda} \rrr E_{\lambda}$ est universel pour les  $S_{\lambda}$-morphismes de $T_{\lambda}$ vers un sch\'ema \'etale de pr\'esentation finie sur $S_{\lambda}$ ; d'apr\`es le d\'ebut de la d\'emonstration, les ferm\'es irr\'eductibles $F_{\lambda} \subset E_{\lambda}$ ont donc une image r\'eciproque connexe dans $T_{\lambda}$, et il faut en d\'eduire que leur intersection 
  $h^{-1}(x)$ est connexe. Or, les morphismes de transition $h^{-1}(F_{\mu}) \rrr h^{-1}(F_{\lambda})$ sont affines, et, $h$ \'etant  quasi-compact,  le sch\'ema $h^{-1}(F_{0})$ est quasi-compact ; on peut donc utiliser \cite[8.4.1, ii)]{EGAIV3}, en rempla\c cant, dans cet \'enonc\'e, $S_{\alpha}$ et les $S_{\lambda}$ par notre $h^{-1}(F_{0})$ et nos $h^{-1}(F_{\lambda})$ ; cela permet de conclure que la fibre $h^{-1}(x)$ est connexe. 
\end{proof}

\subsection{Un théorème de Romagny (et de Laumon--Moret-Bailly)}\

\noindent L'\'e\-qui\-va\-len\-ce \'etablie dans la proposition \ref{p9.1} conduit \`a un changement de point de vue, et elle explique l'importance pour notre propos du travail de {\sc Romagny}, qui porte sur la repr\'esentabilit\'e du foncteur des ouverts qui sont r\'eunion de composantes connexes des fibres.

\begin{prop}[\sc M. Romagny]\phantomsection\label{p10.2}  Soit $f : T \rrr S$ un morphisme lisse et de pr\'esentation finie. Alors il existe un \emph{espace alg\'ebrique} $\pi_{0}(T/S)$ qui est \'etale et quasi-compact sur $S$ et un morphisme $h : T \rrr \pi_{0}(T/S)$ \`a fibres g\'eom\'etriquement connexes. La formation de $\pi_{0}(T/S)$ commute \`a tout changement de base sur $S$.
\end{prop}

\begin{proof} \cite[th. 2.5.2]{Rom11}, \cite[6.8]{LMB00}.
\end{proof}

\n  Le r\'esultat de {\sc Romagny} reste vrai si $T \rrr S$ est un espace alg\'ebrique lisse de pr\'esentation finie. C'est donc un r\'esultat d'adjonction, mais relative \`a l'inclusion de la cat\'egorie des espaces alg\'ebriques \'etales quasi-compacts, dans celle des espaces alg\'ebriques lisses quasi-compacts. Notre question initiale est plus sp\'ecifique en ce qu'elle vise une adjonction r\'ealis\'ee \emph{par un sch\'ema} ;  le travail de {\sc Romagny} indique que cela est en g\'en\'eral sans espoir. Mais il indique aussi une issue : puisqu'un espace alg\'ebrique \'etale est un sch\'ema d\`es qu'il est \emph{s\'epar\'e}, il faut reformuler la question initiale en se restreignant aux \'etales {\it s\'epar\'es}. C'est l'objet de ce qui suit.

Cette restriction ne permet plus d'utiliser  l'\'equivalence \'etablie dans \ref{p9.1} car la d\'emonstration de cette proposition repose sur la construction de sch\'emas \'etales \'eminemment non s\'epar\'es ; aussi, l'adjoint consid\'er\'e plus bas, puisqu'il est essentiellement ind\'ependant de la connexit\'e des fibres, ne peut  plus 
\^etre not\'e $\pi_{0}(T/S)$, et nous avons d\^u  introduire la notation $\pi^s(T/S)$, l'exposant ``$s$'' \'etant  l\`a pour \'evoquer la s\'eparation.

\`A un sch\'ema $T$ sur  $S$ (pour le moins plat de pr\'esentation finie) on cherche donc dans la suite \`a associer un sch\'ema  $\pi^s(T/S)$, \'etale quasi-compact \emph{et s\'epar\'e} sur $S$ muni d'un morphisme de $S$-sch\'emas
 $$
 T \; \rrr \; \pi^s(T/S)
 $$
 qui soit universel pour les $S$-morphismes de $T$ vers un \'etale s\'epar\'e.
  
 Nous allons montrer l'existence de cet adjoint sous des hypoth\`eses g\'en\'erales sur $T/S$, mais \`a condition de supposer $S$ int\`egre.
 Pour obtenir les  propri\'et\'es raisonnables attendues (compatibilit\'e au produit, \`a certains changements de base, etc.) il faudra supposer de plus que $S$ est normal.
 
\subsection{Sorites} Commen\c cons par quelques remarques qui montrent que l'hypoth\`ese de s\'eparation simplifie beaucoup les choses.

Soit $f : T \rrr S$ un morphisme  plat et de pr\'esentation finie de sch\'emas. On consid\`ere   d'abord la cat\'egorie  $\e(T/S)$ dont les objets sont les  factorisations de $f$ en   
$
T \stackrel{h}{\rrr} E   \stackrel{g}{\rrr}  S,
$
o\`u  $g$ est \'etale  et \emph{s\'epar\'e} et o\`u $h$ est  surjectif, une fl\`eche de $(g, h)$ vers $(g', h')$ \'etant un morphisme de $S$-sch\'emas
$u : E \rrr E'$ tel que $h' = uh$ et $g'u = g$. Notons que $h$ \'etant supos\'e surjectif et $gh$ quasi-compact, le morphisme $g$ est quasi-compact.

\begin{lemme}\label{l10.1}\phantomsection Entre deux factorisations il existe au plus un morphisme. 
\end{lemme}

\begin{proof}  Soit $f = gh$ une factorisation de $\e(T/S)$. Le morphisme diagonal $\Delta_{g}$ \'etant une immersion ouverte et ferm\'ee  le morphisme $h$ est  plat et de pr\'esentation finie, tout comme $gh = f$  (cf. \ref{sA.1}) ; il est donc ouvert ; comme il est suppos\'e surjectif, $h$ est un \'epimorphisme \cite[VIII 5.3]{SGA1}.
\end{proof}

\begin{lemme}\phantomsection La cat\'egorie $\e(T/S)$ est filtrante.
\end{lemme}

\begin{proof}  Soient, en effet, $f = g'h' = g''h''$ deux factorisations ; elles donnent lieu au diagramme
$$
\xymatrix{&&&E'\ar[dr]^{g'} &\\
T \ar[urrr]^{h'} \ar[rr]^h \ar[drrr]_{h''} &&E'\times_{S}E'' \ar[ur] \ar[dr] && S\\ 
&&& E'' \ar[ur]_{g''} &}
$$
dans lequel $h$ est ouvert ; par suite, $E = h(T)$ est un ouvert du produit fibr\'e, et il est muni des projections  $u' : E \rrr E'$  et $u'' : E \rrr E''$ qui sont \'etales et s\'epar\'ees, et le morphisme compos\'e $g'u'=g''u'' : E \rrr S$ est donc lui aussi \'etale et s\'epar\'e.
\end{proof}

\begin{lemme}\phantomsection\label{l9.1} Soit $S$ un sch\'ema int\`egre, de point g\'en\'erique $\eta$. Soient $g : E \rrr S$  et  $g' : E' \rrr S$ deux morphismes. On suppose que \\
a) $g$ est fidèlement plat et quasi-compact (fpqc), et s\'epar\'e ;\\
b) $g'$ est étale et quasi-s\'epar\'e. \\
Si un morphisme de $S$-sch\'emas $u : E \rrr E'$ est surjectif et s'il induit un isomorphisme $E_{\eta} \; \wt{\rrr}\; E'_{\eta}$ sur les fibres g\'en\'eriques, alors $u$ est un isomorphisme.
\end{lemme}

\begin{proof} Montrons d'abord que $u$ est un morphisme s\'epar\'e, plat et quasi-compact. Le morphisme $g = g'u$ poss\`ede ces propri\'et\'es ; gr\^ace au proc\'ed\'e de \ref{sA.1}, il suffit de v\'erifier que  $\Delta_{g'}$ les poss\`ede aussi ; or ce morphisme diagonal est une immersion (donc un morphisme s\'epar\'e), qui est ouverte (car $g'$ est \'etale), c'est donc un morphisme plat, et  $\Delta_{g'}$ est quasi-compact (car $g'$ est quasi-s\'epar\'e). 

\n Pour montrer que $u$ est un isomorphisme, il suffit donc, par descente fpqc \cite[VIII 5.4]{SGA1}, de montrer que l'un des morphismes de projection $E\times_{E'}E \rrr E$ est un isomorphisme, ou, encore, que le morphisme diagonal $\Delta_{u} : E \rrr E\times_{E'}E$ est un isomorphisme. Comme $\Delta_{u}$ est une immersion ferm\'ee, il suffit de voir qu'elle est \scd e. Consid\'erons le carr\'e commutatif
$$
\xymatrix{E  \ar[r]^{\Delta_{u}} & E\times_{E'}E\\
E_{\eta} \ar[u] \ar[r]_>>>>{ (\Delta_{u})_{\eta}}& (E\times_{E'}E)_{\eta} \ar[u] .}  
$$
Les morphismes verticaux sont \scd s\, puisqu'ils proviennent par changements de base plats  du monomorphisme quasi-compact  $\eta \rrr S$ qui est \scd, et le morphisme horizontal du bas est un isomorphisme puisque $u_{\eta}$ est un isomorphisme ; $\Delta_{u}$ est donc bien \scd. 
\end{proof}

\n On notera que la conclusion est fausse si on ne suppose pas $g$ s\'epar\'e : par exemple, le sch\'ema $S'$ obtenu par recollement de deux copies du sch\'ema int\`egre $S$ le long d'un ouvert non vide et $\neq S$ fournit un morphisme \'etale surjectif $S' \rrr S$ qui est g\'en\'eriquement un isomorphisme, et qui n'est pas lui-m\^eme un isomorphisme.

\subsection{Construction de $\pi^s$}
 
\begin{prop}\phantomsection \label{p9.3} Soit $f : T \rrr S$ un morphisme fid\`element plat de pr\'esentation finie. On suppose que $S$ est int\`egre. On consid\`ere  la cat\'egorie $\e(T/S)$ dont les objets sont les factorisations  de $f$ de la forme
$$
T \; \stackrel{h}{\rrr} \; E \; \stackrel{g}{\rrr} \; S ,
$$
o\`u $h$ est surjectif et $g$ est \'etale et \emph{s\'epar\'e}.
Alors cette cat\'egorie poss\`ede un objet initial, dont le $S$-sch\'ema \'etale  s\'epar\'e sera not\'e $\pi^s(T/S)$. Ce sch\'ema $\pi^s(T/S)$ d\'epend fonctoriellement de $T$, et ce foncteur est un adjoint \`a gauche de l'inclusion de la cat\'egorie des $S$-sch\'emas \'etales et s\'epar\'es dans celle des sch\'emas fid\`element plats de pr\'esentation finie sur $S$.
\end{prop}

\begin{proof} Remarquons d'abord que, le morphisme $f$ \'etant quasi-com\-pact et $h$ surjectif, les sch\'emas $E$ qui interviennent, et en particulier $\pi^s(T/S)$, sont quasi-compacts sur $S$. Soit $\eta = \s(k)$ le point g\'en\'erique de $S$. Puisque $T_{\eta}$ est de type fini sur un corps, ses composantes connexes sont en nombre fini, et chacune d'elles poss\`ede des points ferm\'es, lesquels ont un corps r\'esiduel fini sur $k$ \cite[6.5.2]{EGAG}. Choisissons un point ferm\'e dans chacune des composantes connexes de $T_{\eta}$, et soit $Z \subset T_{\eta}$ le sch\'ema r\'eduit constitu\'e par la r\'eunion de ces points. C'est un sch\'ema fini sur $\eta$ qui a la propri\'et\'e suivante : pour toute factorisation $T_{\eta} \rrr F \rrr \eta$ o\`u $F$ est un sch\'ema r\'eduit et fini sur $\eta$, le morphisme compos\'e $\varphi : Z \rrr F$ est surjectif ; en effet, puisque chaque point de $F$ est ouvert et ferm\'e, chaque fibre de $T_{\eta}\rrr F$ est une partie ouverte et ferm\'ee de $T_{\eta}$ ; par suite elle est une r\'eunion de composantes connexes, et elle rencontre donc $Z$.

\n On en d\'eduit l'in\'egalit\'e 
$$
{\rm rang}_{k}(F) \leq {\rm rang}_{k}(Z) .
$$
En particulier si $E$ provient de la cat\'egorie $\e(T/S)$, le rang de $E_{\eta}$ sur $k$ est born\'e. (On aura reconnu ici l'argument usuel pour d\'emontrer  \ref{p9.2}.)

\n Consid\'erons une factorisation dans $\e(T/S)$,  $T\stackrel{h}{\rrr}  E  \stackrel{g}{\rrr} S$  pour laquelle le rang de $E_{\eta} \rrr \eta$ est maximum, et montrons qu'elle est un \'el\'ement initial de la cat\'egorie $\e(T/S)$. Comme cette cat\'egorie est filtrante, tout revient \`a voir que pour toute autre factorisation $f = g'h'$, tout morphisme $u : E' \rrr E$ est un isomorphisme. La commutativit\'e du triangle de gauche dans
$$
\xymatrix{
& E' \ar[dd]^u \ar[dr]^{g'} &\\
T \ar[ur]^{h'} \ar[dr]_{h} && S\\
& E \ar[ur]_{g}&}
$$
montre que $u$ est surjectif puisque  $h$ l'est ; par suite le morphisme $u_{\eta} : E'_{\eta} \rrr E_{\eta}$ est surjectif, d'où l'on tire l'inégalité  ${\rm rang}_{k}(E_{\eta}) \leq {\rm rang}_{k}(E'_{\eta})$ ; mais le choix de $E$ implique que ${\rm rang}_{k}(E_{\eta})$ est maximum, donc $u_{\eta}$ est un isomorphisme et le lemme \ref{l9.1} permet de conclure.
\medskip

V\'erifions la fonctorialit\'e en $T$. Soit $\theta : T' \rrr T$ un morphisme (quelconque) de $S$-sch\'emas  plats de pr\'esentation finie. Le morphisme  plat de pr\'esentation finie $ T' \rrr S$ se factorise en
$$
T' \; \stackrel{\theta}{\rrr}\; T \; \stackrel{h}{\rrr}\; \pi^s(T/S) \; \stackrel{g}{\rrr}\; S
$$
Comme $\Delta_{g}$ est une immersion ouverte et ferm\'ee, c'est aussi un morphisme plat de pr\'esentation finie ; donc $h\theta$ est plat de pr\'esentation finie  (cf. \ref{sA.1}), et en particulier c'est un morphisme ouvert ; soit $F \subset \pi^s(T/S)$ son image ; c'est un $S$-sch\'ema \'etale et s\'epar\'e, et quasi-compact puisque $T'$ est quasi-compact sur $S$ ; la propri\'et\'e universelle de $\pi^s(T'/S)$ fournit le morphisme  $\pi^s(T'/S) \rrr F \subset \pi^s(T/S)$ cherch\'e.
\end{proof}

\subsection{Variation de la base: pathologies} Sous les hypoth\`eses de la proposition \ref{p9.3}, le sch\'ema \'etale  $\pi^s(T/S)$ est caract\'eris\'e par le fait que le rang sur $k$ de sa fibre $\pi^s(T/S)_{\eta}$ est maximum, {\it parmi les rangs g\'en\'eriques des \'etales s\'epar\'es de} $\e(T/S)$.
\medskip

\n En utilisant de $k$-sch\'ema \'etale $\pi_{0}(T_{\eta}/\eta)$ de  la proposition \ref{p9.2}, on a donc l'in\'egalit\'e
 \begin{equation}\label{eq9.1}
 {\rm rang}_{k}(\pi^s(T/S)_{\eta}) \leq {\rm rang}_{k}(\pi_{0}(T_{\eta}/\eta)).
 \end{equation}

 Cette in\'egalit\'e peut \^etre stricte, comme le montrent les deux exemples suivants ; dans le premier, la base est normale, le morphisme est fini libre mais il n'est pas \'etale, dans le second le morphisme est \'etale mais la base n'est pas normale.
 
\begin{ex}\phantomsection Considérons l'anneau $ A  = {\bf R}[X] $ comme un sous-anneau de $B = {\bf R}[X, Y]/(X^2+Y^2)$ ;  l'homomorphisme de ${\bf R}[X]$-alg\`ebres   ${\bf R}[X, Y]/(X^2+Y^2) \rrr  {\bf C}[X]$, d\'efini par $Y \mapsto iX$, permet d'identifier $B$ \`a l'ensemble des polyn\^omes complexes dont le terme constant est r\'eel. Le morphisme $A \rrr B$  est fini libre de rang 2, et il n'est pas \'etale, comme on le voit en faisant $X = 0$ ; mais l'homomorphisme localis\'e $A_{X} \rrr B_{X}$ est \'etale. 
 
 Désignons par $f : T = \s(B)  \rrr S = \s(A) = {\bf A}^{1}_{{\bf R}}$ le morphisme associé à l'inclusion $A \subset B$ ; on peut voir $T$ comme le schéma obtenu à partir de ${\bf A}^{1}_{{\bf C}}$ en ``pin\c cant`` l'origine par remplacement du corps résiduel ${\bf C}$ par ${\bf R}$. Au dessus du compl\'ementaire de l'origine, le morphisme $f$ est \'etale fini, de sorte que $\pi_{0}(T_{\eta}/\eta) \simeq T_{\eta}$ est de rang 2. Montrons que $\pi^s(T/S) = S$. Soit $f = gh$ une factorisation avec $h : T \rrr E$ un $S$-morphisme surjectif et $g : E \rrr S$ un morphisme étale et séparé ; il faut voir que $g$ est un isomorphisme. Si ce n'est pas le cas, alors $h$ induit un isomorphisme $T_{\eta} \rrr E_{\eta}$ puisque $T_{\eta}$ est de rang 2 sur $\eta$, et on applique le lemme \ref{l9.1}, en rempla\c cant dans l'énoncé $( g, g', u)$ par $(f, g, h)$, pour voir que $h$ serait un isomorphisme, ce qui est absurde puisque $f$ n'est pas étale.
\end{ex}
  
\begin{ex}\phantomsection On a construit en \ref{ex7.2} un morphisme \emph{\'etale}, donc lisse,  $f : T \rrr S$, o\`u $S$ est le spectre d'un anneau local noeth\'erien int\`egre de dimension 1 (non normal !), et qui n'admet pas de s\'eparateur. Notant $\eta$ le point g\'en\'erique de $S$, on a dans cet exemple, $T_{\eta} = \eta \sqcup \eta \sqcup \eta = \pi_{0}(T_{\eta}/\eta)$ ; pour toute factorisation $T \stackrel{h}{\rightarrow} E \stackrel{g}{\rightarrow} S$ avec $g$ \'etale s\'epar\'e, le rang de $E_{\eta}$ est donc $\leq$ 3 ; s'il était égal à 3, le lemme \ref{l9.1} impliquerait que $h$ serait un isomorphisme, ce qui est impossible puisque $f$ n'est pas séparé et que $g$ l'est. Le rang de $E_{\eta}$ est donc < $3$, et $\pi^s(T/S)$ ne prolonge pas $\pi_{0}(T_{\eta}/\eta) = T_{\eta}$.
 \end{ex}
 
   On va montrer plus bas que si $S$ est normal et si $T$ est lisse sur $S$, l'in\'egalit\'e \eqref{eq9.1} est une \'egalit\'e, i.e. que le morphisme \'etale ``g\'en\'erique''  $\pi_{0}(T_{\eta}/\eta) \rrr \eta$ se prolonge en un \'etale s\'epar\'e sur $S$ ; de plus, dans ce cas, la formation de $\pi^s$ commute aux changements de base lisses. 
  
  Voir ces consid\'erations comme une tentative de prolongement \`a $S$ d'un \'etale  ``g\'en\'erique''  peut \'evoquer la probl\'ematique des mod\`eles de N\'eron \cite[p. 12]{BLR90}; mais l'analogie tourne court, d'abord parce que la construction de N\'eron repose de fa\c con cruciale sur une loi de groupe (g\'en\'erique), alors qu'ici on impose au prolongement \`a $S$ d'\^etre un quotient de $T$, et ensuite parce que le mod\`ele de N\'eron est objet final de la cat\'egorie des prolongements, alors que $\pi^s(T/S)$ est un objet initial.

\subsection{Cas d'une base normale}

\begin{thm}\phantomsection\label{p9.4} Soient $S$ un sch\'ema int\`egre et \emph{normal}, et $f : T \rrr S$ un morphisme lisse et de pr\'esentation finie. D\'esignons par $\pi_{0}(T/S)^{\rm sep}$ le $S$-sch\'ema  \'etale quasi-compact et s\'epar\'e  qui est associ\'e par la proposition \ref{p9.3} au $S$-espace alg\'ebrique $\pi_{0}(T/S)$ de {\sc Romagny} (proposition \ref{p10.2}). Alors la factorisation 
$$
T \; \rrr \;  \pi_{0}(T/S)^{\rm sep} \; \rrr \; S
$$
est la factorisation initiale de $\e(T/S)$. Autrement dit, on a, si $S$ est normal et int\`egre, un isomorphisme de foncteurs {\rm (}en $T${\rm )} : 
$$
\pi^s(T/S) \simeq \pi_{0}(T/S)^{\rm sep}.
$$ 
\end{thm}

\begin{proof} Soit $\eta$ le point g\'en\'erique de $S$ ; puisque la formation de $\pi_{0}$ commute aux changements de base, on a un isomorphisme
$$
\pi_{0}(T/S)_{\eta} \; \wt{\rrr}\; \pi_{0}(T_{\eta}/\eta) ,
$$
o\`u le sch\'ema de droite repr\'esente les composantes connexes de $T_{\eta}$, i.e. c'est exactement celui de la proposition \ref{p9.2}.
D'apr\`es la proposition \ref{p8.1}, le morphisme g\'en\'erique  $\pi_{0}(T/S)_{\eta} \rrr \pi_{0}(T/S)^{\rm sep}_{\eta}$ est un isomorphisme ; par suite, le sch\'ema \'etale $\pi_{0}(T/S)^{\rm sep}$ a le rang g\'en\'erique maximum, celui de $\pi_{0}(T_{\eta}/\eta)$ ; il r\'ealise donc la factorisation initiale cherch\'ee. 
\end{proof}

Noter qu'en g\'en\'eral le morphisme $\pi_{0}(T/S) \rrr \pi_{0}(T/S)^{\rm sep}$ n'est pas un isomorphisme et que, par suite, les fibres de $T \rrr \pi^s(T/S)$ ne sont pas connexes, comme le montre un exemple explicite, et bien classique, celui d'une conique qui se sp\'ecialise en deux droites parall\`eles.

\begin{ex}\phantomsection Soit $S = \s(R)$ le spectre d'un anneau de valuation discr\`ete, d'uniformisante $t$, et tel que $2 \in R^{\times}$. Posons $F(X, Y) = X(X-1)+tY^2$. Comme on a $1 = F'^2_{X} - 4F + 2YF'_{Y}$, le morphisme 
 $$
 f : T = \s(R[X, Y]/(F)) \rrr S
 $$
est lisse. La fibre sp\'eciale ($t=0$) est la r\'eunion de deux droites disjointes, et la fibre g\'en\'erique est g\'eom\'etriquement int\`egre : pour le voir, notons $K$  le corps des fractions de $R$, et $L$ celui de l'anneau de $T$, et v\'erifions que $L$ est une extension transcendante pure de $K$ ; en effet, si $x$ et $y$ d\'esignent les images de $X$ et de $Y$ dans $L$, on a $x = x^2 + ty^2$, soit $1/x = 1 + t(y/x)^2$ ; d'o\`u $L = K(y/x)$. 

\n L'espace alg\'ebrique $\pi_{0}(T/S)$ est ici le sch\'ema obtenu par recollement de deux copies de $S$ le long du point g\'en\'erique, et  $\pi^s(T/S) = \pi_{0}(T/S)^{\rm sep} = S$.
\end{ex}

\begin{prop}\label{p10.1}\phantomsection Soit $S$ un sch\'ema normal  int\`egre de point g\'en\'erique $\eta$. Soient $f : T \rrr S$  et $f' : T' \rrr S$ deux morphismes  lisses et de pr\'esentation finie. Alors
\begin{thlist}
\item {\rm (Changement de base lisse)} \, Le morphisme $\pi^s(T\times_{S}T' / T') \rrr  \pi^s(T/S)\times_{S}T'$ est un isomorphisme.

\item {\rm (Compatibilit\'e au produit)} \, Le morphisme $\pi^s(T\times_{S}T' /S) \rrr  \break\pi^s(T/S)\times_{S}\pi^s(T'/S)$ est un isomorphisme.
\end{thlist}
\end{prop}

\begin{proof} L'hypoth\`ese de lissit\'e du morphisme $T' \rrr S$ et la normalit\'e de $S$ impliquent que le sch\'ema $T'$ est normal \cite[11.3.13]{EGAIV3} ; il est donc somme de ses composantes irr\'eductibles, ce qui permet de le supposer int\`egre ; soit $\xi$ son point g\'en\'erique. 

\n Pour d\'emontrer $i)$, il suffit, d'apr\`es le lemme \ref{l9.1}), de v\'erifier que le morphisme g\'en\'erique
$$
\pi^s(T\times_{S}T' / T')_{\xi} \, \rrr \, \pi^s(T/S)\times_{S}\xi = \pi^s(T/S)_{\eta}\times_{\eta}\xi
$$ 
est un isomorphisme. Or, la formation de l'espace alg\'ebrique $\pi_{0}$ commute \`a tout changement de base ; donc le morphisme de $T'$-espaces alg\'ebriques  
$\pi_{0}(T\times_{S}T' / T') \, \rrr \, \pi_{0}(T/S)\times_{S}T'$ est un isomorphisme, et il en est de m\^eme de 
$\pi_{0}(T\times_{S}T' / T')_{\xi} \, \rrr \, \pi_{0}(T/S)\times_{S}\xi = \pi_{0}(T/S)_{\eta}\times_{\eta}\xi$ 

Or, $T'$ \'etant normal, la proposition \ref{p9.4} permet les identifications 
$$\pi_{0}(T\times_{S}T' / T')_{\xi} = \pi^s(T\times_{S}T' / T')_{\xi} , \quad {\rm et } \quad
\pi_{0}(T/S)_{\eta} = \pi^s(T/S)_{\eta}.
$$ 

La v\'erification de $ii)$ est analogue : on se ram\`ene \`a la situation g\'en\'erique gr\^ace au lemme \ref{l9.1}, 
puis \`a l'\'enonc\'e correspondant pour les $\pi_{0}(T_{\eta}/\eta)$ parce que $S$ est normal, et, 
sur un corps cet \'enonc\'e est d\'emontr\'e dans \cite[I, \S4, 6.10, p.126]{DG70}.
\end{proof}
 
 \appendix
 
 \section{Isomorphismes locaux et mor\-phis\-mes s\'epar\'es}\label{sA}

\subsection{Le morphisme diagonal \cite[\S 4.4  et  \S5]{EGAG}}\label{sA.1}

Soient $f : X \rrr S$ un morphisme de sch\'emas et   $\Delta_{f} : X \rrr X\times_{S}X$ son morphisme diagonal. C'est une immersion, i.e. il existe un ouvert $U$ de $X\times_{S}X$ contenant $\Delta_{f}(X)$ et tel que $\Delta_{f}$ induise une immersion ferm\'ee $X \rrr U$.
\medskip

Rappelons comment interviennent en pratique les propri\'et\'es du morphisme diagonal :
pour  des morphismes $Z \stackrel{h}{\rrr} Y \stackrel{g}{\rrr} X$, des hypoth\`eses sur $gh$ et sur $\Delta_{g}$, entra\^inent souvent  des propri\'et\'es de $h$ ; en effet,  le morphisme $h$ est le compos\'e  $Z \rrr Z\times_{X}Y \rrr Y$, et les carr\'es  repr\'esent\'es ci-dessous sont cart\'esiens
$$
\xymatrix{
&Z \ar[r]^{gh} &X\\
 Z\ar[d]_{h} \ar[r]^<<<<<{(1,h)} &Z\times_{X}Y \ar[u] \ar[r]  \ar[d]^{h\times 1}& Y \ar[u]_{g}\\
 Y \ar[r]_<<<<<{\Delta_{g}}& Y\times_{X}Y &}
 $$
 Par suite, soit {\bf P}  une propri\'et\'e de morphismes qui est stable par changement de base et par composition. Si $gh$ et $\Delta_{g}$ v\'erifient {\bf P}, alors $h$ v\'erifie  {\bf P}. \medskip
 
 \subsubsection{}\label{sA.1.1} Voici un cas particulier utile.
 
 \n Supposons que $g$ soit un monomorphisme (i.e. $\Delta_{g}$ est un isomorphisme), par exemple une immersion, alors le carr\'e
 $$
 \xymatrix{Z \ar[r]^{gh} \ar@{=}[d] & X\\
 Z \ar[r]_{h} & Y \ar[u]_{g}}
 $$
 est cart\'esien.
  
\subsection{Morphismes s\'epar\'es}

 On dit qu'un morphisme de sch\'emas  $f : X \rrr S$ est {\it s\'epar\'e} si son morphisme diagonal $\Delta_{f}$ est une immersion ferm\'ee ; il revient au m\^eme de dire que l'ensemble  $\Delta_{f}(X)$ est ferm\'e dans $X\times_{S}X$.  
 
 \n Un sch\'ema $X$ est dit s\'epar\'e si le morphisme $X \rrr \s({\bf Z})$ est s\'epar\'e.
\medskip

\n Un sch\'ema affine est s\'epar\'e \cite[5.2.2]{EGAG}, et tout ouvert d'un sch\'ema s\'epar\'e est lui-m\^eme s\'epar\'e.
\medskip

\n On dit que $f$ est \emph{quasi-s\'epar\'e} si le morphisme diagonal $\Delta_{f}$ est quasi-compact ; de m\^eme un sch\'ema $X$ est dit quasi-s\'epar\'e si le morphisme $X \rrr \s({\bf Z})$ est quasi-s\'epar\'e.
\medskip

\begin{lemme}\phantomsection\label{lA.1} Si $X$ est un sch\'ema s\'epar\'e (resp. quasi-s\'epar\'e), tout morphisme $f : X\rrr S$ est s\'epar\'e (resp. quasi-s\'epar\'e). R\'eciproquement, si le sch\'ema $S$ est s\'epar\'e (resp. quasi-s\'epar\'e) et si le morphisme $f$ est s\'epar\'e (resp. quasi-s\'epar\'e), alors $X$ est s\'epar\'e (resp. quasi-s\'epar\'e).
\end{lemme}

\begin{proof} Consid\'erons le diagramme suivant, o\`u le carr\'e est cart\'esien
$$
\xymatrix{ X \ar[r]^{\Delta_{f}} & X\times_{S}X \ar[r]^{\delta} \ar[d] & X\times X \ar[d]^{f\times f}\\
& S \ar[r]_{\Delta_{S}} & S\times S .} 
$$
Pour l'\'enonc\'e direct, il suffit d'appliquer \ref{sA.1.1} \`a la premi\`ere ligne du diagramme, puisque $\delta$ est une immersion, tout comme $\Delta_{S}$. Dans l'\'enonc\'e r\'eciproque, on suppose que $\Delta_{S}$, donc aussi $\delta$, sont des immersions ferm\'ees (resp. sont quasi-compacts), ainsi que $\Delta_{f}$ ; le compos\'e $\delta \circ \Delta_{f} = \Delta_{X}$ a donc la m\^eme propri\'et\'e.
\end{proof}

Rappelons que si $U$ et $V$ sont deux ouverts de $X$ le carr\'e suivant est cart\'esien :
$$
\xymatrix{U\times_{S}V \ar[r] & X\times_{S}X\\
U\cap V \ar[r] \ar[u] & X \ar[u]_{\Delta_{f}}
}
$$
\n Cette remarque est \`a la base du crit\`ere suivant \cite[5.3.6]{EGAG}:

\begin{lemme}\phantomsection\label{lA.3} Soient $S$ un sch\'ema affine d'anneau $R$,  et $f : X \rrr S$ un morphisme de sch\'emas. Pour que $f$ soit s\'epar\'e il faut et il suffit que pour tout couple d'ouverts affines $U$ et $V$ de $X$, $U \cap V$ soit un ouvert affine et que le morphisme de $R$-alg\`ebres provenant des restrictions
$$
\varphi_{UV} : \Gamma(U)\otimes_{R} \Gamma(V) \; \rrr \; \Gamma(U\cap V)
$$
soit surjectif.
\end{lemme}


\subsection{Isomorphismes locaux \cite[4.4]{EGAG}}

Un morphisme de sch\'emas $f : X \rrr Y$ est dit un \emph{\isl} si tout point de $X$ est contenu dans un ouvert $U$ de $X$ sur lequel $f$ induit une immersion ouverte $U \rrr Y$.
\medskip

Comme il est signalé dans l'introduction, cette notion traduit et précise l'intuition géométrique du \emph{recoller davantage}. En effet si $U$ et $V$ sont des ouverts de $X$ sur lesquels $f$ induit une immersion ouverte, alors l'image $f(U \cup V)$ est le schéma obtenu en recollant les ouverts $f(U)$ et $f(V)$ - qui sont respectivement isomorphes à $U$ et à $V$ - le long de l'ouvert $f(U) \cap f(V)$, lequel contient $f(U\cap V)$.

La notion d'isomorphisme local sera donc constamment utilisée dans notre article, et cela nous a conduit à approfondir les indications de {\it loc. cit.}.Voir le \S $2.3$. \medskip

Un tel morphisme est ouvert, plat et localement de pr\'esentation finie \cite[6.2.1]{EGAG}.
Pour tout point $x$ de $X$, le morphisme $\theta_{x} : \oo_{Y, f(x)} \rrr \oo_{X, x}$ est alors un isomorphisme. Réciproquement, si un morphisme $f$ est localement de présentation finie, et si les $\theta_{x}$ sont des isomorphismes, alors $f$ est un isomorphisme local \cite[6.6.4]{EGAG}.
\medskip

Noter que la propriété pour un morphisme d'\^{e}tre un isomorphisme local n'est évidemment pas locale pour la topologie fpqc : par exemple, une extension finie séparable de corps $K \rightarrow L$ de degré $n > 1$ n'est pas un isomorphisme local, mais le devient par produit tensoriel :  $\overline{K}\otimes_{K} L \simeq \overline{K}^{n}$. 
\medskip

Un isomorphisme local injectif est une immersion ouverte.

\bb

 Depuis \v{C}ech, un isomorphisme local est devenu familier, celui associé à une famille d'immersions ouvertes $f_{i} : U_{i} \rrr Y$, soit le morphisme
$$
f : X = \coprod_{i} U_{i} \rrr Y .
$$
Notons que ce morphisme est séparé. Cela le caractérise :
\medskip

\begin{prop}\phantomsection\label{pA.1} Soit $f : X \rrr Y$ un isomorphisme local \emph{séparé}. Alors  

\n i) Le morphisme $f$ est associé à une famille d'immersions ouvertes $f_{i} : U_{i} \rrr Y$, au sens o\`u $f$ est isomorphe au morphisme $\coprod_{i} U_{i} \rrr Y$ construit avec les $f_{i}$.

\n ii) Si, de plus, $f$ induit une injection sur l'ensemble des points maximaux de $X$, alors $f$ est une immersion ouverte.

\n iii) Si $Y$ est connexe et si $f$ est entier, alors $f$ est isomorphe au morphisme  $\coprod_{n}Y \rrr Y$, o\`u le premier symbole désigne le schéma somme d'un nombre fini de  copies du schéma $Y$.
\end{prop}

\begin{proof} $i)$\; Désignons par $\mathcal{U}$ l'ensemble des ouverts $U$ de $X$ au dessus desquels la restriction de $f$ est injective, c'est-à-dire une immersion ouverte $U \rrr Y$ ; pour un tel ouvert $U$, $f(U)$ est donc un ouvert de $Y$, et $f$ induit un isomorphisme $U \; \wt{\rrr} \;  f(U)$. Cet ensemble $\mathcal{U}$ est une base de la topologie de $X$. 

 L'injectivité étant une propriété de caractère fini, l'ensemble $\mathcal{U}$ est inductif. Soit $U$ un élément maximal de $\mathcal{U}$. On va montrer que $U$ est fermé dans $X$. Puisque les ouverts de $\mathcal{U}$ forment un recouvrement de $X$, cela montrera qu'il existe une \emph{partition} de $X$ formée de tels ouverts ; c'est une autre fa\c con d'énoncer $i)$. 

Soit donc $U$ un élément maximal de $\mathcal{U}$. L'isomorphisme $U \simeq f(U)$ induit par $f$ se factorise en
$$
U \; \stackrel{u}{\rrr} \; f^{-1}f(U) \; \stackrel{v}{\rrr} \; f(U)
$$
Puisque $v$ est séparé, tout comme $f$, $u$ est une immersion (ouverte et) fermée ; il existe donc un ouvert $V$ de $X$ et une partition 
$$
f^{-1}f(U) \; = \; U \sqcup V .
$$
Montrons que l'on a $X = U \cup \overline{V}$. L'ouvert $X - \overline{V}$ est recouvert par les ouverts $W \in \mathcal{U}$ qui sont  disjoints de $\overline{V}$; un tel ouvert $W$, étant en particulier disjoint de $V$, on a  l'inclusion $W\cap f^{-1}f(U) \subset U$, d'o\`u l'on déduit que $f$ est injective sur l'ouvert $W \cup U$ ; mais comme $U$ a été choisi maximal dans $\mathcal{U}$, on voit que $W \subset U$. 

Montrons ensuite que $\overline{U} \cap \overline{V} = \emptyset$, ce qui, compte tenu de ce qui précède,  entra\^inera l'égalité cherchée $U = \overline{U}$. On raisonne par l'absurde en considérant un élément $x$ de cette intersection. Soit $W$ un ouvert de $\mathcal{U}$ contenant $x$ ; les ouverts $W \cap U$ et $W \cap V$ sont donc non vides ; mais la restriction de $f$ à l'ouvert $(W\cap U)\cup (W \cap V) \subset W$ est injective puisque $W \in \mathcal{U}$, ce qui est absurde en vertu de l'inclusion $f(V) \subset f(U)$.
\medskip

\n $ii)$ (Pour $X$ irréductible, cet énoncé figure aussi dans \cite[4.4.8]{EGAG}.) L'hypothèse implique que les ouverts $f_{i}(U_{i})$ de $Y$ sont deux à deux disjoints. 
\medskip

\n $iii)$ Chaque $U_{i}$ s'identifie à un (ouvert et)  fermé du schéma somme  $X = \coprod_{i\in I} U_{i}$ ; le morphisme $f$ étant fermé,  $f(U_{i})$ est ouvert et fermé dans le schéma $Y$ qui est supposé connexe ; donc chaque $f_{i}$ est un isomorphisme. Enfin, un morphisme entier étant quasi-compact, l'ensemble $I$ est fini. 
\end{proof}

\subsection{Morphismes sch\'ematiquement dominants \cite[5.4]{EGAG}}\label{sA.4}

\begin{defn}\phantomsection\label{dA.1}  On dit qu'un morphisme $t : Y \rrr X$ est \emph{\scd}\, si l'application $\oo_{X} \rrr t_{\star}(\oo_{Y})$ est injective.
\end{defn}

Un tel morphisme est dominant, i.e. l'ensemble $t(Y)$ est dense dans $X$.
Une immersion \scd e est ouverte.
Une immersion ferm\'ee qui est \scd e  est un isomorphisme.

Pour un sch\'ema int\`egre  $X$, de point g\'en\'erique $\eta$, le morphisme $\eta \rrr X$ est \scd. Plus g\'en\'eralement, si $X$ est r\'eduit et si l'ensemble de ses points maximaux est fini, en notant $X_{0}$ le sch\'ema somme de ces points, le morphisme $X_{0} \rrr X$ est quasi-compact, quasi-s\'epar\'e et \scd.
\medskip

Soit $U$ un ouvert d'un sch\'ema $X$. Si l'immersion  $U\rrr X$ est \scd e, alors $U$ contient les points maximaux de $X$. Si $X$ est localement noeth\'erien, $U \rrr X$ est sch\'ematiquement dominant si et seulement si $U$ contient l'ensemble Ass$(X)$ des points associ\'es \cite[3.1.8]{EGAIV2}.

\begin{lemme}\phantomsection \label{lA.5} Soit $t : Y \rrr X$ un morphisme quasi-compact, quasi-s\'epar\'e et \scd. Pour tout $\oo_{X}$-module $P$, quasi-coh\'erent et plat, l'application
$$
P \; \rrr \; t_{\star}t^{\star}(P)
$$
\it est injective. En particulier, si $t^{\star}(P) = 0$, alors $P = 0$.
\end{lemme}

\begin{proof} La platitude de $P$ entra\^ine que l'application $P \rrr P\otimes_{\oo_{X}}t_{\star}(\oo_{Y})$ est injective. Il suffit donc voir que l'application $P\otimes_{\oo_{X}}t_{\star}(\oo_{Y}) \rrr t_{\star}t^{\star}(P))$ est un isomorphisme. C'est \'evident pour $P = \oo_{X}$, puis pour $P$ libre de type fini. Dans le cas g\'en\'eral, on peut se restreindre au cas o\`u $X$ est affine, $X = \s(A)$ ; puisque $P$ est quasi coh\'erent il est associ\'e \`a un $A$-module plat, lequel est donc limite inductive filtrante de $A$-modules libres de type fini. D'o\`u le r\'esultat.
\end{proof}

La propri\'et\'e suivante sera tr\`es utilis\'ee.

\begin{lemme}[\protect{\cite[9.3.3]{EGAG}}]\phantomsection \label{lA.6} Soit $t : Y \rrr X$ un morphisme quasi-compact, quasi-s\'epar\'e et \scd. Alors le morphisme $t' : X'\times_{X}Y \rrr X'$ obtenu par un changement de base \emph{plat} $X' \rrr X$ est encore \scd.
\end{lemme}

\subsection{Adh\'erence sch\'ematique  \cite[6.10]{EGAG}}

Soit $t : Y \rrr X$ un morphisme quasi-compact et quasi-s\'epar\'e. L'image directe  $t_{\star}(\oo_{Y})$ est une $\oo_{X}$-alg\`ebre quasi-coh\'erente ; le noyau $\mathcal{I}$ du morphisme canonique $\oo_{X} \rrr t_{\star}(\oo_{Y})$ est alors quasi-coh\'erent.

\n \emph{L'adh\'erence sch\'ematique} de $t$, ou de $Y$ dans $X$, est  le ferm\'e $Z$ de $X$ d\'efini par l'id\'eal $\mathcal{I}$. Le morphisme $t :Y \rrr X$ se factorise donc en
 $$
 Y \; \stackrel{v}{\rrr} \; Z \; \stackrel{u}{\rrr} \; X
 $$
 o\`u $u$ est une immersion ferm\'ee et o\`u $v$ est \emph{\scd}, i.e. l'application $\oo_{Z} \rrr v_{\star}(\oo_{Y})$ est injective \cite[5.4.2]{EGAG}.

 \begin{lemme}[\protect{\cite[2.3.2]{EGAIV2}}]\phantomsection \label{lA.7} La formation de l'adh\'erence sch\'ematique d'un morphisme quasi-compact et quasi-s\'epar\'e commute aux changements de base plats.
 \end{lemme}

Si $Y$ est un sous-espace localement fermé d'un espace topologique $X$, on sait que $Y$ est ouvert dans son adhérence dans $X$ ([TG] I.20). Voici l'énoncé analogue pour les schémas.

 \begin{prop}\phantomsection \label{lA.8} Soient $t : Y \rrr X$ une immersion quasi-compacte, et $Y\stackrel{w}{\rrr} U \stackrel{j}{\rrr} X$ une factorisation de $t$, où $w$ est une immersion fermée et $j$ une immersion ouverte. Soit, d'autre part, $Y \; \stackrel{v}{\rrr} \; Z \; \stackrel{u}{\rrr} \; X$ la factorisation de $t$ par adh\'erence sch\'ematique. Alors, le carré
 $$
 \xymatrix{Y \ar[r]^{v} \ar[d]_{w} & Z \ar[d]^{u}\\
 U \ar[r]_{j}&X}
 $$
 est cartésien; en particulier $v$ est une immersion ouverte. De plus, $v$  quasi-compacte, et l'espace sous-jacent \`a $Z$ est l'adh\'erence de $Y$dans $X$.
\end{prop}
 
\begin{proof} L'existence de l'immersion  fermée $w$ et de l'immersion ouverte $j$ provient de la définition d'une immersion;  la factorisation de $t$ par adhérence schématique existe puisque $t$ est supposée quasi-compacte; on en déduit immédiatement que $v$ est quasi-compact (\ref{sA.1.1}). Le produit cartésien $V = U\times_{U}Z$ s'insère dans le diagramme commutatif suivant
$$
\xymatrix{&&Z\ar[dr]^{u} &\\
Y\ar[rru]^{v} \ar[r]^{\theta} \ar[rrd]_{w} & V\ar[ru]_{j'} \ar[dr]^{u'} && X\\
&& U\ar[ru]_{j}&} 
$$
 dans lequel $j$, et donc aussi $j'$, sont des immersions ouvertes. Il s'agit de voir que le morphisme $\theta$ est bijectif. Or, $\theta$ est \scd\,  puisque c'est la restriction à l'ouvert $V \subset Z$ du morphisme \scd \, $v = j'\theta$ ; de plus, $u'$ est une immersion fermée, tout comme $u$, ainsi que $w = u'\theta$ par hypothèse ; donc $\theta$ est aussi une immersion fermée, et la conjonction de ces deux propriétés implique que $\theta$ est  un isomorphisme.
 Enfin, comme $v : Y \rrr Z$ est \scd, il est dominant i.e. $\overline{v(Y)} = Z$.
 \end{proof}
  
\begin{lemme}[Unicit\'e de l'adh\'erence sch\'ematique]\phantomsection \label{lA.9} Consid\'erons un carr\'e commutatif de morphismes de sch\'emas
  $$
  \xymatrix{Y  \ar[r]^v \ar[d]_{v'} & Z \ar[d]^{u}\\
  Z' \ar[r]_{u'}& X .} 
  $$
 On suppose que $u$ et $u'$ sont des immersions ferm\'ees d\'efinies respectivement par les id\'eaux $\mathcal{I}$  et $\mathcal{I'}$ de $\oo_{X}$.
 Si le morphisme $v$ est sch\'ematiquement dominant, i.e. si \, $\oo_{Z} \rrr v_{\star}(\oo_{Y})$ est injectif, alors $\mathcal{I'} \subset \mathcal{I}$, et il existe un morphisme $w : Z \rrr Z'$ rendant les deux triangles commutatifs.
 
 \n Si, de plus, $v'$ est aussi sch\'ematiquement dominant, alors $w$ est un isomorphisme.\qed
 \end{lemme}

\begin{lemme}\phantomsection \label{lA.10} Soient $t : Y \rrr X$ et $t' : Y' \rrr X'$ des morphismes quasi-compacts et quasi-s\'epar\'es, $Y \; \stackrel{v}{\rrr} \; Z \; \stackrel{u}{\rrr} \; X$  et 
 $Y' \; \stackrel{v'}{\rrr} \; Z' \; \stackrel{u'}{\rrr} \; X'$ leur factorisation par adh\'erence sch\'ematique. Pour tout couple $f : X \rrr X'$  et  $g : Y \rrr Y'$ de morphismes de sch\'emas tels que $t'g = ft$, il existe un unique morphisme $h : Z \rrr Z'$ rendant commutatif le diagramme suivant.
 $$
 \xymatrix{Y \ar[r]^g \ar[d]_{v} &Y' \ar[d]^{v'}\\
 Z \ar[d]_{u} \ar[r]^h & Z' \ar[d]^{u'}\\
 X \ar[r]_{f} & X'}
 $$
 \end{lemme}
 
\begin{proof} L'unicit\'e de $h$ tel que $u'h = fu$ provient de ce que $u'$ est une immersion ferm\'ee, donc un monomorphisme. Pour voir l'existence  de $h$  introduisons  le produit fibr\'e $Z'' = X\times_{X'}Z'$ ; on a donc le diagramme commutatif suivant
 $$
  \xymatrix{&Y  \ar[dl]_{v} \ar[d]^{v''} \ar[r]& Y' \ar[d]\\
  Z \ar@{-->}[r] \ar[dr]_{u}&Z'' \ar[d]^{u''} \ar[r] & Z' \ar[d]\\
  & X \ar[r] & X'} 
  $$
 auquel on peut appliquer le lemme pr\'ec\'edent, qui fournit un morphisme $Z\rrr Z''$ ;  en le composant avec la projection $Z'' = X\times_{X'}Z' \rrr Z'$, on obtient le morphisme $h$ cherch\'e.
\end{proof}

\begin{lemme}\phantomsection \label{lA.11}  Soit $X$ un sch\'ema dont les anneaux locaux sont int\`egres, par ex. un sch\'ema normal, et qui n'a qu'un nombre fini de composantes irr\'eductibles. Soit $t : Y \rrr X$ un morphisme plat quasi-compact et quasi-s\'epar\'e, et 
$$
Y \stackrel{v}{\rrr} Z  \stackrel{u}{\rrr} X
$$
 sa factorisation par adh\'erence sch\'ematique. Alors $u$ est une immersion ouverte (et ferm\'ee)  ; en particulier, $u$ est plat.
 \end{lemme}
  
\begin{proof} Comme les anneaux locaux de $X$ sont int\`egres, les composantes irr\'eductibles de $X$ sont deux \`a deux disjointes ; \'etant en nombre fini, elles sont donc aussi ouvertes. De plus, $X$ est r\'eduit.
 
 \n L'image par le morphisme plat $t$ d'un point maximal $y$ de $Y$ est un point maximal $t(y)$ de $X$ ; soit $Z$ la r\'eunion des adh\'erences de ces points maximaux $t(y)$ ; c'est un ouvert et ferm\'e de $X$, que l'on munit de la structure de sch\'ema induite. Il est clair que $t$ se factorise en  $Y \stackrel{v}{\rrr} Z  \stackrel{u}{\rrr} X$. Il reste \`a voir que le morphisme $\oo_{Z} \rrr v_{\star}(\oo_{Y})$ est injectif ; mais $Z$ est r\'eduit, tout comme $X$, et les points maximaux de $Z$ sont dans $v(Y)$. 
 \end{proof}

\section{Adh\'erence sch\'ematique d'une  relation d'\'e\-qui\-va\-len\-ce}\label{sB}

Soit $T \rrr S$ un morphisme de sch\'emas. Soit $R$ le sous-sch\'ema ferm\'e de $T\times_{S}T$ adh\'erence sch\'ematique de la diagonale. On d\'egage ici une condition pour que $R$ d\'efinisse une relation d'\'equivalence effective sur $T$.
\vspace{1cm}

 Nous suivrons les notations et les conventions d'indices adopt\'ees par {\sc Gabriel} dans \cite[V, \S\S1 \`a 3]{SGA3}. 
En particulier, les morphismes de projection entre produits,  $p_{k} :X^n \rrr X^{n-1}$ sont index\'es de $0$ \`a $n-1$, l'indice $k$ d\'esignant la composante \emph{omise} ; ainsi, $p_{0}(x_{0}, x_{1}, x_{2}) = (x_{1}, x_{2})$.

\enlargethispage*{20pt}

\subsection{Vocabulaire}

\begin{para}\phantomsection Un sch\'ema de base $S$ est fix\'e, et sera omis en indice dans les produits de $S$-sch\'emas.
Soit $T_{0}$ un $S$-sch\'ema, et soit $d : T_{1} \rrr T_{0}^2\, ( = T_{0}\times_{S}T_{0})$ une immersion. Rappelons comment traduire en termes de diagrammes le fait que cette donn\'ee d\'efinisse une relation d'\'equivalence sur $T_{0}$. 

\n La reflexivit\'e \'equivaut \`a l'existence d'un morphisme $\varepsilon : T_{0} \rrr T_{1}$ tel que la diagonale se factorise en $T_{0} \stackrel{\varepsilon}{\rrr} T_{1} \stackrel{d}{\rrr} T_{0}^2$, et la sym\'etrie se traduit par l'existence d'un automorphisme $\sigma$ de $T_{1}$ compatible avec la permutation des facteurs de $T_{0}^2$. 

La transitivit\'e requiert un peu d'attention.

\n Soient 
$$\xymatrix{T_{1} \ar@<0.5ex>[r]^{d_{1}} \ar@<-0.5ex>[r]_{d_{0}}& T_{0}}
$$ 
les morphismes obtenus en composant $d$ et les projections $p_{0}, p_{1} :\xymatrix{T_{0}^2 \ar@<0.5ex>[r] \ar@<-0.5ex>[r]& T_{0}}$.

 \n On introduit le produit fibr\'e $T_{2} = (T_{1}, d_{0})\times_{T_{0}}(T_{1}, d_{1})$, d'o\`u  un carr\'e cart\'esien
 $$
\xymatrix{T_{2} \ar[d]_{d'_{2}} \ar[r]^{d'_{0}} &T_{1} \ar[d]^{d_{1}}\\
T_{1}\ar[r]_{d_{0}} & T_{0}
}
$$
Pour clarifier $T_{2}$ on peut ins\'erer ce carr\'e dans le diagramme suivant
\begin{equation}\label{eqB.1}
\xymatrix{T_{2} \ar[rr]^{d'_{0}} \ar[dr]^{d'} \ar[d]_{u} && T_{1} \ar[d]^{d}  \ar@/^2pc/[dd]^{d_{1}}\\
V \ar[d] \ar[r]_{v} &T_{0}^3 \ar[r]^{p_{0}} \ar[d]_{p_{2}} & T_{0}^2 \ar[d]^{p_{1}}\\
T_{1} \ar[r]_{d} \ar@/_2pc/[rr]_{d_{0}}& T_{0}^2 \ar[r]_{p_{0}} & T_{0}
}
\end{equation}
o\`u on a not\'e $V = (T_{1}, d_{0})\times_{T_{0}}(T^2_{0}, p_{1})$, de sorte que tous les carr\'es sont cart\'esiens, et o\`u $d' = vu$. Puisque $u$ et $v$ proviennent de $d$ par changement de base, ce sont des immersions, et il en est de m\^eme de $d'$ ; cela montre aussi que si $d$ est une immersion ferm\'ee, alors $d'$ est une immersion ferm\'ee.

Si on interpr\`ete, via $d$, $T_{1}$ comme l'ensemble des couples $(x, y)$ d'\'el\'ements  de $T_{0}$ ``en relation'', (notation : $x \sim y$), alors le morphisme $d' : T_{2} \rrr T_{0}^3$  permet de voir (de concevoir) un \'el\'ement de $T_{2}$  comme un triplet $(x, y, z)$ tel que $x \sim y$  et $y\sim z$, les projections \'etant $ d'_{0}(x, y, z) = (y, z)$ et $d'_{2}(x, y, z) = (x, y)$. La transitivit\'e exige que l'on ait alors $x\sim z$ ; comme $(x, z) = p_{1}(x, y, z)$, on peut \'enoncer, plus correctement,
\medskip

\n {\it Pour qu'un couple de morphismes $d_{0}, d_{1} : \xymatrix{T_{1} \ar@<0.5ex>[r] \ar@<-0.5ex>[r]& T_{0}}$ d\'efinisse un relation transitive sur $T_{0}$ il faut que, en posant comme plus haut, $T_{2} = (T_{1}, d_{0})\times_{T_{0}}(T_{1}, d_{1})$,  il existe un morphisme $d'_{1} : T_{2} \rrr T_{1}$ tel que $dd'_{1} = p_{1}d'$ (un tel morphisme est alors unique puisque $d$ est une immersion)}
\medskip

\n Autrement dit on doit avoir un morphisme de $S$-sch\'emas semi-simpliciaux (tronqu\'es)
\begin{equation}\label{eqB.2}
\xymatrix{T_{2} \ar@<1ex>[rr]^{d'_{0},d'_{1},d'_{2}}  \ar[rr]  \ar@<-1ex>[rr]  \ar[d]_{d'}&& T_{1} \ar@<0.5ex>[rr]^{d_{0},d_{1}} \ar@<-0.5ex>[rr]  \ar[d]^{d} && T_{0} \ar@{=}[d]\\
T_{0}^3  \ar@<1ex>[rr]  \ar[rr]  \ar@<-1ex>[rr]_{p_{0},p_{1},p_{2}} &&T_{0}^2  \ar@<0,5ex>[rr]  \ar@<-0.5ex>[rr]_{p_{0}, p_{1}} && T_{0} }
\end{equation}
\end{para}
 
\subsection{\'Equivalence et platitude}
 
\begin{prop}\phantomsection \label{pB.1} Soit $S$ un sch\'ema. Soit $T_{\star}$ une relation d'\'equivalence, au sens d\'ecrit plus haut, telle que le morphisme $d : T_{1} \rrr T_{0}^2$ soit une immersion quasi-compacte.
 Consid\'erons la factorisation de $d$ par adh\'erence sch\'ematique, soit
 $$
 T_{1} \stackrel{i}{\rrr} \wt{T_{1}} \stackrel{\wt{d}}{\rrr} T^2_{0}
 $$
 Alors, si les  morphismes  $\xymatrix{\wt{d_{1}}, \wt{d_{0}} : \wt{T_{1}} \ar@<0.5ex>[r] \ar@<-0.5ex>[r]& T_{0}}$, compos\'es de $\wt{d}$ et des projections sur $T_{0}$,
sont plats, ils proviennent d'une relation d'\'equivalence sur $T_{0}$.
\end{prop}

\begin{proof} Notons d'abord que l'hypoth\`ese implique que les morphis\-mes $d_{0}, d_{1} :\xymatrix{T_{1} \ar@<0.5ex>[r] \ar@<-0.5ex>[r]& T_{0}}$ sont plats puisque $i$ est une immersion ouverte (lemme \ref{lA.8}), et que $d_{j} = \wt{d_{j}}\circ i$. Consid\'erons le morphisme \eqref{eqB.2} de sch\'emas semi-simpliciaux (tronqu\'es) et les factorisations par adh\'erence sch\'ematique des morphisme verticaux $d$ et $d'$. En appliquant le lemme \ref{lA.10}, on obtient un sch\'ema semi-simplicial $\wt{T}_{\star}$ et des morphismes 
 $$
 T_{\star} \rrr \wt{T}_{\star} \rrr T_{0}^{\star +1}
 $$
 En particulier, le morphisme $d'_{1} : T_{2}\rrr T_{1}$ induit un morphisme analogue $\wt{d'_{1}}: \wt{T_{2}} \rrr \wt{T_{1}}$ ; le m\^eme lemme donne aussi imm\'ediatement les morphismes de r\'eflexivit\'e et de sym\'etrie pour $\wt{T_{1}}$. Le seul point non formel est que l'adh\'erence sch\'ematique $\wt{T_{2}}$   de $T_{2}$ dans $T_{0}^3$ soit le produit fibr\'e attendu $(\wt{T_{1}}, \wt{d_{0}})\times_{T_{0}}(\wt{T_{1}}, \wt{d_{1}})$ ; c'est pour cela qu'il faut supposer que les morphismes $\wt{d_{1}}$ et $ \wt{d_{0}}$ sont plats.

 Pour le montrer, notons d'abord que, dans la factorisation de $d$ par adh\'erence sch\'ematique $T_{1} \stackrel{i}{\rrr} \wt{T_{1}} \stackrel{\wt{d}}{\rrr} T^2_{0}$, 
 le morphisme $i$ est une immersion ouverte quasi-compacte (lemme \ref{lA.8}), et \scd e. Consid\'erons le diagramme
 $$
\xymatrix{T_{2} \ar[rr]^{d'_{0}} \ar[dr]^{w} \ar[d]_{u} && T_{1} \ar[d]^{i}  \ar@/^2pc/[dd]^{d_{1}}\\
V \ar[d] \ar[r]_{v} & U \ar[r]^{\wt{d'_{0}}} \ar[d]_{\wt{d'_{2}}} & \wt{T_{1}} \ar[d]^{\wt{d_{1}}}\\
T_{1} \ar[r]_{i} \ar@/_2pc/[rr]_{d_{0}}& \wt{T_{1}} \ar[r]_{\wt{d_{0}}} & T_{0}}
$$
 o\`u on a not\'e $U = (\wt{T_{1}}, \wt{d_{0}})\times_{T_{0}}(\wt{T_{1}}, \wt{d_{1}})$, et $V = (T_{1}, d_{0})\times_{T_{0}}(\wt{T_{1}}, \wt{d_{1}})$. On veut montrer que le produit fibr\'e $U$ est isomorphe \`a l'adh\'erence sch\'ematique $\wt{T_{2}}$ ; il est d'abord clair (surtout si on a le diagramme  \eqref{eqB.1} pr\'esent \`a l'esprit) que le morphisme $U \rrr T_{0}^3$ est une immersion ferm\'ee ; il reste donc \`a voir  que $w = vu$ est \scd . Or, tous les carr\'es du diagramme sont cart\'esiens  ; comme $d_{0}$ est plat, il en est de m\^eme du morphisme horizontal m\'edian $\wt{d'_{0}}v: V \rrr \wt{T_{1}}$, donc $u$ est \scd \, tout comme $i$  (lemme \ref{lA.7}); d'autre part, dans le carr\'e en bas \`a gauche, $i$ est quasi-compact et \scd\,, et $\wt{d'_{2}}$ est plat puisque $\wt{d_{1}}$ l'est par hypoth\`ese, donc $v$ est \scd.
 \end{proof}
  
 Le même raisonnement montrerait que si $\wt{d_{0}}$  et $\wt{d_{1}}$ sont universellement ouverts, alors ils proviennent d'une relation d'équivalence.
 
 \renewcommand{\refname}{Bibliographie}

\end{document}